\documentclass[10pt,a4paper]{article}
\usepackage[utf8]{inputenc}
\usepackage[english]{babel}
\usepackage{multicol,graphicx,color}
\usepackage{pslatex}
\usepackage{authblk}
\usepackage{amsthm}
\usepackage{cite}
\usepackage{amsmath}
\usepackage{amssymb}
\usepackage{latexsym}
\usepackage{lscape}
\usepackage{epsfig}
\usepackage{amsfonts}
\usepackage{mathrsfs,bbm}
\usepackage[
hmarginratio={1:1},     
vmarginratio={1:1},     
textwidth=15cm,        
textheight=21cm,
heightrounded,          
]{geometry}

\usepackage{tikz}

\usepackage[
linktocpage=true,colorlinks,citecolor=magenta,linkcolor=blue,urlcolor=magenta]{hyperref}
\makeatletter

\usepackage[showonlyrefs]{mathtools}  
\mathtoolsset{showonlyrefs=true}

\newtheorem{theorem}{Theorem}[section]
\newtheorem{corollary}[theorem]{Corollary}
\newtheorem{definition}[theorem]{Definition}

\newtheorem{lemma}[theorem]{Lemma}
\newtheorem{proposition}[theorem]{Proposition}
\newtheorem{remark}[theorem]{Remark}

\theoremstyle{definition} \theoremstyle{remark}
\numberwithin{equation}{section}

\newcommand{\R}{\mathbb{R}}
\newcommand{\N}{\mathbb{N}}
\newcommand{\B}{\mathcal{B}}

\newcommand{\p}{\mathcal{P}}

\newcommand{\la}{\lambda}

\newcommand{\ep}{\epsilon}

\newcommand{\HR}{{\mathcal{H}}}

\newcommand{\PR}{{\mathcal{P}}}

\newcommand{\sbr}[1]{\left(#1\right)}
\newcommand{\mbr}[1]{\left[#1\right]}
\newcommand{\lbr}[1]{\left\{#1\right\}}

\newcommand{\abs}[1]{\left\lvert#1\right\rvert}
\newcommand{\nm}[1]{\Vert #1 \Vert}

\begin{document}

	\title{Multiple sign-changing and semi-nodal normalized solutions for a  Gross-Pitaevskii type system on bounded domains: \\ the $L^2$-supercritical case
	}

	\author{Tianhao Liu$^{\mathrm{a}}$ \footnote{liuthmath@gmail.com}}
	
	\author{Linjie Song$^{\mathrm{b,c}}$ \footnote{songlinjie18@mails.ucas.edu.cn}} 
	
	\author{Qiaoran Wu$^{\mathrm{b}}$ \footnote{wuqr24@mails.tsinghua.edu.cn}}

	\author{Wenming Zou$^{\mathrm{b}}$\footnote{zou-wm@mail.tsinghua.edu.cn}}
	
	\affil{{{\small $^{\mathrm{a}}$ School of Mathematics, Statistics and Mechanics, Beijing University of Technology, Beijing 100124, China }}}
	
	\affil{{\small $^{\mathrm{b}}$ Department of Mathematical Sciences, Tsinghua University, Beijing 100084, China}}
	
	\affil{{\small $^{\mathrm{c}}$ Universit\'e Marie et Louis Pasteur, CNRS, LmB (UMR 6623), Besan\c{c}on F-25000, France}}

	\date{}	
	\maketitle
	
	\begin{abstract}
		This paper is concerned with the existence and multiplicity of    sign-changing and semi-nodal normalized solutions for the nonlinear  coupled  elliptic system
		\begin{equation}
			\left\{
			\begin{aligned}
				&-\Delta u_j + \la_j u_j =    \mu_j u_j^3+ \sum_{k=1,k\neq j }^m\beta_{kj} u_k^2 u_j,  \quad u_j \in H_0^1(\Omega), \\
				&\int_\Omega u_j^2dx = c_j, \quad j = 1,2,\cdots,m.
			\end{aligned}	
			\right.
		\end{equation}
		This system    arises in the search for standing waves of the Gross–Pitaevskii equations. Here, $\Omega \subset \R^N$   is a bounded domain, and we focus on the dimensions $N = 3, 4$. Note that the case $N=4$ corresponds to a Sobolev critical setting. The constants    $ \mu_j \neq 0, \beta_{kj} \neq 0$ and $c_j > 0$ are prescribed,  while $\la_1, \cdots, \la_m$ are unknown and appear as Lagrange multipliers. 	We present the first result in the literature on the existence and multiplicity of such sign-changing and semi-nodal normalized solutions  across all regimes of the  parameters $\beta_{kj}$.  Our main tool is a new skill of vector linking; this article is the first attempt to apply linking techniques to a coupled system of this type. To obtain semi-nodal solutions, we introduce a  partial vector linking approach, which is new to the best of our knowledge. 
		Finally, we obtain  some bifurcation results by investigating the asymptotic behavior   as $\vec{c} = (c_1, \ldots, c_m) \to \vec{0}$.

		\vskip0.1in
		{\small \noindent \text{\bf Key Words:} Gross-Pitaevskii equations; $L^2$-supercritical; sign-changing and semi-nodal normalized solutions; vector linking method; bifurcation
			\vskip0.1in
			\noindent\text{\bf Mathematics Subject Classification:} 35J50, 35J15, 35J60}
		
	\end{abstract}

	\newpage
	\section{Introduction and main results}
	The coupled nonlinear Schr\"odinger system,   widely known as the Gross–Pitaevskii  equations, has attracted considerable interest from the mathematical community, particularly since the seminal work of Lin and Wei \cite{Lin-Wei=CMP=2005}. This system takes the form
	\begin{align} \label{GP system}
		\left\{
		\begin{aligned}
			-	i \frac{\partial}{\partial t} \Phi_j =\Delta \Phi_j  + \mu_j \abs{\Phi_j}^{2} \Phi_j +\sum\limits_{k=1,k\neq j }^{m}\beta_{kj}\abs{\Phi_k}^{2} \Phi_j   &\quad \text{ for } ~x\in \Omega, ~~t>0,  \\
			\Phi_j(x,t) = 0, \quad j= 1,\cdots,m,\quad \quad\quad\quad\quad\quad\quad\quad\quad & \quad\text{ for  }  ~ x\in \partial \Omega, ~~ t>0 ,
		\end{aligned} 
		\right.
	\end{align}
	where $\Phi_j = \Phi_j(x,t) \in \mathbb{C}$ are complex-valued wave functions.
	It has applications in many physical models,    especially in the Hartree–Fock theory for Bose–Einstein condensates with multiple states  (see, e.g., \cite{BG1,BG3,BG4,BG5,BG6,ebgcbjbj,BEC-1998}) and  nonlinear optics (see, e.g.,  \cite{anaa,BG7,nonlinearoptics}).  Physically, the wave function $\Phi_j$ represents the condensate amplitude of the $j$-th component in Bose–Einstein condensates, while in the context of nonlinear optics, it denotes the $j$-th component of an optical beam in Kerr-like photorefractive media. Furthermore, problem \eqref{GP system} also appears as a variational model in population dynamics, we refer to  \cite{CC2003,CTV2002,CTV2003} for more details.

	\medbreak
	
	System \eqref{GP system} is Hamiltonian and thus conserves the masses, defined by
	\begin{equation}
		M_j(t) :=	\int_{\Omega} |\Phi_j(x,t)|^2 dx, \quad j= 1,\cdots,m.
	\end{equation}
	In the standard quantum mechanical interpretation for a single particle, this conserved quantity represents the total probability of finding the particle of the $j$-th species. In Bose–Einstein condensates,   $M_j(t)$ represents the number of particles of the $j$-th component,  while  in nonlinear optics, it corresponds to the power supply.	
	\medbreak
	A fundamental class of solutions to \eqref{GP system} is that of standing waves, which take the form $$\Phi_j(x,t)=e^{i\la_j t}u_j(x)$$ with $\lambda_j \in \mathbb{R}$ and real-valued functions $u_j \in H_0^1(\Omega)$. In view of the mass conservation, it is physically natural to study standing waves with prescribed $L^2$-norms. This leads to the following steady-state nonlinear Schr\"odinger system with mass constraints
	\begin{align} \label{eq:mainequation}
		\left\{
		\begin{aligned}
			&-\Delta u_j + \la_j u_j =  \mu_j u_j^3+ \sum_{k=1,k\neq j }^m\beta_{kj} u_k^2 u_j , \quad  u_j \in H_0^1(\Omega),  \\
			& 	\int_\Omega u_{j}^2dx=c_{j}, \quad j = 1,2,\cdots,m,
		\end{aligned}
		\right.
	\end{align}
	which is known in the literature as a \textit{normalized problem}. Here, $\Omega\subset\R^N $ is a bounded regular domain, and
	$c_{1},\cdots,c_{m}$ are given positive constants. Throughout this paper, we assume that 
	\begin{equation} \label{assumption}
		\mu_j\neq0 \quad \text{ and }  \quad \beta_{kj}=\beta_{jk}\neq 0 , ~~ \forall ~ j, k\in \lbr{1,\cdots,m},~~ j\neq k. 
	\end{equation}
	Physically, the parameters $\mu_j$ and $\beta_{kj}$ represent the intraspecies and interspecies scattering lengths, respectively, while  $\la_j $ arise from the chemical potentials.  In particular, $\mu_j > 0$ corresponds to the focusing case, and $\mu_j < 0$ corresponds to the defocusing
	case.  The sign of   $\beta_{kj} $   describes the interaction between	particles of two different condensates:   $\beta_{kj}>0$ means attractive  (or cooperative) interaction, while  $\beta_{kj}<0$ means repulsive  (or competitive) interaction.
	
	\medbreak
	The relation $\beta_{kj} = \beta_{jk}$ reflects the symmetry of interactions between different components and provides a variational structure
	to the problem. Consequently, the constrained variational method becomes a natural and effective approach. In this framework, the parameters $\lambda_j$ in \eqref{eq:mainequation} arise as unknown Lagrange multipliers. From a variational viewpoint, normalized solutions to \eqref{eq:mainequation} can be obtained as critical points of the energy functional $E: H_{0}^1(\Omega,\mathbb{R}^m) \to \R$ defined by
	\begin{align} \label{eqofE}
		E(\vec{u}) = \frac12\sum_{i=1}^m\int_\Omega |\nabla u_i|^2dx - \frac14\sum_{i=1}^m \mu_i\int_\Omega  u_i^4dx  -   \frac14\sum_{i,j=1, j \neq i}^m\beta_{ij}\int_\Omega u_i^2u_j^2 dx
	\end{align}
	on the constraint set   
	\begin{equation}\label{sets}
		S_{\vec{c}}:=  \bigg\{ \vec{u} = (u_1,\cdots,u_m) \in H_{0}^1(\Omega,\mathbb{R}^m):~ \int_\Omega u_j^2dx = c_j, ~~ j \in \lbr{ 1,\cdots,m }\bigg\},
	\end{equation}
	where $\vec{c}=\sbr{ c_1,\ldots,c_m} \in \sbr{\R^+}^m$.
	\medbreak
	Before proceeding, we   introduce some important concepts.
	\begin{definition}
		{\rm	Let $ \vec{u} := (u_1,\cdots,u_m)$ be a solution of \eqref{eq:mainequation}.  
			\begin{itemize}
				\item[$\bullet$]  A solution $\vec{u} $ is called {\it trivial} if all its components are vanishing; it is called {\it semi-trivial} if some components are vanishing and the other ones are nontrivial;  it is called {\it nontrivial} if all of its components are nontrivial, that is, $u_i \not\equiv 0$ for every $i \in  \lbr{1, \cdots, m}$.
				\item[$\bullet$] A solution $ \vec{u}$ is called a {\it positive} solution, if $u_i > 0$   for every $i\in  \lbr{1, \cdots, m}$. A solution $\vec{u} $ is  called a  {\it sign-changing} solution, if  $u_i$ changes sign for every $i\in  \lbr{1, \cdots, m}$. A solution $\vec{u} $ is called a  {\it semi-nodal} solution, if some components change sign and the other ones are positive.	
		\end{itemize}	  }
	\end{definition}
	
	In the case where $\Omega = \R^N$, the existence of positive normalized solutions for system \eqref{eq:mainequation} in the focusing setting has been extensively studied over the past decade.  All known results concern the case of $m = 2$ equations; see \cite{btjl,btjlsn,btlhzw,btsn,btsn3,btsn2,lhzw,llz} and the references therein. The main ingredient  is the application of the Pohozaev identity combined with the scaling transformation
	$s\star u(x)= e^{sN/2}u(e^sx)$ with $s\in \R$  which was initially introduced in \cite{Jeanjean97}.     Further developments  concerning the whole space case can be found in \cite{btzxzw,Borthwick-C-J-S,Jeanjean+lu,Jeanjean-Trung,Jeanjean-Zhang-Zhong} and the references therein.
	\medbreak
	In contrast, the situation on {\bf bounded domains} is markedly different due to the absence of dilation invariance, and the literature on  normalized solutions remains relatively limited.   While there has been some recent progress, the focus has  primarily been on   positive  solutions, see  \cite{DDGS,DST,PV=CVPDE,PVY=SIAM25,NTV=APDE,BQZ=MA2024,Song,slzw1,PPV=JDE} for single equation case, and \cite{nbthvg,slzw2} for systems of two equations.  In particular, the work \cite{nbthvg} also established the existence of a positive normalized minimizer for Br\'ezis–Nirenberg-type critical equations. Moreover, the existence of a second positive normalized solution on bounded domains was independently obtained in \cite{slzw1} and \cite{PVY=SIAM25}. The article \cite{NTV=APDE} addressed the existence and orbital stability of positive ground states with prescribed mass for the $L^2$-critical and supercritical scalar equations on bounded domains; related developments for scalar equations can be found in \cite{PV=CVPDE, Song, PPV=JDE}. The work \cite{BQZ=MA2024} proved the existence of positive normalized solutions to  Schr\"odinger equations on large bounded domains, thereby resolving an open problem posed in \cite{BMRV=CPDE2021}. Recently, under an $L^2$-supercritical setting, the existence of a second positive normalized solution was established in \cite{slzw2}.
	
	\vskip0.1in 
	However, to the best of our knowledge, there are no results on sign-changing normalized solutions for multi-coupled systems   \eqref{eq:mainequation} on bounded domains. Indeed, even for a single equation, this problem remains challenging.  This is primarily because   positive solutions can often be obtained using methods such as sub-solution techniques, maximum principles, or symmetry (e.g., radial symmetry), and they are frequently associated with energy minimization. In contrast, sign-changing solutions possess a more complex structure due to the presence of nodal points. This complexity not only invalidates many standard tools but also means that such solutions typically correspond to critical points at higher energy levels.
	\vskip0.1in

	To address these challenges, we develop in this paper a novel framework that bridges this gap. In particular, under the   assumption \eqref{assumption},
	we establish the existence and multiplicity of both sign-changing and semi-nodal normalized solutions for \eqref{eq:mainequation} on a bounded regular domain $\Omega\subset \R^N$ (with $N=3,4$).

	\medbreak
	Our main existence results are stated as follows.

	\begin{theorem}\label{thm1.1}
		Assume that $N=3,4,$ and that \eqref{assumption} holds. For any given  positive integer $j$, there exists $\tilde c_j > 0$ such that for each $\vec{c} = (c_1,\cdots,c_m) \in (0,\infty)^m$ with $$\sbr{\max\limits_{1\leq i\leq m} c_i^2} \sbr{\min\limits_{1\leq i\leq m} c_i}^{-1} < \tilde c_j,$$
		then 	the problem \eqref{eq:mainequation}  admits at least $j$ sign-changing normalized solutions.
	\end{theorem}

	\begin{theorem}\label{thm1.2}
		Assume that $N=3,4,$ and that \eqref{assumption} holds.  For an integer $d$ with $1 \le d \le m-1$ and any positive integer $j$, there exists $\tilde c_{j,d} > 0$ such that for each $\vec{c} = (c_1,\cdots,c_m) \in (0,\infty)^m$ with $$\sbr{\max\limits_{1\leq i\leq m} c_i^2} \sbr{\min\limits_{1\leq i\leq m} c_i}^{-1} < \tilde c_{j,d},$$
		then	the problem \eqref{eq:mainequation} admits at least $j$  semi-nodal  normalized solutions  with sign-changing components $u_1, \cdots, u_d$ and positive components $u_{d+1}, \cdots, u_m$.		
	\end{theorem}
	
	\begin{remark}
		{\rm 	We emphasize that  the values $\tilde c_j$ and $\tilde c_{j,d}$ can be explicitly estimated. Moreover, we prove the existence of an arbitrarily large number of sign-changing and semi-nodal normalized solutions,  thereby directly filling  the gap in the study of multi-coupled systems on bounded domains highlighted in the introduction. }
	\end{remark}
	\begin{remark}
		{\rm  Assumption \eqref{assumption} covers all regimes for the coefficients \(\mu_j\) and \(\beta_{kj}\). In particular, this assumption encompasses not only the purely focusing case (\(\mu_j>0\) for all \(j\)), the purely defocusing case (\(\mu_j<0\) for all \(j\)), and mixed cases, but also the purely attractive case (\(\beta_{kj}>0\) for all \(k\neq j\)), the purely repulsive case (\(\beta_{kj}<0\) for all \(k\neq j\)), and mixed attractive-repulsive cases. Although we do not explicitly consider the case where some coefficients \(\mu_j\) or \(\beta_{kj}\) vanish, our results still hold, provided that we interpret expressions containing these coefficients in the denominator as equal to \(+\infty\) when the denominator is zero; see, for instance, \eqref{defi of rho}.}
	\end{remark}

	\begin{remark}
		{\rm  
			Theorems \ref{thm1.1} and \ref{thm1.2} deal with  the dimensions $N=3$ and $4$, both of which are $L^2$-supercritical since $4 > p_c := 2 + 4/N$. In this regime, the energy functional  $E$ is  unbounded from below on $H_0^1(\Omega,\R^m)$. Consequently, additional ideas are required to obtain the bounded   Palais-Smale sequences.  Furthermore,  the case $N=4$ is Sobolev-critical and presents additional challenges compared to $N=3$, due to the non-compactness of the embedding $H_0^1(\Omega) \hookrightarrow L^4(\Omega)$. As a result, further technical estimates are required.   
		}
	\end{remark}

	\medbreak

	The main tool of this paper is a new skill of \emph{vector linking} introduced in Section \ref{seclink}. In particular, to obtain semi-nodal  normalized solutions, we introduce \emph{partial vector linking} which is innovative up to our knowledge. It seems that the vector linking method has never been used to search for solutions of a coupled system, dealing with both the fixed mass problems and  the fixed frequency problems. This paper is the first attempt in this direction in the literature.
	
	\medbreak
	In our proof, we also generalize the well-known Br\'ezis-Martin result, which is stated for a Banach space, to a \emph{product manifold} setting where the mass constrained problem for Schr\"odinger systems is included. See Section \ref{Sect2}. We believe that this generalization is of independent interest. It constitutes a general framework to prove that some subset is invariant with respect to a flow on a product manifold, which extends the arguments in \cite[Section 4]{JS}. Then in Section \ref{secdescendingflow} we construct a descending flow with the gradient field or with a just right pseudo gradient vector field. After checking the condition used in the general framework, we successfully discover  some sets which are invariant under the descending flow. Here we point out that there is a main difference in checking this condition with the unconstrained case where $\la_1,\cdots,\la_m$ are independent of $\vec{u}$; see details in Lemma \ref{lemG(u)beta<0}.
	


	\begin{remark} \label{remark 1.10}
		{\rm It is worth noticing that our main results can be easily generalized. In fact, let  $f, g : \R \to \R$  be two continuous functions such that $f(t), g(t) \ge 0$ for   $t > 0$, $f(t), g(t) \le 0$ for   $t < 0$, and
			\begin{align} \label{eqassumefg}
				|f(t)| \le C(|t|^{p-1} + |t|^{\frac{2^*}{2}-1}), \quad |g(t)| \le C(|t|^{q-1} + |t|^{2^*-1}), \quad \forall t \in \R,
			\end{align}
			for some $C > 0$, $1 < p \le 2^*/2$, $2 < q \le 2^*$, where $2^* = 2N/(N-2)$ is the Sobolev critical exponent.  Our method can be applied to the following system
			\begin{align*}
				\left\{
				\begin{aligned}
					& -\Delta u_i + \la_i u_i = \mu_i g(u_i) + \sum_{j\neq i}\beta_{ij} f(u_i) F(u_j) \quad \text{ in } \Omega , \\
					& u_1=\cdots=u_m=0   \quad \text{ on } \partial \Omega,
				\end{aligned}
				\right.
			\end{align*}
			where $\Omega \subset \R^N$ ($N \ge 3$) and $\mu_i \neq 0$, $\beta_{ij} = \beta_{ji}  \neq 0$, $\forall i \neq j$, and $F(t) = \int_0^t f(s)ds$.  We emphasize that our method is applicable to \emph{non-even functionals} and so we assume nothing on the symmetry of $f$ and $g$.}
	\end{remark}

	

	\medbreak

	Finally we are concerned with the bifurcation phenomenon of \eqref{eq:mainequation} by investigating limit properties of the normalized solutions and Lagrange multipliers obtained in Theorems \ref{thm1.1} and \ref{thm1.2} as $\vec{c}\to\vec{0}^+$. We now introduce the definition of bifurcation points for system \eqref{eq:mainequation}. Owing to the presence of semi-trivial solutions, bifurcation points will be  categorized as either nontrivial or semi-trivial.

	\begin{definition}[Bifurcation points] \label{bifurcation}
		{\rm  We say $\vec{\la}^* = (\la_1^*,\cdots,\la_m^*) \in \R^m$ is a bifurcation point of system \eqref{eq:mainequation} (without prescribed $L^2$-norms) if and only if there exists $\{\la_{n,1} ,\cdots, \la_{n,m}; \vec{u}_n\} \subset \R^m \times H_0^1(\Omega,\R^m)$ such that $\vec{u}_n \not\equiv 0$ solves \eqref{eq:mainequation} with $\la_j = \la_{n,j}$, $-\la_{n,j} \to \la_j^*$, $j = 1, \cdots, m$, but $\vec{u}_n \to 0$ strongly in $H_0^1(\Omega,\R^m)$ as $n \to \infty$. Moreover, $\vec{\la}^*$ is called a nontrivial bifurcation point if $\vec{u}_n$ is nontrivial for each $n$ while is called a semi-trivial bifurcation point if $\vec{u}_n$ is semi-trivial for every $n$. Given a positive integer $1 \le d \le m-1$, we call $\vec{\la}^*$ a $d$-semi-trivial bifurcation point if $\vec{u}_n$ has exactly $d$ nontrivial components.}
	\end{definition}

	\begin{theorem}\label{limitbehaviorhsignchanging}
		Assume that $N=3,4,$ and that \eqref{assumption} holds. Then the set of nontrivial bifurcation points of system \eqref{eq:mainequation} is exactly
		\[
		\bigg\{ (\lambda_{1},\cdots,\lambda_{m}) \in \R^m: \la_i \text{ is an eigenvalue of $-\Delta$ on $H_0^1(\Omega)$ for each } i \in \big\{1,\cdots,m\big\} \bigg\}.
		\]
	\end{theorem}
	
	\begin{theorem}\label{thm:semi-trivialbifurpoint}
		Assume that $N=3,4,$ and that \eqref{assumption} holds. Then the set of semi-trivial bifurcation points of system \eqref{eq:mainequation} is exactly
		\[
		\bigg\{ (\lambda_{1},\cdots,\lambda_{m}) \in \R^m: \exists\ i \in \big\{1,\cdots,m\big\}, ~ \la_i \text{ is an eigenvalue of $-\Delta$ on } H_0^1(\Omega)\bigg\}.
		\]
		Moreover, for any given  positive integer $1 \le d \le m-1$, the set of $d$-semi-trivial bifurcation points of system \eqref{eq:mainequation} is exactly
		\[
		\bigg\{ (\lambda_{1},\cdots,\lambda_{m}) \in \R^m: \exists\ i_1 < \cdots < i_d \in \big\{1,\cdots,m\big\}, ~ \la_{i_1}, \cdots, \la_{i_d} \text{ are eigenvalues of $-\Delta$ on } H_0^1(\Omega)\bigg\}.
		\]
	\end{theorem}
	
	\begin{remark}
		{\rm	We remark that a bifurcation point may be both a nontrivial one and a semi-trivial one. As for the specific case of system \eqref{eq:mainequation}, from Theorem \ref{limitbehaviorhsignchanging} and Theorem \ref{thm:semi-trivialbifurpoint} it can be seen that any nontrivial bifurcation point is also a semi-trivial one; but there exist semi-trivial bifurcation points which are not nontrivial ones.}
	\end{remark}

	For topics involving the bifurcation phenomenon of system \eqref{eq:mainequation}, we refer to \cite{Bartsch,BDW,BTW,btzxzw,CZ2013,WZZ} but $\beta=\beta_{ij}=\beta_{ji}$ is taken as the bifurcation parameter in results of all these references. As far as the authors know, there is no bifurcation result in the literature of system \eqref{eq:mainequation} stemming from trivial zero solution with $\vec{\la}$ being the bifurcation parameter; on the contrary, this direction was widely studied for a single equation, see e.g. \cite{CFS,SZ,Stuart1,Stuart2,Stuart3}. We also point out that the classical Crandall-Rabinowitz theory (see \cite{CR1,CR2,Rab}) and abstract results on potential operators (see e.g. \cite{BC,FR,Kie,Rab2}), working for the single equation case, are not applicable to the system \eqref{eq:mainequation} when we take $\vec{\la}$ as the bifurcation parameter. Both nontrivial solutions and semi-trivial solutions may stem from zero solution, which makes the bifurcation phenomenon of system \eqref{eq:mainequation} more complicated and difficult to study. The main idea in this paper to get bifurcation points is to study the limit behaviors of normalized solutions with respect to the mass tending to zero. This idea has been used in \cite{SZ} for a single equation. Precisely in \cite{SZ}, the limit behavior of energy of normalized solutions was shown and using this information the authors got the limit value of Lagrange multipliers. However, for system \eqref{eq:mainequation}, the Lagrange multiplier $\vec{\la}$ is $m$-dimensional, so the information of limit behavior of energy is not sufficient to obtain the precise limit value of $\vec{\la}$. New ideas are needed to solve this issue. Additionally, as mentioned in Remark \ref{rmk:mdimension} later, the bifurcation phenomenon from the nontrivial bifurcation points given in Theorem \ref{limitbehaviorhsignchanging} for system \eqref{eq:mainequation} is $m$-dimensional. The dimension from semi-trivial bifurcation points given in Theorem \ref{thm:semi-trivialbifurpoint} can also be described.

	\vskip0.15in
	This article is organized as follows. In Section \ref{seclink}, we introduce the notion of vector linking and partial vector linking, the former is used to construct minimax values that yield sign-changing normalized solutions for \eqref{eq:mainequation}, and the latter derives minimax values that produce semi-nodal normalized  solutions for \eqref{eq:mainequation}. In Section \ref{Sect2}, we generalize the well-known Br\'ezis-Martin result which is stated for a Banach space to a product manifold setting where the mass constrained problem for Schr\"odinger systems is included. This provides a general framework to prove some set is invariant under a descending flow on a product manifold. Using the abstract result established in Section \ref{Sect2}, we construct a descending flow with gradient field when all  $\mu_i >0$ and $\beta_{ij} > 0 ~ \forall i \neq j$ and with a good pseudogradient field when $\mu_i < 0$ for some $i$ or $\beta_{ij} < 0$ for some $i\neq j $, and succeed to find some invariant sets under this descending flow in Section \ref{secdescendingflow}. We complete the proof of Theorem \ref{thm1.1} and Theorem \ref{thm1.2} in Section \ref{secsign-changing} and \ref{secsemi-nodal} respectively. Finally, the analysis of limit behaviors is completed in Section \ref{limit}, and we provide the proof of Theorem \ref{limitbehaviorhsignchanging} and Theorem \ref{thm:semi-trivialbifurpoint}.
	
	\vskip0.1in
	We end this introduction by introducing the main functional spaces and some notations.
	\begin{itemize}
		\item We endow the Sobolev space  $ H_0^1(\Omega)$   with scalar products and norms
		\begin{align*}
			\langle u,v \rangle = \int_{\Omega}\nabla u \cdot \nabla v ~dx,\quad  \text{ and } \quad  \ \|u\|: = \|u\|_{H_0^1(\Omega)} =  \langle u,u \rangle^{\frac{1}{2}}.
		\end{align*}
		\item We consider the product space $	\HR := H_0^1(\Omega,\R^m) $, with standard scalar products and norms
		\begin{align*}
			\langle \vec{u},\vec{v} \rangle = \sum_{i=1}^m \langle u_i,v_i \rangle, \quad  \text{ and } \quad 	\|\vec{u}\|=  \sbr{\sum_{i=1}^m  \|u_i\|^2}^{\frac{1}{2}}  .
		\end{align*}
		\item Let $\mathcal{C}_N$ be the Sobolev best constant of $H_0^1(\Omega)\hookrightarrow L^4(\Omega)$,
		\begin{equation} \label{defi of CN}
			\mathcal{C}_N :=\inf_{u \in H_0^1(\Omega) \setminus \{0\} } \cfrac{\left( \int_{\Omega} |\nabla u|^2 dx\right) ^{1/2}}{\left(\int_{\Omega} |u|^4dx\right)^{1/4}}.
		\end{equation}
	\end{itemize}

	\section{Vector linking and the minimax level} \label{seclink}

	\subsection{Vector linking}
	
	We introduce a new notion of (partial) vector linking on product manifolds in an abstract setting. Let $\mathcal{M}_i$  ($i \in \lbr{ 1,\cdots,m}$) stand for some $C^2$-Finsler manifolds, and let us define
	\begin{align*}
		\mathcal{M}^{(m)} := \mathcal{M}_1 \times \mathcal{M}_2 \times \cdots \times \mathcal{M}_m.
	\end{align*}
	
	\begin{definition}[Abstract vector linking] \label{def:link}
		{\rm Let $A$ be a compact subset of $\mathcal{M}^{(m)}$ with $D \subset A$. Then the set $A$ is said to be linked to a subset $S \subset\mathcal{M}^{(m)} $, if $D \cap S = \emptyset$, and for each $h \in C(A,\mathcal{M}^{(m)})$ such that $h|_{D} = \textbf{id}$, there holds $$h(A) \cap S \neq \emptyset.$$}
	\end{definition}
	
	\begin{definition}[Abstract partial vector linking] \label{def:partial link}
		{\rm Let $d$ be an integer with $1 \le d \le m-1$. Define the projection map by
		\begin{align*}
			T_d: \mathcal{M}^{(m)} \to \mathcal{M}_1 \times \cdots \times \mathcal{M}_d, \quad \vec{u} = (u_1,\cdots,u_m) \mapsto (u_1,\cdots,u_d).
		\end{align*}
		Let $A$ be a compact subset of $\mathcal{M}^{(m)}$ with $D \times \lbr{\vec{v}_{m-d}} \subset A$ where $D \subset \mathcal{M}_1 \times \cdots \times \mathcal{M}_d$ and $\vec{v}_{m-d} \in \mathcal{M}_{d+1} \times \cdots \times \mathcal{M}_m$.  Then the set $A$ is said to be partially $d$-linked to a subset $S \subset \mathcal{M}^{(d)}$, if $D \cap S = \emptyset$, and  for each $h \in C(A,\mathcal{M}^{(m)})$ such that $h|_{D \times \lbr{\vec{v}_{m-d}}} = \textbf{id}$, there holds $$T_d \circ h(A) \cap S \neq \emptyset.$$}
	\end{definition}
	
	\begin{remark}
		{\rm $ $There is no fundamental difference between Definition \ref{def:link} and the usual notion of linking in the literature, by seeing $\mathcal{M}^{(m)}$ as a $C^2$-Finsler manifold; while Definition \ref{def:partial link}, the partial vector linking, is essentially new.}
	\end{remark}
	
	In what follows, we will use Definition \ref{def:link} to construct minimax values that yield sign-changing  normalized solutions for \eqref{eq:mainequation}, and Definition \ref{def:partial link} to derive minimax values that produce semi-nodal  normalized solutions for \eqref{eq:mainequation}.
	
	\subsection{The minimax level for sign-changing normalized solutions}\label{subsection 2.2}

	Recall that the eigenvalues of $-\Delta$ on $H_0^1(\Omega)$ satisfy
	\begin{equation}
		0 < \Lambda_1 < \Lambda_2 \leq \Lambda_3 \leq \cdots.
	\end{equation}
	Let $\{\phi_k\}_{k \in \mathbb{N}}$ be the family of the $L^2$-normalized eigenfunctions of $-\Delta$ on $H_{0}^1({\Omega})$ with $\phi_1 > 0$. Then $$H_0^1(\Omega) = \overline{\text{span}\{\phi_k: k \in \mathbb{N}\}}.$$  In our application later, we  are particularly interested in the integer $k$  such that  $\Lambda_{k } < \Lambda_{k+1}$, of which there are infinitely many since $\Lambda_{k } \to \infty$ as $k \to \infty$.
	\medbreak
	
	For each $k\in \N$,  we define the $k$-dimensional subspace $H_k$ of $H_0^1(\Omega)$ as $H_k :=   {\text{span}\{\phi_1, \cdots, \phi_k\}}$ and its orthogonal complement as  $$H_{k }^\bot := \overline{\text{span}\{\phi_i: i \in \mathbb{N}, i \ge k+1\}}.$$  We then introduce the product spaces
	\begin{equation}\label{defi of Hk}
		\HR_{k}:= \underbrace{H_{k }\times H_{k }\times \cdots \times H_{k }}_{m} , \quad \text{ and } \quad  \HR_{k}^\bot:= \underbrace{H_{k }^\bot\times H_{k }^\bot\times \cdots \times H_{k }^\bot}_{m}.
	\end{equation}
	Then,   for all $k \geq 1$, all the components of fully nontrivial elements of $ \HR_{k}^\bot$ are sign-changing.
	Set
	\begin{equation}
		\mathbb{S}_{ k}  :=S _{\vec{c}} \cap \HR_ {k}, \quad \text{ and } \quad \mathbb{S}_{k}^\bot  :=S _{\vec{c}} \cap \HR_{k}^\bot ,
	\end{equation}
	then  one can observe that $ 	\mathbb{S}_{ k}\cong \mathbb{S}_k^{1 } \times \mathbb{S}_k^{2} \times \cdots \times \mathbb{S}_k^{m} $, where
	\begin{align} \label{defi of Sk}
		\mathbb{S}_k^{i} := \bigg\{u_{i} = \sum_{j=1}^{k }t_{i,j}\sqrt{c_i}\phi_j: \sum_{j=1}^{k }t_{i,j}^2 = 1\bigg\}, \quad i\in \lbr{1,\ldots, m}.
	\end{align}
	For each $i\in \lbr{1,\ldots, m}$, we  introduce
	\begin{align} \label{defi of Mk}
		\mathbb{M}_{k+1}^i := \bigg\{u_i = \sum_{j=1}^{k+1}t_{i,j}\sqrt{c_i}\phi_j: \sum_{j=1}^{k+1}t_{i,j}^2 = 1 ~ \text{ with } ~t_{i,k+1} \geq 0\bigg\},	
	\end{align}
	and  let
	$\mathbb{M}_{k+1}  :=  \mathbb{M}_{k+1}^1 \times \cdots \times \mathbb{M}_{k+1}^{m} $. It follows that $ \mathbb{S}_k \subset \mathbb{M}_{k+1}$ since $\partial \mathbb{M}_{k+1}^i=\mathbb{S}_{k}^i$.
	\medbreak
	Let \begin{align*}
		S_{c_i} := \bigg\{ u \in H_0^1(\Omega): \int_\Omega u^2dx = c_i\bigg\}, \quad i\in\{1,\cdots,m\}.
	\end{align*}
	Then $ 	S_{\vec{c}}= S_{c_1} \times S_{c_2} \times \cdots \times S_{c_m}	$.
	\medbreak
	The abstract Definition \ref{def:link} of vector linking  in the  specific case where $\mathcal{M}_i = S_{c_i}$,  $\mathcal{M}^{(m)}= S_{\vec{c}}$, $D = \partial \mathbb{M}_{k+1}$ and $S = \mathbb{S}_{k}^\bot$ is stated as follows.
	
	\begin{definition}[Vector linking] \label{link}
		{\rm A compact subset $A$ of $S_{\vec{c}}$ with $\partial \mathbb{M}_{k+1} \subset A$ is said to be linked to $\mathbb{S}_{k}^\bot $, if, for any continuous mapping $h \in C(A, {S_{\vec{c}}})$ satisfying $h|_{\partial \mathbb{M}_{k+1}} = \textbf{id}$, there holds $$h(A) \cap\mathbb{S}_{k}^\bot  \not= \emptyset.$$}
	\end{definition}

	Next we show how to construct minimax values at which sign-changing normalized solutions can be found. In this part, we will provide detailed arguments to prove that the minimax level we construct is finite and well defined. Further, we will estimate the minimax level and show that the linking sets have well value-separated property.
	\medbreak
	
	To obtain sign-changing solutions, as in many references such as \cite{BLW,Zou}, we should use cones of positive and negative functions. Precisely, we define
	\begin{equation}
		\mathcal{P}_i:=\lbr{\vec{u} = (u_1,u_2,\cdots,u_m)\in \HR: u_i\geq 0}; \quad  \mathcal{P}:=\bigcup_{i=1}^m (\mathcal{P}_i \cup -\PR_i).
	\end{equation}
	Then all the elements outside $\mathcal{P}$ are sign-changing.  For $\delta>0$, we define $\PR_\delta:=\lbr{ \vec{u} \in \HR: \text{dist}(\vec{u}, \PR) <\delta }$, where
	\begin{equation}
		\begin{aligned}
			& \text{dist}(\vec{u}, \PR):=\min \lbr{\text{dist}(\vec{u}, \PR_i), ~\text{dist}(\vec{u},-\PR_i): i = 1,\ldots,m },\\
			&\text{dist}(\vec{u}, \pm\PR_i) := \inf \lbr{ \|\vec{u}-\vec{v}\|: \vec{v}\in\pm\PR_i } .
		\end{aligned}
	\end{equation}
	Note that
	\begin{align*}
		\PR_\delta = \bigcup_{i = 1}^m \sbr{(\PR_i)_\delta \cup (-\PR_i)_\delta}, \quad (\pm\PR_i)_\delta = \lbr{ \vec{u} \in \HR: \text{dist}(\vec{u}, \pm\PR_i) <\delta }.
	\end{align*}
	Now we introduce the set $S^*(\delta)$ by
	\begin{align}
		S^*(\delta) := \mathcal{H} \backslash \overline{\PR_\delta}.
	\end{align}
	
	When $N = 3,4$, the case is $L^2$-supercritical and the boundedness of Palais-Smale sequences becomes a challenging issue. To avoid this difficulty, for $\vec{c} = \sbr{c_1,\ldots,c_m}\in \sbr{\R^+}^m$, we introduce the following bounded subset of $S_{\vec{c}}$
	\begin{align} \label{defBrho}
		\B_\rho := \lbr{\vec{u} = (u_1,u_2,\cdots,u_m) \in S_{\vec{c}} : \sum_{i=1}^m\int_\Omega |\nabla u_i|^2dx < \rho}.
	\end{align}
	In what follows, we define
	\begin{equation}
		I^+:=\lbr{i\in \lbr{1,\ldots, m}: \mu_i>0}, \quad  I^-:=\lbr{i\in \lbr{1,\ldots, m}: \mu_i<0},
	\end{equation}
	and
	\begin{equation}
		J^+:=\lbr{(i,j)\in \lbr{1,\ldots, m}^2 : \beta_{ij}>0,~i\neq j }, \quad J^-:=\lbr{(i,j)\in \lbr{1,\ldots, m}^2 : \beta_{ij}<0, ~i\neq j}.
	\end{equation} 	
	Throughout this paper, we use the notations \begin{equation}
		\mu_{\max}^+ :=\max \lbr{ 0, ~\mu_i : i\in I^+}\geq 0, \quad  \beta_{\max}^+:=\max \lbr{ 0, ~\beta_{ij} :(i,j)\in J^+} \geq 0,
	\end{equation}
	\begin{equation}
		\mu_{\min}^-  :=\min \lbr{ 0, ~\mu_i : i\in I^-} \leq 0, \quad  \beta_{\min}^- :=\min \lbr{ 0, ~\beta_{ij} : (i,j)\in J^-} \leq 0,
	\end{equation}
	Take
	\begin{equation} \label{defi of rho}
		0< \rho < \frac{2\mathcal{C}_N^4}{4 \sbr{ \mu_{\max}^+ + \beta_{\max}^+} -3 \sbr{	\mu_{\min}^-  + \beta_{\min}^-} }.
	\end{equation}	
	Then we have the following result.
	
	\begin{lemma} \label{lem>0}
		For any $\rho > 0$ satisfying \eqref{defi of rho}, there is a constant $\delta_0 > 0$ such that $${\rm dist}(\mathbb{S}_{k}^\bot \cap \B_\rho,  \mathcal{P}) = \delta_0 > 0.$$
	\end{lemma}
	
	\begin{proof}
		By contradiction, we assume that
		$$
		\text{dist}(\mathbb{S}_{k}^\bot \cap  \B_\rho,  \mathcal{P}) = 0.
		$$
		Then we find $\{\vec{u}_n\} \subset \mathbb{S}_{k}^\bot \cap  \B_\rho$ and  $\{\vec{p}_n\} \subset  \mathcal{P}$ such that $\|\vec{u}_n - \vec{p}_n\| \to 0$, and so $\|\vec{u}_n - \vec{p}_n\|_{L^2(\Omega,\R^m)} \to 0$. From  $\{\vec{u}_n\} \subset \B_\rho$ we see that $\{\vec{u}_n\}$ is bounded in $\HR$. Up to a subsequence, we  may assume that $\vec{u}_n \rightharpoonup \vec{u}^* \in \HR$; moreover, we have  $\vec{u}_n \to \vec{u}^*$ strongly in $ L^2(\Omega,\R^m)$. Then we observe that $\vec{u}^* \in \mathbb{S}_{k}^\bot$ and that $ \vec{u}^* \in	S_{\vec{c}} $, which implies that all components of   $\vec{u}^*$  are nontrivial and sign-changing. On the other hand, we have
		\begin{align*}
			\|\vec{u}^* - \vec{p}_n\|_{L^2(\Omega,\R^m)} \le \|\vec{u}^* - \vec{u}_n\|_{L^2(\Omega,\R^m)} + \|\vec{u}_n - \vec{p}_n\|_{L^2(\Omega,\R^m)} \to 0,
		\end{align*}
		and it follows that $\|u^*_i - p_{n,i}\|_{L^2(\Omega)} \to 0$ for $i  =1,\cdots,m$. From $\vec{p}_n \in \p$ it follows that $p_{n,j}$ is nonnegative or nonpositive for some $j \in \{1,\cdots,m\}$. We have reached a contradiction, since $u^*_j$ is sign-changing. The proof is complete.
	\end{proof}


	\begin{lemma} \label{lemma 2.1}
		Let
		\begin{equation}
			M_0 = M_0(\rho) := \frac12- \frac{1}{4} \sbr{  \mu_{\max}^+ + \beta_{\max}^+  }  \mathcal{C}_N^{-4}\rho  >0,
		\end{equation}
		and
		\begin{equation}\label{defi of M1}
			M_1 = M_1(\rho) := M_0\rho.
		\end{equation}
		Then we have $E(\vec{u}) \geq M_0\|\vec{u}\|^2$ for any $\vec{u} \in \overline{	\B_\rho }$ and  $E(\vec{u}) \geq M_1$ for any $\vec{u} \in \partial	\B_\rho   $.
	\end{lemma}
	\begin{proof}
		For any  $u \in \overline{	\B_\rho}$, by using  \eqref{defi of CN}, we obtain
		\begin{equation}\label{e5}
			\sum_{i=1}^m\mu_i\int_\Omega  u_i^4 dx \leq \mu_{\max}^+\mathcal{C}_N^{-4} \sum_{i\in I^+} \sbr{\int_\Omega |\nabla u_i|^2dx}^2 \leq  \mu_{\max}^+  \mathcal{C}_N^{-4} \rho\sum_{i=1}^m \int_\Omega |\nabla u_i|^2dx
		\end{equation}
		and
		\begin{equation} \label{e3}
			\begin{aligned}
				\sum_{i,j=1, j \neq i}^m\beta_{ij}\int_{\Omega}u_i^2u_j^2 dx &\leq 	   \sum_{ (i,j)\in J^+}\beta_{ij}\int_{\Omega}u_i^2u_j^2 dx\leq \beta_{\max}^+ \sum_{(i,j)\in J^+}\sbr{\int_\Omega  u_i^4 dx}^{\frac{1}{2}}\sbr{\int_\Omega  u_j^4 dx}^{\frac{1}{2}}  \\
				&	\leq   \beta_{\max}^+\mathcal{C}_N^{-4} \sum_{(i,j)\in J^+} \sbr{\int_\Omega |\nabla u_i|^2dx\int_\Omega |\nabla u_j|^2dx}\\
				& \leq\beta_{\max}^+C_{N}^{-4}\left(\sum_{i=1}^m\int_{\Omega}\lvert\nabla u_{i}\rvert^2dx\right)\left(\sum_{j=1}^m\lvert\nabla u_{j}\rvert^2dx\right)\\
				&\leq \beta_{\max}^+\mathcal{C}_N^{-4} \rho \sum_{i=1}^m \int_\Omega |\nabla u_i|^2dx ,
			\end{aligned}
		\end{equation}
		from which we get
		\begin{equation}
			\begin{aligned}
				E(\vec{u})& = \frac12\sum_{i=1}^m\int_\Omega |\nabla u_i|^2dx -  \frac14\sum_{i=1}^m \mu_i\int_\Omega  u_i^4dx  -   \frac14\sum_{i,j=1, j \neq i}^m\beta_{ij}\int_\Omega u_i^2u_j^2 dx  \\
				& \geq \mbr{\frac12- \frac{1}{4} \sbr{ \mu_{\max}^+ + \beta_{\max}^+ }  \mathcal{C}_N^{-4}\rho }\sum_{i=1}^m\int_\Omega |\nabla u_i|^2dx
				=  M_0\|\vec{u}\|^2 .
			\end{aligned}
		\end{equation}
		Hence, for any $u \in \partial\B_\rho$, we have  $E(\vec{u}) \geq M_1=M_0\rho$. This completes the proof.
	\end{proof}
	\vskip0.15in
	At this point, we define
	\begin{align*}
		\B_\rho^{M_1} := \bigl\{\vec{u} \in \B_\rho: E(\vec{u}) < M_1\bigr\}.
	\end{align*}
	By the choice of $M_1$, readers can see that the set $\B_\rho^{M_1}$ is invariant under a descending flow, which is the main reason why we introduce $\B_\rho^{M_1}$. Then we consider
	\begin{align} \label{def:lrhok}
		\mathcal{L}_{\rho,k} := \lbr{A \subset \B_\rho: \sup_{\vec{u} \in A}  E(\vec{u}) < M_1, ~ A \text{ is compact and linked to } \mathbb{S}_{k}^\bot},
	\end{align}
	where $M_1$ is the constant depending on $\rho$ given in Lemma \ref{lemma 2.1}, and define a minimax level $m_{\vec{c},\rho,k}^\delta$ by
	\begin{align} \label{def:mcrhokdelta}
		m_{\vec{c},\rho,k}^\delta := \inf_{A \in \mathcal{L}_{\rho,k}}\sup_{\vec{u} \in A \cap S^*(\delta)}E(\vec{u}).
	\end{align}
	
	\begin{lemma} \label{lemma 2.2}
		Let $k$ be a  positive integer such that  $\Lambda_k<\Lambda_{k+1}$.
		For any $\rho>0$ and  $\vec{c} = (c_1,\cdots,c_m)\in \sbr{\R^+}^m$ satisfying \eqref{defi of rho},
		\begin{align} \label{eqci1}
			\frac{1}{4}\mbr{2- \sbr{ \mu_{\min}^- + \beta_{\min}^-} \mathcal{C}_N^{-4} \rho }\Lambda_{k+1}  \sum_{i= 1}^m  c_i   <M_1(\rho)
	\end{align}	
	and
	\begin{align}\label{eqci2}
		& {\sbr{\beta_{\max}^+-\mu_{\min}^- - \beta_{\min}^-} \mathcal{C}_N^{-4} \rho }\Lambda_{k+1} { \sum_{i=1}^m c_i} +  \mu_{\max}^+\mathcal{C}_N^{-4}\Lambda_{k+1} ^2  \sum_{i=1}^m{c_i }^2
		\\
		&<\mbr{2- \sbr{ \mu_{\min}^- + \beta_{\min}^-} \mathcal{C}_N^{-4} \rho }(\Lambda_{k +1} -\Lambda_{k})\min_{1 \le i \le m} c_i
\end{align}
we have	
\begin{equation}
	\mathbb{S}_{k+1} \subset 	\B_\rho , \quad \sup\limits_{\mathbb{S}_{k+1}}E(\vec{u}) < M_1, \quad \displaystyle \sup_{\partial \mathbb{M}_{k+1}}E(\vec{u}) < \inf_{  	\B_\rho \cap\mathbb{S}_{k}^\bot}E(\vec{u}).
\end{equation}
\end{lemma}
\begin{proof}
For any $\vec{u}\in 	\mathbb{S}_{k+1}$, observe that for $i \in \lbr{1,\ldots,m}$, there holds
\begin{equation} \label{s1}
	\begin{aligned}
		\int_\Omega|\nabla u_i|^2 dx = \sum_{j=1}^{k+1} t_{i,j}^2 c_i \int_\Omega|\nabla \phi_j|^2 dx = \sum_{j=1}^{k+1} t_{i,j}^2 c_i  \Lambda_j \leq c_i  \Lambda_{k+1}.
	\end{aligned}
\end{equation}
Moreover,  it follows from  \eqref{defi of M1} and \eqref{eqci1}  that $ \Lambda_{k+1}\sum_{i= 1}^m  {c_i }<\rho$. Hence, $$ \nm{ \vec{u}}^2 = \sum_{i=1}^m \int_\Omega|\nabla u_i|^2 dx \leq \Lambda_{k+1}\sum_{i= 1}^m  {c_i }<\rho,$$
from which we get   $	\mathbb{S}_{k+1} \subset 	\B_\rho $.
\medbreak
In the following, for any $\vec{u}\in 	\mathbb{S}_{k+1} \subset \B_\rho$, similar to \eqref{e5} and \eqref{e3},    we deduce from \eqref{eqci1} and \eqref{s1}  that
\begin{equation} \label{e7}
	\begin{aligned}
		E(\vec{u}) &\leq  \frac{1}{2}\sum_{i=1}^m\int_\Omega |\nabla u_i|^2dx  -\frac{1}{4} \sum_{i\in I^-} \mu_i  \int_\Omega   u_i^4 dx -\frac14 \sum_{(i,j)\in J^-} \beta_{ij} \int_{\Omega}u_i^2u_j^2 dx \\
		& \leq  \frac{1}{2}\sum_{i=1}^m\int_\Omega |\nabla u_i|^2dx - \frac{1}{4}\mu_{\min}^- \mathcal{C}_N^{-4} \sum_{i\in I^-} \sbr{ \int_\Omega |\nabla u_i|^2dx}^2\\
		&\qquad -\frac14 \beta_{\min}^-\mathcal{C}_N^{-4} \sum_{(i,j)\in J^-}\sbr{\int_\Omega |\nabla u_i|^2dx\int_\Omega |\nabla u_j|^2dx}\\
		&\leq \frac{1}{4}\mbr{2- \sbr{ \mu_{\min}^- + \beta_{\min}^-} \mathcal{C}_N^{-4} \rho }\sum_{i=1}^m\int_\Omega |\nabla u_i|^2dx \\
		&\leq \frac{1}{4}\mbr{2- \sbr{  \mu_{\min}^- + \beta_{\min}^-} \mathcal{C}_N^{-4} \rho }\Lambda_{k+1}\sum_{i= 1}^m  {c_i } ,
	\end{aligned}
\end{equation}
from which, together with \eqref{eqci1}, we get $ \sup\limits_{\mathbb{S}_{k+1}}E(\vec{u})  < M_1.  $

Finally,  for any $\vec{u} \in	\B_\rho \cap \mathbb{S}_{k}^\bot $,
we have  $c_i \Lambda_{k+1}  \leq  \int_\Omega |\nabla u_i|^2dx  <\rho$ for $i\in \lbr{1,\ldots,m}$. Then 	 from \eqref{defi of rho}, \eqref{e5}, \eqref{e3},    we obtain
\begin{equation} \label{e9}
		\begin{aligned}
			\inf_{  	\B_\rho \cap\mathbb{S}_{k}^\bot}	E(\vec{u}) &\geq 	\inf_{  	\B_\rho \cap\mathbb{S}_{k}^\bot} \frac 14 \sum_{i=1}^m \mbr{  \sbr{2-  \beta_{\max}^+\mathcal{C}_N^{-4}\rho}\int_\Omega |\nabla u_i|^2dx -\mu_{\max}^+\mathcal{C}_N^{-4}  \sbr{\int_\Omega |\nabla u_i|^2dx}^2}
			\\& \geq \frac 14 \sum_{i=1}^m \mbr{  \sbr{2- \beta_{\max}^+\mathcal{C}_N^{-4}\rho}c_i\Lambda_{k+1}  -\mu_{\max}^+\mathcal{C}_N^{-4}  \sbr{c_i\Lambda_{k+1} }^2}.
		\end{aligned}
\end{equation}
The energy lower bound estimate in \eqref{e9} involves the quadratic function
	\begin{align}
		f(x) = A x - B x^2, \quad \text{with } A = 2 - \beta_{\max}^+ C_N^{-4}\rho > 0, \ B = \mu_{\max}^+ C_N^{-4} > 0,
	\end{align}
	where $x = \int_{\Omega} |\nabla u_i|^2 dx$. By the choice of $\rho$ in \eqref{defi of rho}, we have $f'(x) = A - 2Bx > 0$ for all $x < \rho$, which means $f(x)$ is strictly increasing on the interval $[0, \rho]$. Therefore, to obtain the lower bound of $f(x)$, we take the minimum value of $x$, i.e., $x = c_i \Lambda_{k+1}$, which gives
	$$
	f(x) \geq f(c_i \Lambda_{k+1}) = A c_i \Lambda_{k+1} - B (c_i \Lambda_{k+1})^2.
	$$ 
Observe that $$\partial \mathbb{M}_{k+1} =\bigcup_{s=1}^m \sbr{  \mathbb{M}_{k+1}^1\times \cdots\times \partial \mathbb{M}_{k+1}^s \times \cdots\times  \mathbb{M}_{k+1}^m}.$$ Then for any $s\in \lbr{1,\ldots,m}$ and   $\vec{u}=\sbr{u_1,\ldots,u_m} \in  \mathbb{M}_{k+1}^1\times \cdots\times \partial \mathbb{M}_{k+1}^s \times \cdots\times  \mathbb{M}_{k+1}^m$, we have
\begin{equation}
	\begin{aligned}
		\int_\Omega |\nabla u_s|^2dx= \sum_{j=1}^{k } t_{s,j}^2 c_s \int_\Omega|\nabla \phi_j|^2 dx = \sum_{j=1}^{k } t_{s,j}^2 c_s  \Lambda_j \leq c_s  \Lambda_{k}< c_s  \Lambda_{k+1},
	\end{aligned}
\end{equation}
and
\begin{equation}
	\int_\Omega |\nabla u_i|^2dx= \sum_{j=1}^{k+1} t_{i,j}^2 c_i \int_\Omega|\nabla \phi_j|^2 dx = \sum_{j=1}^{k+1 } t_{i,j}^2 c_i  \Lambda_j \leq c_i  \Lambda_{k+1}, \quad \text{ for }~~i\neq s.
\end{equation}
Hence,
\begin{equation}\label{e8}
	\sup_{\partial \mathbb{M}_{k+1}}\sum_{i=1}^m\int_\Omega |\nabla u_i|^2dx \leq {\Lambda_{k+1}  { \sum_{i=1}^m c_i}  - (\Lambda_{k +1} -\Lambda_{k})\min_{1 \le i \le m} c_i}.
\end{equation}
Therefore, by \eqref{e7}, we have
\begin{equation} \label{e10}
		\begin{aligned}
			\sup_{\partial \mathbb{M}_{k+1}}    E(\vec{u})  &\leq \frac{1}{4}\mbr{2- \sbr{  	\mu_{\min}^- +  \beta_{\min}^-} \mathcal{C}_N^{-4} \rho }\sup_{\partial \mathbb{M}_{k+1}}\sum_{i=1}^m\int_\Omega |\nabla u_i|^2dx \\
			&\leq \frac{1}{4}\mbr{2- \sbr{ 	\mu_{\min}^- +  \beta_{\min}^-} \mathcal{C}_N^{-4} \rho }\sbr{\Lambda_{k+1}  { \sum_{i=1}^m c_i}  - (\Lambda_{k +1} -\Lambda_{k})\min_{1 \le i \le m} c_i}.
		\end{aligned}
\end{equation}
Therefore, from \eqref{eqci2}, \eqref{e9}, and \eqref{e10} we obtain  that $ \sup_{\partial \mathbb{M}_{k+1}}E(\vec{u}) < \inf_{  	\B_\rho \cap\mathbb{S}_{k}^\bot}E(\vec{u})$. The proof is completed.
\end{proof}

\begin{remark} \label{rmktildeck}
{\rm Let  $k$ be a  positive integer such that $\Lambda_k<\Lambda_{k+1}$, and  for any $\vec{c} = (c_1,\cdots,c_m) \in \sbr{\R^+}^m$, we  denote $c_{\min} := \min\{c_1,\cdots,c_m\}$, $c_{\max} := \max\{c_1,\cdots,c_m\}$. Then  for  small constant $\tilde c_k > 0$ satisfying $\frac{c_{\max}^2}{c_{\min}} < \tilde c_k$ (hence, $c_{\max}$ is also small), we can always take $\rho > 0$ such that \eqref{defi of rho}, \eqref{eqci1} and \eqref{eqci2} hold true. Indeed, by choosing $\rho$ satisfying
\begin{align}\label{rho}
	\Lambda_{k+1}\sum_{i=1}^mc_i < \frac34 \rho < \rho < 2\Lambda_{k+1}\sum_{i=1}^mc_i,
\end{align}
the relation  \eqref{defi of rho} follows directly since $c_{\max}$ is  small. Then
\begin{align*}
	\Lambda_{k+1}\sum_{i=1}^mc_i <	\frac34 \rho < \frac{ 2-\sbr{\mu_{\max}^+ +\beta_{\max}^+ }\mathcal{C}_N^{-4} \rho}{2-\sbr{	\mu_{\min}^- +  \beta_{\min}^-}\mathcal{C}_N^{-4} \rho},
\end{align*}
from which  we obtain the relation  \eqref{eqci1}. Finally, observe that \eqref{eqci2} is equivalent to
\begin{equation} \label{e13}
	\begin{aligned}
		\rho   <\overline{A} :=\frac{\mathcal{C}_N^{4}}{A_1-A_2 } \sbr{ 2 (\Lambda_{k +1} -\Lambda_{k})\frac{c_{\min}}{\sum_{i=1}^mc_i^2} -   \mu_{\max}^+ \mathcal{C}_N^{-4}\Lambda_{k+1} ^2 } \frac{\sum_{i=1}^mc_i^2}{\sum_{i=1}^mc_i },
	\end{aligned}
\end{equation}
where
\begin{equation}
	\begin{aligned}
		A_1&:= \sbr{\beta_{\max}^+-	\mu_{\min}^- -  \beta_{\min}^-}\Lambda_{k+1}>0,\\
		A_2&:= -\sbr{	\mu_{\min}^- +  \beta_{\min}^-}(\Lambda_{k +1} -\Lambda_{k})\frac{c_{\min}}{\sum_{i=1}^mc_i } >0.
	\end{aligned}
\end{equation}
Obviously,  $0<A_1-A_2<A_1$.
Since $\frac{c_{\max}^2}{c_{\min}}$ is small, we can assume that $\frac{c_{\min}}{\sum_{i=1}^mc_i^2}$ is large such that
\begin{equation}
	\begin{aligned}
		\frac{1}{m} \frac{\mathcal{C}_N^{4}}{A_1   } \sbr{ 2 (\Lambda_{k +1} -\Lambda_{k})\frac{c_{\min}}{\sum_{i=1}^mc_i^2} -   \mu_{\max}^+ \mathcal{C}_N^{-4}\Lambda_{k+1} ^2 }> 2\Lambda_{k +1}.
	\end{aligned}
\end{equation}
Then  \eqref{e13} holds, because \begin{align*}
	\frac{\sum_{i=1}^mc_i^2}{\sum_{i=1}^mc_i} \ge \frac{1}{m}\sum_{i=1}^mc_i.
\end{align*}
Therefore, all the requirements in \eqref{defi of rho}, \eqref{eqci1} and \eqref{eqci2} are satisfied by taking $\rho$ with \eqref{rho}.}
\end{remark}

\begin{lemma} \label{lemnemp}
For any $\rho$ and  $\vec{c} = (c_1,\cdots,c_m)\in \sbr{\R^+}^m$ satisfying \eqref{defi of rho},  \eqref{eqci1} and \eqref{eqci2}, we have $\mathbb{M}_{k+1} \in \mathcal{L}_{\rho,k}$, and so, $\mathcal{L}_{\rho,k} \neq \emptyset$.
\end{lemma}

\begin{proof}
By Lemma \ref{lemma 2.2}, we have  $\mathbb{M}_{k+1} \subset \mathbb{S}_{k+1} \subset \B_\rho$ and $\sup E(\mathbb{M}_{k+1}) < M_1$. Obviously, we know that  $\mathbb{M}_{k+1} \cap \mathbb{S}_{k}^\bot \neq \emptyset$, which consists of one point denoted by $\sbr{\sqrt{c_1}\phi_{k+1},\ldots,\sqrt{c_m}\phi_{k+1}}$. Then,  for any continuous mapping $h \in C( \mathbb{M}_{k+1},S_{\vec{c}})$  satisfying $h|_{ \partial \mathbb{M}_{k+1} } = \textbf{id}$, we claim that $h(\mathbb{M}_{k+1}) \cap \mathbb{S}_{k}^\bot \not= \emptyset$.  In fact, observe that  $\mathbb{M}_{k+1}^i \cong \mathbb{P}_k^i  \subset H_k$, where
\begin{equation}
	\mathbb{P}_k^i:=\lbr{ u_i = \sum_{j=1}^{k}t_{i,j}\sqrt{c_i}\phi_j: \sum_{j=1}^{k}t_{i,j}^2 \leq 1}, \quad i \in \lbr{1,\ldots, m} .
\end{equation} Hence, $\mathbb{M}_{k+1} \cong \mathbb{P}_k^1\times \cdots \times  \mathbb{P}_{k}^m $, and the topology degree $\deg(P_k \circ h,\mathbb{M}_{k+1} ,\vec{0})$ and $\deg(P_k|_{\mathbb{M}_{k+1} },\mathbb{M}_{k+1},\vec{0})$ are well-defined where $P_k : \HR \to \HR_k$ is the orthogonal projection. Let
\begin{align*}
	g (t,u) := (1-t)P_k(h(\vec{u})) + t P_k|_{\mathbb{M}_{k+1}} \vec{u}, \quad \forall t \in [0,1], \quad \vec{u} \in \mathbb{M}_{k+1}.
\end{align*}
For any $u \in \partial \mathbb{M}_{k+1}$ and $t \in [0,1]$, we have $g(t,u) = P_k u \neq \vec{0}$. Hence, we have
\begin{align*}
	\deg(P_k \circ h,\mathbb{M}_{k+1},\vec{0}) = \deg(P_k|_{\mathbb{M}_{k+1}},\mathbb{M}_{k+1},\vec{0}) = 1,
\end{align*}
and so, there exists a $\vec{u}_0 \in \mathbb{M}_{k+1}$ such that $P_k(h(\vec{u}_0)) = \vec{0}$, i.e., $h(\vec{u}_0) \in  {\HR_k^\bot} \cap S_{\vec{c}}=\mathbb{S}_{k}^\bot$. This yields that $h(\mathbb{M}_{k+1}) \cap   \mathbb{S}_{k}^\bot \not= \emptyset$ and we complete the proof.
\end{proof}

\begin{proposition} \label{propwelldefined}
For any $\rho$ satisfying \eqref{defi of rho} and for any  $\vec{c} = (c_1,\cdots,c_m) \in \sbr{\R^+}^m$ with  \eqref{eqci1}, \eqref{eqci2}, then the value $m_{\vec{c},\rho,k}^\delta$ is well defined and finite for $0 < \delta < \delta_0$, where $\delta_0$ is given in Lemma \ref{lem>0}. Moreover,
\begin{align} \label{eqestiofm}
	\sup_{\partial \mathbb{M}_{k+1}}E(\vec{u}) < \inf_{  	\B_\rho \cap \mathbb{S}_{k}^\bot}E(\vec{u}) \le m_{\vec{c},\rho,k}^\delta \le \sup_{\mathbb{M}_{k+1}}E(\vec{u}) < M_1.
\end{align}
\end{proposition}

\begin{proof}
By Lemma \ref{lemnemp}, we know $\mathcal{L}_{\rho,k}$ is not empty. Moreover, by the choice of $\delta$, we have $A \cap S^*(\delta) \neq \emptyset$ for any $A \in \mathcal{L}_{\rho,k}$. Thus $m_{\vec{c},\rho,k}^\delta$ is well defined. Additionally, on the one hand, according to the definition of linking and using Lemma \ref{lemma 2.2}, we obtain
\begin{align} \label{eqvaluesepe}
	m_{\vec{c},\rho,k}^\delta \geq \inf_{  	\B_\rho   ^{M_1}\cap\mathbb{S}_{k}^\bot}E(\vec{u}) > \sup_{\partial \mathbb{M}_{k+1}}E(\vec{u}) > -\infty.
\end{align}
On the other hand, by $\mathbb{M}_{k+1} \in \mathcal{L}_{\rho,k}$ it follows that
\begin{align*}
	m_{\vec{c},\rho,k}^\delta \leq \sup_{\mathbb{M}_{k+1}}E(\vec{u}) < M_1.
\end{align*}
Therefore, the value $m_{\vec{c},\rho,k}^\delta$ is finite.
\end{proof}

\subsection{The minimax level for semi-nodal normalized solutions}\label{subsection 2.3}
Let   $d$  be an integer such that $1 \le d \le m-1$. We are concerned with the existence of semi-nodal normalized solutions with the property that the first $d$ components $u_1,\ldots,u_d$ change sign,  while the remaining components $ u_{d+1},\ldots, u_m$ are positive.

\medbreak
To deal with  the sign-changing components $\sbr{u_1,\ldots,u_d}$, we use the notations in Subsection \ref{subsection 2.2} and define the following product spaces:
$$ \HR_{k,d}:= \underbrace{H_{k }\times H_{k }\times \cdots \times H_{k }}_{d} , \quad \text{ and } \quad  \HR_{k,d}^\bot:= \underbrace{H_{k }^\bot\times H_{k }^\bot\times \cdots \times H_{k }^\bot}_{d} . $$
Moreover, we consider
\begin{equation}
S_{\vec{c},d}: = \left\{ (u_{1},\cdots, u_d) \in H_0^1(\Omega,\R^{d}): ~\int_\Omega u_i^2dx = c_i, ~ i  \in  \lbr{1,\cdots,d} \right\}.
\end{equation}
and introduce the subspace
\begin{equation}
\mathbb{S}_{ k,d}  :=  S _{\vec{c},d} \cap \HR_ {k,d} , \quad \text{ and } \quad  \mathbb{S}_{ k,d}^\bot  :=    S _{\vec{c},d} \cap \HR_ {k,d}^\bot. \end{equation}
Then we have $\mathbb{S}_{ k,d} \cong \mathbb{S}_k^{1 } \times \mathbb{S}_k^{2} \times \cdots \times \mathbb{S}_k^{d}$, where  $\mathbb{S}_k^{i }$ is  defined in \eqref{defi of Sk}. Let
\begin{equation}
\mathbb{M}_{k+1,d}  := \mathbb{M}_{k+1 }^1 \times\mathbb{M}_{k+1 }^2 \times \cdots \times \mathbb{M}_{k+1 }^{d},
\end{equation}
where  $\mathbb{M}_k^{i } $  is defined in  \eqref{defi of Mk}. Obviously,  there holds  $ \mathbb{S}_{ k,d} \subset \mathbb{M}_{k+1,d}$.
\medbreak
The abstract definition \ref{def:partial link} of partial vector linking  in the  specific case where $\mathcal{M}_i = S_{c_i}$,  $\mathcal{M}^{(m)}= S_{\vec{c}}$, $D = \partial \mathbb{M}_{k+1,d}$ and $S = \mathbb{S}_{k,d}^\bot$ is stated as follows.

\begin{definition}[Partial vector linking] \label{link semi-nodal}
{\rm Let
\begin{equation}
	T_d:H_0^1(\Omega,\R^m) \to H_0^1(\Omega,\R^d) ,  \quad \vec{u} = (u_1,\cdots,u_m) \mapsto (u_1,\cdots,u_d)
\end{equation}
be the projection map.  Then a compact subset $A$ of $S_{\vec{c}}$ with $\partial \mathbb{M}_{k+1,d} \times \lbr{(\sqrt{c_{d+1}}\phi_1, \cdots, \sqrt{c_{m}}\phi_1)} \subset A$ is said to be partially $d$-linked to $\mathbb{S}_{k,d}^\bot $, if, for any continuous mapping $h \in C(A, {S_{\vec{c}}})$ satisfying
\[h|_{\partial \mathbb{M}_{k+1,d} \times \lbr{(\sqrt{c_{d+1}}\phi_1, \cdots, \sqrt{c_{m}}\phi_1)}} = \textbf{id},\]
there holds $$T_d\circ h(A) \cap\mathbb{S}_{k,d}^\bot  \not= \emptyset.$$}
\end{definition}
\medbreak
In what follows, we focus on the construction of minimax levels at which  semi-nodal  normalized solutions can be found.  Let
\begin{equation}
\mathcal{P}_i:=\lbr{\vec{u} = (u_1,u_2,\cdots,u_m)\in \HR: u_i\geq 0}; \quad \mathcal{P}^{(d)} :=\bigcup_{i=1}^d(\mathcal{P}_i \cup -\PR_i).
\end{equation}
Then the first $d$ components of elements outside $\mathcal{P}^{(d)}$ are sign-changing.  For $\delta>0$, we define $\PR^{(d)}_\delta:= \bigl\{\vec{u} \in \HR: \text{dist}(\vec{u}, \PR^{(d)}) <\delta\}$, where
\begin{equation}
\text{dist}(\vec{u}, \PR^{(d)}):=\min \lbr{\text{dist}(\vec{u}, \PR_i), ~\text{dist}(\vec{u},-\PR_i): i = 1,\ldots,d}.
\end{equation}
Note that
\begin{align*}
\PR^{(d)}_\delta = \bigcup_{i = 1}^d \sbr{(\PR_i)_\delta \cup (-\PR_i)_\delta}, \quad (\pm\PR_i)_\delta = \lbr{ \vec{u} \in \HR: \text{dist}(\vec{u}, \pm\PR_i) <\delta }.
\end{align*}
Now we introduce the set $S_d^*(\delta)$ by
\begin{align}
S_d^*(\delta) := \mathcal{H} \backslash \overline{\PR^{(d)}_\delta},
\end{align}
and define
\begin{equation}
S_{\vec{c},m-d}: = \left\{ (u_{d+1},\cdots, u_m) \in H_0^1(\Omega,\R^{m-d}): ~\int_\Omega u_i^2dx = c_i, ~ i  \in  \lbr{d+1,\cdots,m}\right\}.
\end{equation}
Recall the definition of $\B_\rho$ in \eqref{defBrho}, the following lemma repeats the result in Lemma \ref{lem>0} in a case replacing $\p$ with $\p^{(d)}$. For convenience, we provide details of the proof.
\begin{lemma} \label{lem>0 semi-nodal}
For any $\rho > 0$ satisfying \eqref{defi of rho},  there is a constant $\delta_{0,d} > 0$ such that
\begin{align*}
	{\rm dist} \sbr{\B_\rho \cap \sbr{\mathbb{S}_{k,d}^\bot \times S_{\vec{c},m-d }},  \mathcal{P}^{(d)}}  = \delta_{0,d} > 0.
\end{align*}
\end{lemma}

\begin{proof}
Assume by contradiction that
$$
\text{dist} \sbr{\B_\rho \cap \sbr{\mathbb{S}_{k,d}^\bot \times S_{\vec{c},m-d }},  \mathcal{P}^{(d)}} =0.
$$
Then there exists a sequence  $\{\vec{u}_n\} \subset\B_\rho   \cap \sbr{\mathbb{S}_{k,d}^\bot \times S_{\vec{c},m-d }}$ and  $\{\vec{p}_n\} \subset  \mathcal{P}^{(d)}$ such that $\|\vec{u}_n - \vec{p}_n\| \to 0$, and so $\|\vec{u}_n - \vec{p}_n\|_{L^2(\Omega,\R^m)} \to 0$. By $\{\vec{u}_n\} \subset \B_\rho$ it follows that $\{\vec{u}_n\}$ is bounded in $\HR$. Up to a subsequence, we  may assume that $\vec{u}_n \rightharpoonup \vec{u}^* \in \HR$ weakly in  $\HR$; moreover, we have  $\vec{u}_n \to \vec{u}^*$ strongly in $ L^2(\Omega,\R^m)$. Then we observe that $\vec{u}^* \in \mathbb{S}_{k,d}^\bot \times S_{\vec{c},m-d }$, which implies that the first $d$ components of   $\vec{u}^*$ are sign-changing. On the other hand, we have
\begin{align*}
	\|\vec{u}^* - \vec{p}_n\|_{L^2(\Omega,\R^m)} \le \|\vec{u}^* - \vec{u}_n\|_{L^2(\Omega,\R^m)} + \|\vec{u}_n - \vec{p}_n\|_{L^2(\Omega,\R^m)} \to 0,
\end{align*}
and it follows that $\|u^*_i - p_{n,i}\|_{L^2(\Omega)} \to 0$ for $i  =1,\cdots,d$. From $\vec{p}_n \in \p^{(d)}$ it follows that $p_{n,j}$ is nonnegative or nonpositive for some $j \in \{1,\cdots,d\}$. We have reached a contradiction, since $u^*_j$ is sign-changing. The proof is complete.
\end{proof}




Next, we consider
\begin{align*}
\mathcal{L}_{\rho,k,d}:= \lbr{A \subset \sbr{\bigcap_{i = d+1}^m\p_i} \cap \B_\rho: \sup_{\vec{u} \in A}  E(\vec{u}) < M_1, ~A \text{ is partially $d$-linked to } \mathbb{S}_{k,d}^\bot },
\end{align*}
where $M_1=M_1(\rho)$ is the constant given in Lemma \ref{lemma 2.1}, and introduce the minimax level 
\begin{align}
m_{\vec{c},\rho,k,d}^\delta := \inf_{A \in \mathcal{L}_{\rho,k,d}}\sup_{\vec{u} \in A \cap S_{d}^*(\delta)}E(\vec{u}).
\end{align}

	\begin{lemma} \label{lemma 5.2}
		Let $k$ be a  positive integer such that  $\Lambda_k<\Lambda_{k+1}$. For any $\rho>0$ and  $\vec{c} = (c_1,\cdots,c_m) \in \sbr{\R^+}^m$ satisfying \eqref{defi of rho},
		\begin{align} \label{eqci5.1}
			\frac{1}{4}\mbr{2- \sbr{  	\mu_{\min}^- + \beta_{\min}^-} \mathcal{C}_N^{-4} \rho }\Lambda_{k+1}  \sum_{i= 1}^m  c_i   <M_1(\rho) ,
		\end{align}	
		and
		\begin{equation}\label{eqci5.2}
			\begin{aligned}
				&{\sbr{\beta_{\max}^+ -	\mu_{\min}^-  -  \beta_{\min}^-} \mathcal{C}_N^{-4} \rho }\sbr{\Lambda_{k+1} \sum_{i=1}^d c_i + \Lambda_1\sum_{i=d+1}^m c_i}+   \mu_{\max}^+ \mathcal{C}_N^{-4}\sbr{\Lambda_{k+1}^2 \sum_{i=1}^d c_i^2 + \Lambda_1^2\sum_{i=d+1}^m c_i^2}
				\\
				&<\mbr{2- \sbr{ 	\mu_{\min}^-  +  \beta_{\min}^-} \mathcal{C}_N^{-4} \rho }(\Lambda_{k +1} -\Lambda_{k})\min_{1 \le i \le d} c_i,
			\end{aligned}		
		\end{equation}
		we have	
		\begin{equation} \label{s10}
			\mathbb{S}_{ k+1,d}  \times \lbr{(\sqrt{c_{d+1}}\phi_1, \cdots, \sqrt{c_{m}}\phi_1)} \subset 	\B_\rho , \quad \sup\limits_{\mathbb{S}_{ k+1,d} \times \lbr{(\sqrt{c_{d+1}}\phi_1, \cdots, \sqrt{c_{m}}\phi_1)}}E(\vec{u}) < M_1,
		\end{equation}
		and
		\begin{equation}\label{s11}
			\displaystyle \sup_{\partial \mathbb{M}_{k+1,d} \times \lbr{(\sqrt{c_{d+1}}\phi_1, \cdots, \sqrt{c_{m}}\phi_1)}}E(\vec u) <  \inf_{ \B_\rho \cap\sbr{\mathbb{S}_{k,d}^\bot \times S_{\vec{c},m-d}}}E(\vec u) .
		\end{equation}
	\end{lemma}
	
	\begin{proof}
		For any $\vec{u}\in \mathbb{S}_{ k+1,d} \times \lbr{(\sqrt{c_{d+1}}\phi_1, \cdots, \sqrt{c_{m}}\phi_1)}$, similar to \eqref{s1}, we have
		\begin{equation}
			\int_\Omega |\nabla u_i|^2 dx \leq  c_i  \Lambda_{k+1}, \quad \text{ for } ~ i \in \lbr{1,\ldots,d},
		\end{equation}
		which, together with $ \Lambda_{1}< \Lambda_{k+1}$, gives
		\begin{equation}\label{ss1}
			\begin{aligned}
				\nm{ \vec{u}}^2 & \leq \sum_{i= 1}^d \int_\Omega |\nabla u_i|^2 dx  +  \sum_{i=d+1}^m c_i	\int_\Omega |\nabla \phi_i|^2 dx \\
				&  \leq \Lambda_{k+1}\sum_{i= 1}^d   c_i    +  \Lambda_{1} \sum_{i=d+1}^m   c_i  \\
				&< \Lambda_{k+1}\sum_{i= 1}^m   c_i  <\rho,
			\end{aligned}	
		\end{equation}
		where the last inequality follows from \eqref{defi of M1}  and \eqref{eqci1}. Hence, 	  the first relation in \eqref{s10} holds.
		\medbreak
		For any $\vec{u}\in 	\mathbb{S}_{ k+1,d}   \times \lbr{(\sqrt{c_{d+1}}\phi_1, \cdots, \sqrt{c_{m}}\phi_1)} \subset 	\B_\rho$, similar to \eqref{e7}, we  deduce from \eqref{ss1}   that
		\begin{equation} \label{Eveculeq}
			\begin{aligned}
				E(\vec{u})
				&\leq \frac{1}{4}\mbr{2- \sbr{  	\mu_{\min}^-  +  \beta_{\min}^-} \mathcal{C}_N^{-4} \rho }\sum_{i=1}^m\int_\Omega |\nabla u_i|^2dx \\
				&< \frac{1}{4}\mbr{2- \sbr{ 	\mu_{\min}^-   +  \beta_{\min}^-} \mathcal{C}_N^{-4} \rho }   \Lambda_{k+1}\sum_{i= 1}^m  c_i  ,
			\end{aligned}
		\end{equation}
		from which, together with \eqref{eqci1}, we get $\sup\limits_{\mathbb{S}_{ k+1,d} \times \lbr{(\sqrt{c_{d+1}}\phi_1, \cdots, \sqrt{c_{m}}\phi_1)}}E(\vec{u}) < M_1$.

		%
		
		\medbreak
		Finally,  for any  {$\vec{u} \in	\B_\rho \cap \sbr{\mathbb{S}_{k,d}^\bot \times S_{\vec{c},m-d}} \subset  \B_\rho $},    there holds
		\begin{equation} \label{e21}
			c_i \Lambda_{k+1} \leq 	\int_\Omega |\nabla u_i|^2dx \leq  \rho  \quad \forall i\in \lbr{1,\ldots,d},
		\end{equation}
		and
		\begin{equation} \label{e22}
			c_i \Lambda_{1} \leq 	\int_\Omega |\nabla u_i|^2dx \leq  \rho  \quad \forall i\in \lbr{d+1,\ldots,m}.
		\end{equation}
		Then similar to  \eqref{e9}, we have		
		\begin{equation}
			\begin{aligned} \label{qw1}
				\inf_{   \B_\rho \cap\sbr{\mathbb{S}_{k,d}^\bot \times S_{\vec{c},m-d}}}	E(\vec{u})
				& \geq \frac 14 \sum_{i=1}^d \mbr{  \sbr{2- \beta_{\max}^+\mathcal{C}_N^{-4}\rho}c_i\Lambda_{k+1}  - \mu_{\max}^+ \mathcal{C}_N^{-4}  \sbr{c_i\Lambda_{k+1} }^2}\\
				&+ \frac 14 \sum_{i=d+1}^m \mbr{  \sbr{2-\beta_{\max}^+\mathcal{C}_N^{-4}\rho}c_i\Lambda_{1}  - \mu_{\max}^+ \mathcal{C}_N^{-4}  \sbr{c_i\Lambda_{1} }^2}.
			\end{aligned}
		\end{equation}
		Observe that $$\partial \mathbb{M}_{k+1,d} =\bigcup_{s=1}^d \sbr{  \mathbb{M}_{k+1}^1\times \cdots\times \partial \mathbb{M}_{k+1}^s \times \cdots\times  \mathbb{M}_{k+1}^d}.$$ Then for any $s\in \lbr{1,\ldots,d}$ and   $\vec{u}_d=\sbr{u_1,\ldots,u_d} \in  \mathbb{M}_{k+1}^1\times \cdots\times \partial \mathbb{M}_{k+1}^s \times \cdots\times  \mathbb{M}_{k+1}^d$, we have
		\begin{equation}
			\begin{aligned}
				\int_\Omega |\nabla u_s|^2dx= \sum_{j=1}^{k } t_{s,j}^2 c_s \int_\Omega|\nabla \phi_j|^2 dx = \sum_{j=1}^{k } t_{s,j}^2 c_s  \Lambda_j \leq c_s  \Lambda_{k},
			\end{aligned}
		\end{equation}
		and
		\begin{equation}
			\int_\Omega |\nabla u_i|^2dx= \sum_{j=1}^{k+1} t_{i,j}^2 c_i \int_\Omega|\nabla \phi_j|^2 dx = \sum_{j=1}^{k+1 } t_{i,j}^2 c_i  \Lambda_j \leq c_i  \Lambda_{k+1}, \quad \text{ for }~~i\neq s.
		\end{equation}
		Hence,
		\begin{equation}
			\begin{aligned}
				&\sup_{\partial \mathbb{M}_{k+1,d}  \times \lbr{(\sqrt{c_{d+1}}\phi_1, \cdots, \sqrt{c_{m}}\phi_1)}}\sum_{i=1}^m\int_\Omega |\nabla u_i|^2dx
				\\& \leq {\Lambda_{k+1} { \sum_{i=1}^d c_i}  - (\Lambda_{k +1} -\Lambda_{k})\min_{1 \le i \le d} c_i}+ \Lambda_{1} { \sum_{i=d+1}^m c_i} .
			\end{aligned}
		\end{equation}
		Therefore, similar to \eqref{e7}, we have
		\begin{equation}  \label{qw2}
			\begin{aligned}
				&\sup_{\partial \mathbb{M}_{k+1,d} \times \lbr{(\sqrt{c_{d+1}}\phi_1, \cdots, \sqrt{c_{m}}\phi_1)}}    E(\vec{u}) \\ &\leq \frac{1}{4}\mbr{2- \sbr{ 	\mu_{\min}^- +  \beta_{\min}^-} \mathcal{C}_N^{-4} \rho }\sup_{\partial \mathbb{M}_{k+1,d} \times \lbr{(\sqrt{c_{d+1}}\phi_1, \cdots, \sqrt{c_{m}}\phi_1)}}\sum_{i=1}^m\int_\Omega |\nabla u_i|^2dx \\
				&\leq \frac{1}{4}\mbr{2- \sbr{  \mu_{\min}^- +  \beta_{\min}^-} \mathcal{C}_N^{-4} \rho }\sbr{\Lambda_{k+1}  { \sum_{i=1}^d c_i} +\Lambda_{1} { \sum_{i=d+1}^m c_i} - (\Lambda_{k +1} -\Lambda_{k})\min_{1 \le i \le d} c_i}.
			\end{aligned}
		\end{equation}
		Hence, from \eqref{eqci5.2}, \eqref{qw1}, and \eqref{qw2} we obtain  that \eqref{s11} holds.  The proof is completed.
	\end{proof}	
	
	\begin{lemma} \label{lemnemp semi-nodal}
		For any $\rho>0$  and $\vec{c} = (c_1,\cdots,c_m) \in \sbr{\R^+}^m $ satisfying \eqref{defi of rho},  \eqref{eqci5.1}, \eqref{eqci5.2}, we have  $\mathbb{M}_{k+1,d}\times \lbr{(\sqrt{c_{d+1}}\phi_1, \cdots, \sqrt{c_{m}}\phi_1)}  \in \mathcal{L}_{\rho,k,d}$, and so, $\mathcal{L}_{\rho,k,d} \neq \emptyset$.
	\end{lemma}
	
	\begin{proof}
		Since $\mathbb{M}_{k+1,d}\times \lbr{(\sqrt{c_{d+1}}\phi_1, \cdots, \sqrt{c_{m}}\phi_1)} \in (\cap_{i = d+1}^m \p_i)$, then by Lemma \ref{lemma 5.2}, it belongs to the set $(\cap_{i = d+1}^m \p_i) \cap \B_\rho$ and we have $$\sup\limits_{\vec{u} \in	\mathbb{M}_{k+1,d}\times \lbr{(\sqrt{c_{d+1}}\phi_1, \cdots, \sqrt{c_{m}}\phi_1)}}E(\vec{u}) < M_1.$$ Note that $\mathbb{M}_{k+1,d} \cap \mathbb{S}_{k,d}^\bot \neq \emptyset$, which consists of one point denoted by $\sbr{\sqrt{c_1}\phi_{k+1},\ldots,\sqrt{c_d}\phi_{k+1}}$.
		For any $h \in C( \mathbb{M}_{k+1,d}\times \lbr{(\sqrt{c_{d+1}}\phi_1, \cdots, \sqrt{c_{m}}\phi_1)}  ,S_{\vec{c}})$  satisfying $h|_{\partial \mathbb{M}_{k+1,d}\times \lbr{(\sqrt{c_{d+1}}\phi_1, \cdots, \sqrt{c_{m}}\phi_1)}} = \textbf{id}$, we claim that
		\begin{equation}\label{sq3}
			T_d \circ h(\mathbb{M}_{k+1,d} \times \lbr{(\sqrt{c_{d+1}}\phi_1, \cdots, \sqrt{c_{m}}\phi_1)}) \cap \mathbb{S}_{k,d}^\bot \not= \emptyset.
		\end{equation}  In fact, observe that  $\mathbb{M}_{k+1}^i \cong \mathbb{P}_k^i  \subset H_k$, where
		\begin{equation}
			\mathbb{P}_k^i:=\lbr{ u_i = \sum_{j=1}^{k}t_{i,j}\sqrt{c_i}\phi_j: \sum_{j=1}^{k}t_{i,j}^2 \leq 1}, \quad i \in \lbr{1,\ldots, d} .
		\end{equation} Hence, $\mathbb{M}_{k+1,d} \cong \mathbb{P}_k^1\times \cdots \times  \mathbb{P}_{k}^d $, and the topology degree $\deg(P_k \circ \tilde h,\mathbb{M}_{k+1,d} ,\vec{0})$ and $\deg(P_k|_{\mathbb{M}_{k+1,d} },\mathbb{M}_{k+1,d} ,\vec{0})$ are well-defined,  where $P_k :  H_0^1(\Omega,\R^d)\to \HR_{k,d}$ is the orthogonal projection, and  $\tilde{h}:  \mathbb{M}_{k+1,d} \to S_{\vec{c},d}$ is defined by
		\begin{equation}
			\tilde{h}(\vec{u}_d) =T_d \circ h(\vec{u}_d \times \lbr{(\sqrt{c_{d+1}}\phi_1, \cdots, \sqrt{c_{m}}\phi_1)}).
		\end{equation}  Let  $\vec u_d=\sbr{u_1,\ldots,u_d}\in \mathbb{M}_{k+1,d}$ and
		\begin{align*}
			g (t,\vec u_d) := (1-t)P_k(\tilde h(\vec{u}_d)) + t P_k|_{\mathbb{M}_{k+1,d}} \vec{u}_d, \quad \forall t \in [0,1]  .
		\end{align*}
		For any $\vec u_d\in  \partial\mathbb{M}_{k+1,d}  $ and $t \in [0,1]$, we have $g(t,\vec u_d) = P_k \vec u _d\neq \vec{0}$. Hence, we have
		\begin{align*}
			\deg(P_k \circ \tilde h,\mathbb{M}_{k+1,d},\vec{0}) = \deg(P_k|_{\mathbb{M}_{k+1,d}},\mathbb{M}_{k+1,d},\vec{0}) = 1,
		\end{align*}
		and so, there exists a $\vec{w}_d \in \mathbb{M}_{k+1,d}$ such that $P_k( \tilde h(\vec{w}_d)) = \vec{0}$, i.e., $ \tilde h(\vec{w}_d) \in  {\HR_{k,d}^\bot} \cap S_{\vec{c},d}=\mathbb{S}_{k,d}^\bot$. This yields that  relation \eqref{sq3} holds, and we complete the proof.
	\end{proof}
	
	\begin{remark} \label{rmktildeck semi-nodal}
		{\rm Let $k$ be the positive integer such that  $\Lambda_k<\Lambda_{k+1}$. Similar to Remark \ref{rmktildeck}, there exists $\tilde c_{k,d} > 0$ such that for any $\vec{c} = (c_1,\cdots,c_m)$ with $c_i > 0$, $i = 1,\cdots,m$, and $\frac{\max_{1 \le i \le m} c_i^2}{\min_{1 \le i \le m} c_i} < \tilde c_{k,d}$, we can always take $\rho > 0$ such that \eqref{defi of rho}, \eqref{eqci5.1} and \eqref{eqci5.2} hold true.}
	\end{remark}
	
	\begin{proposition} \label{propwelldefined semi-nodal}
		For any $\rho> 0$  and  $\vec{c} = (c_1,\cdots,c_m) \in \sbr{\R^+}^m$ satisfying \eqref{defi of rho},    \eqref{eqci5.1} and  \eqref{eqci5.2},  then the value $m_{\vec{c},\rho,k,d}^\delta$ is well defined and finite for $0 < \delta < \delta_{0,d}$ where $\delta_{0,d}$ is given in Lemma \ref{lem>0 semi-nodal}. Moreover,
		\begin{align} \label{eqestiofm semi-nodal}
			\sup_{\partial \mathbb{M}_{k+1,d} \times \lbr{(\sqrt{c_{d+1}}\phi_1, \cdots, \sqrt{c_{m}}\phi_1)}}E(\vec{u}) & <  {\inf_{  	\B_\rho \cap \sbr{\mathbb{S}_{k,d}^\bot \times S_{\vec{c},m-d}}}E(\vec{u})} \nonumber \\
			& \le m_{\vec{c},\rho,k,d}^\delta \le \sup_{\mathbb{M}_{k+1,d} \times \lbr{(\sqrt{c_{d+1}}\phi_1, \cdots, \sqrt{c_{m}}\phi_1)}}E(\vec{u}) < M_1.
		\end{align}
	\end{proposition}
	
	\begin{proof}
		By Lemma \ref{lemma 5.2} and Lemma \ref{lemnemp semi-nodal}, arguing as the proof of Proposition \ref{propwelldefined} we can complete the proof.
	\end{proof}
	
	\section{General framework to prove some set is invariant under a descending flow on a product manifold} \label{Sect2}
	
	This section is devoted to generalize the following Proposition \ref{propbm} in a general setting where the mass constrained problem for Schr\"odinger systems is included. A similar extension in a setting where a single equation is considered has been made in \cite{JS} and it has been used in \cite{SZ} for Br\'ezis-Nirenberg problem and also used in \cite{JSZ}. Here we generalize the arguments in \cite[Section 4]{JS}. Readers can consult \cite[Theorem 4.1]{Dei} for the proof of Proposition \ref{propbm}. Similar results can be found in \cite[Theorem 1]{Bre} for finite-dimensional Euclidean space cases and in \cite[Theorem 3]{Mar1} or \cite{Mar} for Banach space cases.
	\medskip
	
	\begin{proposition} \label{propbm}
		Let $X$ be a Banach space, $B \subset X$, $u \in B$, $U_X$ be a neighborhood of $u$ in $X$ and $B \cap \overline{U_X}$ be closed. If $\widetilde V: \overline{U_X} \to X$ is a Lipschitz mapping and
		\begin{align} \label{eqd=0banachcase}
			\lim_{s \searrow 0}s^{-1}{\rm dist}(w + s\widetilde V(w),B) = 0, \quad \forall w \in \overline{U_X} \cap B,
		\end{align}
		then there exist $r = r(u) > 0$ and $\tilde\eta(t,u)$ satisfying
		\begin{align*}
			\left\{
			\begin{aligned}
				& \frac{\partial}{\partial t}\tilde\eta(t,u) = \widetilde V(\tilde\eta(t,u)), \quad \forall t \in [0,r), \\
				& \tilde\eta(0,u) = u, \quad \tilde\eta(t,u) \in B.
			\end{aligned}
			\right.
		\end{align*}
	\end{proposition}
	
	Before proceeding, we introduce our general setting in which the mass constraint is included. Let $(E,\langle \cdot,\cdot\rangle)$ and $(H,( \cdot,\cdot))$ be two \emph{infinite-dimensional} Hilbert spaces and assume that
	\begin{align*}
		E \hookrightarrow H \hookrightarrow E',
	\end{align*}
	with continuous injections. For simplicity, we assume that the continuous injection $E \hookrightarrow H$ has norm at most $1$ and identify $E$ with its image in $H$. We also introduce
	\begin{align*}
		\left\{
		\begin{aligned}
			\|u\|^2 & = \langle u,u \rangle, \\
			|u|^2 \ & = (u,u), \\
		\end{aligned}
		\right.
		\quad \quad u \in E,
	\end{align*}
	and, for $\vec{c} = (c_1,c_2,\cdots,c_m)$ with $c_j > 0$, $j = 1,2,\cdots,m$, we define
	\begin{align*}
		S_{\vec{c}} := \bigg\{\vec{u} = (u_1,u_2,\cdots,u_m) \in \underbrace{E\times E\times \cdots \times E}_{m}: |u_j|^2 = c_j, j = 1,2,\cdots,m \bigg\}.
	\end{align*}
	For $c > 0$ we set
	\begin{align*}
		S_c := \bigl\{u \in E: |u|^2 = c\bigr\}.
	\end{align*}
	Clearly, $S_{\vec{c}} \cong S_{c_1} \times S_{c_2} \times \cdots \times S_{c_m}$.
	
	\medskip
	Our main result reads as follows.
	\begin{proposition} \label{propbmmanifold}
		Let $\vec{u} \in B_{\vec{c}} \subset S_{\vec{c}}$, $U$ be a neighborhood of $\vec{u}$ in $S_{\vec{c}}$ and $B_{\vec{c}} \cap \overline{U}$ closed in $S_{\vec{c}}$. Let $V$ be a Lipschitz mapping on $\overline{U}$ with
		$$
		V(\vec{w}) \in T_{\vec{w}}S_{\vec{c}} \cong T_{w_1} S_{c_1} \times T_{w_2} S_{c_2} \times \cdots \times T_{w_m} S_{c_m}
		$$
		for any $\vec{w} = (w_1,w_2,\cdots,w_m) \in \overline{U}$, where $T_{\vec{w}}S_{\vec{c}}$ is the tangent space of $S_{\vec{c}}$ at $\vec{w}$ defined as
		\begin{align*}
			T_{\vec{w}}S_{\vec{c}} := \bigg\{\vec{\phi} = (\phi_1,\phi_2,\cdots,\phi_m) \in \underbrace{E\times E\times \cdots \times E}_{m}: (\phi_1,w_1) = (\phi_2,w_2) = \cdots = (\phi_m,w_m) = 0\bigg\},
		\end{align*}
		which satisfies
		\begin{align} \label{eqd=0}
			\lim_{s \searrow 0}s^{-1}{\rm dist}(\vec{w} + sV(\vec{w}),B_{\vec{c}}) = 0, \quad \forall \vec{w} \in \overline{U} \cap B_{\vec{c}}.
		\end{align}
		Then there exist $r = r(\vec{u}) > 0$ and $\eta(t,\vec{u})$ satisfying
		\begin{align*}
			\left\{
			\begin{aligned}
				& \frac{\partial}{\partial t}\eta(t,\vec{u}) = V(\eta(t,\vec{u})), \quad \forall t \in [0,r), \\
				&\eta(0,\vec{u}) = \vec{u}, \quad \eta(t,\vec{u}) \in B_{\vec{c}}.
			\end{aligned}
			\right.
		\end{align*}
	\end{proposition}
	
	\begin{proof}
		Let
		\begin{align*}
			B & := \bigl\{(\theta_1 w_1,\theta_2 w_2,\cdots,\theta_m w_m): \theta_j \in [1/2,3/2], ~j = 1,2,\cdots,m, ~(w_1,w_2,\cdots,w_m) \in B_{\vec{c}}\bigr\}, \\
			U_E & := \bigl\{(\theta_1 w_1,\theta_2 w_2,\cdots,\theta_m w_m): \theta_j \in (1/2,3/2), ~j = 1,2,\cdots,m, ~(w_1,w_2,\cdots,w_m) \in U\bigr\}.
		\end{align*}
		Then we can check that $$\vec{u} \in B \subset \underbrace{E\times E\times \cdots \times E}_{m},$$
		and can check that $U_E$ is a neighborhood of $u$ in $E\times E\times \cdots \times E$ and $B \cap \overline{U_E}$ is closed in $E$. For any $\vec{v} = (v_1,v_2,\cdots,v_m) \in \overline{U_E}$, we define
		\begin{align} \label{def:tildeV}
			\widetilde V(\vec{v}) := \left( \frac{|v_1|}{\sqrt {c_1}}V_1(\vec{v}_*), \frac{|v_2|}{\sqrt {c_2}}V_2(\vec{v}_*), \cdots, \frac{|v_m|}{\sqrt {c_m}}V_m(\vec{v}_*)\right)  \in \underbrace{E\times E\times \cdots \times E}_{m},
		\end{align}
		where
		\begin{align*}
			\vec{v}_* = \left(\frac{\sqrt {c_1}}{|v_1|} v_1, \frac{\sqrt {c_2}}{|v_2|} v_2, \cdots, \frac{\sqrt {c_m}}{|v_m|} v_m \right) \in S_{\vec{c}},
		\end{align*}
		and $V(\vec{w}) = (V_1(\vec{w}),V_2(\vec{w}),\cdots,V_m(\vec{w}))$ for $\vec{w} \in S_{\vec{c}}$. Next we check that $\widetilde{V}$ is a Lipschitz mapping. Let $C > 0$ be a constant which may change line by line in this proof. Since $V$ is Lipschitz, for any $\vec{w} = (w_1,w_2,\cdots,w_m), \vec{w}' = (w_1',w_2',\cdots,w_m') \in \overline{U}$ we have
		\begin{align*}
			\|V_j(\vec{w}')-V_j(\vec{w})\| \leq C\|\vec{w}'-\vec{w}\|_{E\times E\times \cdots \times E}, \quad j = 1,2,\cdots,m.
		\end{align*}
		For $$(\theta_1 w_1,\theta_2 w_2,\cdots,\theta_m w_m) \in \overline{U_E}, \quad (\theta_1' w_1',\theta_2' w_2',\cdots,\theta_m' w_m') \in \overline{U_E}$$ with $\vec{w} = (w_1,w_2,\cdots,w_m) \in \overline{U}$ and $\vec{w}' = (w_1',w_2',\cdots,w_m') \in \overline{U}$, we have
		\begin{align*}
			& \widetilde V(\theta_1 w_1,\theta_2 w_2,\cdots,\theta_m w_m) = \left(\theta_1V_1(\vec{w}),\theta_2V_2(\vec{w}), \cdots, \theta_mV_m(\vec{w}) \right), \\
			& \widetilde V(\theta_1' w_1',\theta_2' w_2',\cdots,\theta_m' w_m') = (\theta_1'V_1(\vec{w}'),\theta_2'V_2(\vec{w}'), \cdots, \theta_m'V_m(\vec{w}') ).
		\end{align*}
		Computing directly, we get
		\begin{align*}	
			\|\theta_j'V_j(\vec{w}') - \theta_j V_j(\vec{w})\| & \leq \theta_j'\|V_j(\vec{w}')-V_j(\vec{w})\| + |\theta_j'-\theta_j|\|V_j(\vec{w})\| \\
			& \leq \ C\theta_j'\|\vec{w}'-\vec{w}\|_{E\times E\times \cdots \times E} + |\theta_j'-\theta_j |\|V_j(\vec{w})\|, \quad j = 1,2,\cdots,m.
		\end{align*}
		Note that $\theta_j' = |\theta_j'w_j'|/\sqrt{c_j}$ and $\theta_j  = |\theta_j w_j|/\sqrt{c_j}$, $j = 1,2,\cdots,m$. We have
		\begin{align*}
			|\theta_j'-\theta_j| \leq c_j^{-\frac12}|\theta_j'w_j'-\theta_j w_j|.
		\end{align*}
		Then it holds
		\begin{align*}
			\|\vec{w}'-\vec{w}\|_{E\times E\times \cdots \times E} & \leq C\sum_{j=1}^m\|w_j'-w_j\| \\
			& \leq C\sum_{j=1}^m\|\theta_j'w_j'-\theta_j'w_j\| \\
			& \leq C\sum_{j=1}^m\left( \|\theta_j'w_j'-\theta_jw_j\| + \|w_j\||\theta_j'-\theta_j|\right) \\
			& \leq C\sum_{j=1}^m\left( \|\theta_j'w_j'-\theta_jw_j\| + c_j^{-\frac12}\|w_j\||\theta_j'w_j'-\theta_j w_j|\right) \\
			& \leq C\sum_{j=1}^m\|\theta_j'w_j'-\theta_jw_j\| \\
			& \leq C\|(\theta_1' w_1',\theta_2' w_2',\cdots,\theta_m' w_m') - (\theta_1 w_1,\theta_2 w_2,\cdots,\theta_m w_m)\|_{E\times E\times \cdots \times E}.
		\end{align*}
		It follows that
		\begin{align*}
			& \ \|\theta_j'V_j(\vec{w}') - \theta_j V_j(\vec{w})\| \\
			\leq & \ C\|(\theta_1' w_1',\theta_2' w_2',\cdots,\theta_m' w_m') - (\theta_1 w_1,\theta_2 w_2,\cdots,\theta_m w_m)\|_{E\times E\times \cdots \times E} + |\theta_j'-\theta_j|\|V_j(\vec{w})\| \\
			\leq & \ C\|(\theta_1' w_1',\theta_2' w_2',\cdots,\theta_m' w_m') - (\theta_1 w_1,\theta_2 w_2,\cdots,\theta_m w_m)\|_{E\times E\times \cdots \times E} + C|\theta_j'w_j'-\theta_j w_j| \\
			\leq & \ C\|(\theta_1' w_1',\theta_2' w_2',\cdots,\theta_m' w_m') - (\theta_1 w_1,\theta_2 w_2,\cdots,\theta_m w_m)\|_{E\times E\times \cdots \times E} + C\|\theta_j'w_j'-\theta_j w_j\| \\
			\leq & \ C\|(\theta_1' w_1',\theta_2' w_2',\cdots,\theta_m' w_m') - (\theta_1 w_1,\theta_2 w_2,\cdots,\theta_m w_m)\|_{E\times E\times \cdots \times E}, \quad j = 1,2,\cdots,m,
		\end{align*}
		which implies that
		\begin{align*}
			& \ \|\widetilde V(\theta_1' w_1',\theta_2' w_2',\cdots,\theta_m' w_m') - \widetilde V(\theta_1 w_1,\theta_2 w_2,\cdots,\theta_m w_m)\|_{E\times E\times \cdots \times E} \\
			\leq & \ C\|(\theta_1' w_1',\theta_2' w_2',\cdots,\theta_m' w_m') - (\theta_1 w_1,\theta_2 w_2,\cdots,\theta_m w_m)\|_{E\times E\times \cdots \times E},
		\end{align*}
		that is, $\widetilde{V}$ is Lipschitz on $\overline{U_E}$.
		
		Now, we apply Proposition \ref{propbm} with $X = E\times E\times \cdots \times E$, $U_X = U_E$, and $\widetilde V$ given by \eqref{def:tildeV}. It needs to check \eqref{eqd=0banachcase}. For any $(\theta_1 w_1,\theta_2 w_2,\cdots,\theta_m w_m) \in \overline{U_E} \cap B$ with $\vec{w} = (w_1,w_2,\cdots, w_m) \in \overline{U} \cap B_{\vec{c}}$, using \eqref{eqd=0} we get the existence of $\vec{v}_s = (v_{s,1},v_{s,2},\cdots,v_{s,m}) \in B_{\vec{c}}$ such that
		\begin{align*}
			\lim_{s \searrow 0}s^{-1}\|\vec{w} + sV(\vec{w})-\vec{v}_s\|_{E\times E\times \cdots \times E} = 0.
		\end{align*}
		Hence,
		\begin{align*}
			& \lim_{s \searrow 0}s^{-1}\text{dist}((\theta_1 w_1,\theta_2 w_2,\cdots,\theta_m w_m) + s\widetilde V(\theta_1 w_1,\theta_2 w_2,\cdots,\theta_m w_m),B) \leq \\
			& \lim_{s \searrow 0}s^{-1}\|(\theta_1 w_1,\theta_2 w_2,\cdots,\theta_m w_m) + s(\theta_1 V_1(\vec{w}),\theta_2 V_2(\vec{w}),\cdots,\theta_m V_m(\vec{w})) - (\theta_1 v_{s,1},\theta_2 v_{s,2},\cdots,\theta_m v_{s,m})\|_{E\times E\times \cdots \times E} \\
			& \leq \max\{\theta_1,\theta_2,\cdots,\theta_m\}\lim_{s \searrow 0}s^{-1}\|\vec{w} + sV(\vec{w})-\vec{v}_s\|_{E\times E\times \cdots \times E} = 0.
		\end{align*}
		Then, by Proposition \ref{propbm} we obtain the existence of $r = r(\vec{u}) > 0$ and $\tilde\eta(t,\vec{u})$ such that
		\begin{align*}
			\left\{
			\begin{aligned}
				& \frac{\partial}{\partial t}\tilde\eta(t,\vec{u}) = \widetilde V(\tilde\eta(t,\vec{u})), \quad \forall t \in [0,r), \\
				& \tilde\eta(0,\vec{u}) = \vec{u} \in B_{\vec{c}}, \quad \tilde\eta(t,\vec{u}) \in B.
			\end{aligned}
			\right.
		\end{align*}
		Let $\eta(\cdot,\vec{u}): [0,T) \to S_{\vec{c}}$ be the local solution of
		\begin{align*}
			\left\{
			\begin{aligned}
				& \frac{\partial}{\partial t}\eta(t,\vec{u}) = V(\eta(t,\vec{u})), \\
				& \eta(0,\vec{u}) = \vec{u}.
			\end{aligned}
			\right.	
		\end{align*}
		Clearly $\eta(\cdot,\vec{u})$ can be viewed as a function mapping $[0,T)$ to $E\times E \times \cdots \times E$ since $S_{\vec{c}} \subset E\times E \times \cdots \times E$. It is easy to see that $\eta(\cdot,\vec{u}): [0,T) \to E\times E \times \cdots \times E$ satisfies
		\begin{align} \label{eqetaonE}
			\left\{
			\begin{aligned}
				& \frac{\partial}{\partial t}\eta(t,\vec{u}) = \widetilde V(\eta(t,\vec{u})), \\
				& \eta(0,\vec{u}) = \vec{u},
			\end{aligned}
			\right.	
		\end{align}
		since $\widetilde V(\vec{w}) = V(\vec{w})$ for any $\vec{w} \in S_{\vec{c}}$. By the local existence and uniqueness of the solution for equation \eqref{eqetaonE} and by decreasing $r >0$ if necessary, we know that $\eta(t,\vec{u}) \equiv \tilde \eta(t,\vec{u})$ for $t \in [0,r)$. Hence, $\eta(t,\vec{u}) \in B \cap S_{\vec{c}} = B_{\vec{c}}$ for all $t \in [0,r)$ and this completes the proof.
	\end{proof}
	
	\begin{corollary} \label{corbm}
		Let $B_{\vec{c}} \subset S_{\vec{c}}$ be closed in $S_{\vec{c}}$.
		Suppose that $V$ is a locally Lipschitz mapping with
		$$
		V(\vec{w}) \in T_{\vec{w}}S_{\vec{c}} \cong T_{w_1}S_{c_1} \times T_{w_2}S_{c_2} \times \cdots \times T_{w_m}S_{c_m}
		$$
		for any $\vec{w} = (w_1,w_2,\cdots,w_m) \in S_{\vec{c}}$ and define $\eta(t,\cdot): S_{\vec{c}} \to S_{\vec{c}}$ by
		\begin{align*}
			\left\{
			\begin{aligned}
				& \frac{\partial}{\partial t}\eta(t,\vec{u}) = V(\eta(t,\vec{u})), \\
				& \eta(0,\vec{u}) = \vec{u}.
			\end{aligned}
			\right.
		\end{align*}
		Let $T_{\vec{u}} > 0$ be the maximal time such that $\eta(\cdot,\vec{u})$ exists for $\vec{u} \in S_{\vec{c}}$ and $B^* \subset S_{\vec{c}}$ be an open subset such that $\eta(t,\vec{u}) \in B^*$ for any $\vec{u} \in B^*$ and all $0 < t < T_{\vec{u}}$. Then, if
		\begin{align} \label{eqbm}
			\lim_{s \searrow 0}s^{-1}\text{dist}(\vec{u} + sV(\vec{u}),B_{\vec{c}}) = 0, \quad \forall \vec{u} \in B_{\vec{c}} \cap B^*,
		\end{align}
		it holds that $\eta(t,\vec{u}) \in B_{\vec{c}} \cap B^*$ for any $\vec{u} \in B_{\vec{c}} \cap B^*$ and all $0 < t < T_{\vec{u}}$.
	\end{corollary}
	
	\begin{proof}
		For any $\vec{u} \in B_{\vec{c}} \cap B^*$, since $B^*$ is open and $V$ is a locally Lipschitz mapping, we can choose a neighborhood $U$ of $\vec{u}$ in $S_{\vec{c}}$ such that $\overline{U} \subset B^*$ and $V$ is Lipschitz on $\overline{U}$. Also, it is clear that $B_{\vec{c}} \cap \overline{U}$ is closed since $B_{\vec{c}}$ is closed in $S_{\vec{c}}$ and that \eqref{eqbm} implies
		\begin{align*}
			\lim_{s \searrow 0}s^{-1}\text{dist}(\vec{w} + sV(\vec{w}),B_{\vec{c}}) = 0, \quad \forall \vec{w} \in \overline{U} \cap B_{\vec{c}}.
		\end{align*}
		Proposition \ref{propbmmanifold} then yields that $\eta(t,\vec{u}) \in B_{\vec{c}}$ in $t \in [0,r)$ for some $0 < r \leq T_{\vec{u}}$. By the assumption on $B^*$, we know $\eta(t,\vec{u}) \in B^*$ for all $0 < t < T_{\vec{u}}$. Let
		\begin{align*}
			r^* := \sup\bigl\{T: 0 < T < T_{\vec{u}}, \eta(t,\vec{u}) \in B_{\vec{c}} \text{ in } t \in [0,T)\bigr\}.
		\end{align*}
		Obviously, $r^ * \geq r$. We claim that $r^* = T_{\vec{u}}$. By contradiction, we suppose $r^* < T_{\vec{u}}$. According to the definition of $r^*$, we know that $\eta(t,\vec{u}) \in B_{\vec{c}}$ for any $t \in [0,r^*)$. Using the closure of $B_{\vec{c}}$ we deduce that $\eta(r^*,\vec{u}) \in  B_{\vec{c}}$.  Since $\eta(r^*,\vec{u}) \in  B^*$, that is, $\eta(r^*,\vec{u}) \in  B_{\vec{c}} \cap B^*$, using Proposition \ref{propbmmanifold} again we have $\eta(t,\vec{u}) \in B_{\vec{c}}$ in $t \in [0,r^* + \bar r)$ for some $0 < \bar r \leq T_u-r^*$. This contradicts the definition of $r^*$ and we complete the proof.
	\end{proof}

	In the next lemma, we give a condition to check that hypothesis \eqref{eqd=0} (or  hypothesis  \eqref{eqbm}) holds.
	
	\begin{lemma} \label{lemconvex}
		Let $B_{\vec{c}} = \tilde B \cap S_{\vec{c}}$ where $\tilde B \subset E\times E\times \cdots \times E$ is closed, convex in $E\times E\times \cdots \times E$ and satisfies the assumption:
		\begin{align} \label{eqconecondition}
			(k_1w_1, k_2w_2, \cdots, k_mw_m) \in \tilde B, \quad \forall (k_1,k_2,\cdots,k_m) \in (0,1]^m, ~(w_1,w_2,\cdots,w_m) \in \tilde B.
		\end{align}
		For $\vec{u} \in B_{\vec{c}}$, if $V(\vec{u})$ has the form of $G(\vec{u})-\vec{u}$ and $G(\vec{u}) \in \tilde B$, then
		\begin{align} \label{eqd=0atapoint}
			\lim_{s \searrow 0}s^{-1}{\rm dist}(\vec{u} + sV(\vec{u}),B_{\vec{c}}) = 0.
		\end{align}
	\end{lemma}
	
	\begin{proof}
		Let $\vec{u}_s = \vec{u} + sV(\vec{u}) = (u_{s,1},u_{s,2},\cdots, u_{s,m})$ where $\vec{u} = (u_1,u_2,\cdots,u_m) \in B_{\vec{c}} \subset S_{\vec{c}}$. Since $V(\vec{u}) \in T_{\vec{u}}S_{\vec{c}}$, we have
		\begin{align*}
			|u_{s,j}|^2 & = (u_j+sV_j(\vec{u}),u_j+sV_j(\vec{u})) \\
			& = (u_j,u_j) + 2s(u_j,V_j(\vec{u})) + s^2(V_j(\vec{u}),V_j(\vec{u})) \\
			& = c_j + s^2|V_j(\vec{u})|^2, \quad j = 1,2,\cdots,m.
		\end{align*}
		Note that $\vec{u}_s = sG(\vec{u}) + (1-s)\vec{u} \in \tilde B$ by the convexity of $\tilde B$. Then, since $|u_{s,j}| \geq \sqrt{c_j}$, using \eqref{eqconecondition} we know $$\left( \frac{\sqrt{c_1}}{|u_{s,1}|}u_{s,1}, \frac{\sqrt{c_2}}{|u_{s,2}|}u_{s,2}, \cdots, \frac{\sqrt{c_m}}{|u_{s,m}|}u_{s,m}\right) \in \tilde B \cap S_{\vec{c}} = B_{\vec{c}}.$$
		Hence,
		\begin{align*}
			\lim_{s \searrow 0}s^{-1}dist(\vec{u} + sV(\vec{u}),B_{\vec{c}}) & \leq \lim_{s \searrow 0}s^{-1}\left\|\vec{u}_s -  \left( \frac{\sqrt{c_1}}{|u_{s,1}|}u_{s,1}, \frac{\sqrt{c_2}}{|u_{s,2}|}u_{s,2}, \cdots, \frac{\sqrt{c_m}}{|u_{s,m}|}u_{s,m}\right)\right\|_{E\times E\times \cdots \times E} \\
			& = \lim_{s \searrow 0}\sqrt{\sum_{j=1}^m\left( \frac{s|V_j(\vec{u})|^2}{2c_j}\right) ^2\|u_{s,j}\|^2} = 0,
		\end{align*}
		which completes the proof.
	\end{proof}

	\section{Gradient, pseudogradient, descending flow, and invariant sets} \label{secdescendingflow}
	
	In this section, we will construct a descending flow with gradient field or a pseudogradient field. Using the abstract result established in Section \ref{Sect2}, we prove that $\overline{(\pm \p_i)_\delta} \cap \B_\rho^{M_1}$, as well as $\pm \p_i \cap \B_\rho^{M_1}$, $i = 1,\cdots,m$, are invariant under this descending flow. The descending flow technique and flow invariance arguments are crucial to prove Theorem \ref{thm1.1} and \ref{thm1.2} in the coming sections, precisely, to obtain Palais-Smale sequences at minimax levels constructed in Section \ref{seclink} with location information. The location information will be used to get the sign change properties of solutions.
	
	\vskip0.1in
	
	When $\mu_i > 0$, $\beta_{ij} > 0$, $\forall i \neq j$, we can use the gradient to construct the descending flow; while if $\mu_i < 0$ for some $i$ or $\beta_{ij} <0$ for some $i \neq j$, the gradient flow does not satisfy our requirements due to the sign of the term $\mu_i u_i^3$ or $\beta_{ij} u_j^2u_i$. In the latter case, inspired by \cite{TT}, we aim to find a pseudogradient for $E$ over $S_{\vec{c}}$ for which the sets $\B_\rho^{M_1} \cap \overline{(\pm\PR_i)_\delta}$ are positively invariant for the associated flow. Thanks to Proposition \ref{propbmmanifold}, Corollary \ref{corbm} and Lemma \ref{lemconvex}, we can work on $S_{\vec{c}}$ directly rather than working on its neighborhood which is what was done by the authors in \cite{TT}. More precisely, we are devoted to search for a pseudogradient being of the type $Id - G$, where $Id$ is the identity in $H_0^1(\Omega,\R^m)$ and $G$ is an operator such that $\vec{u} - G(\vec{u}) \in T_{\vec{u}}S_{\vec{c}}$ for $\vec{u} \in S_{\vec{c}}$ and $G(\vec{u}) \in (\pm \p_i)_{\delta/2}$ for all $\vec{u} \in \overline{(\pm\p_i)_\delta} \cap \B_\rho$ with $\delta > 0$ small enough under certain conditions on $\vec{c}$.
	
	\begin{proposition} \label{propuniquewi}
		Assume that \eqref{assumption} holds. Given $\vec{u} \in S_{\vec{c}}$ and $i \in \{1,2,\cdots,m\}$, if $\mu_i > 0$ there exists a unique solution $(w_i,\lambda_{i}) \in H_0^1(\Omega)\times\mathbb{R}$ of the problem
		\begin{align} \label{eqsolutionw}
			\left\{
			\begin{aligned}
				& - \Delta w_i - \sum_{(i,j) \in J^-}\beta_{ij}u_j^2 w_i = \mu_iu_i^3 + \sum_{(i,j) \in J^+}\beta_{ij}u_j^2u_i - \la_iu_i \quad \text{in } \Omega, \\
				& \int_\Omega u_iw_i dx = c_i,
			\end{aligned}
			\right.
		\end{align}
		and if $\mu_i < 0$ there exists a unique solution $(w_i,\lambda_{i}) \in H_0^1(\Omega)\times\mathbb{R}$ of the problem
		\begin{align} \label{eqsolutionw<0}
			\left\{
			\begin{aligned}
				& - \Delta w_i - \mu_iu_i^2 w_i - \sum_{(i,j) \in J^-}\beta_{ij}u_j^2 w_i = \sum_{(i,j) \in J^+}\beta_{ij}u_j^2u_i - \la_iu_i \quad \text{in } \Omega, \\
				& \int_\Omega u_iw_i dx = c_i,
			\end{aligned}
			\right.
		\end{align}
		where $\la_i \in \R$ is an unknown Lagrange multiplier.
	\end{proposition}
	
	\begin{proof}
		We just prove the case of $\mu_i > 0$ since the proof when $\mu_i < 0$ is similar.
		
		\emph{Existence:} Let us consider the following minimization problem
		\begin{align*}
				e := \inf\left\{\int_\Omega\bigg(\frac12|\nabla w|^2 - \frac12 \sum_{(i,j)\in J^-}\beta_{ij}u_j^2w^2-\mu_iu_i^3w - \sum_{(i,j)\in J^+}\beta_{ij}u_j^2u_i w \bigg) dx: w \in H_0^1(\Omega), \int_\Omega u_iw dx = c_i\right\}.
		\end{align*}
		Since
		\begin{align} \label{eqnonterm}
			\int_\Omega u_i^3wdx \leq \left( \int_\Omega u_i^4dx\right)^{\frac34}\left( \int_\Omega w^4dx\right)^{\frac14} \leq C\left(\int_\Omega |\nabla w|^2dx \right)^{\frac12}
		\end{align}
		and
		\begin{align} \label{eqnontermbeta}
			\int_\Omega u_j^2u_iwdx \leq \left( \int_\Omega u_j^4dx\right)^{\frac12}\left( \int_\Omega u_i^4dx\right)^{\frac14}\left( \int_\Omega w^4dx\right)^{\frac14} \leq C\left(\int_\Omega |\nabla w|^2dx \right)^{\frac12},
		\end{align}
		where $C$ is a positive constant and we have used the H\"older inequality and Sobolev inequality in \eqref{eqnonterm} and \eqref{eqnontermbeta}, we get $e > -\infty$.
		
		Take a minimizing sequence $\{w_n\}_n$ such that
		\begin{align*}
			\int_\Omega\bigg(\frac12|\nabla w_n|^2 - \frac12 \sum_{(i,j)\in J^-}\beta_{ij}u_j^2w_n^2-\mu_iu_i^3w_n - \sum_{(i,j)\in J^+}\beta_{ij}u_j^2u_i w_n \bigg) dx \to e \quad \text{and} \quad \int_\Omega u_iw_n dx = c_i.
		\end{align*}
		Using \eqref{eqnonterm} and \eqref{eqnontermbeta} again, we obtain that $\{w_n\}_n$ is bounded in $H_0^1(\Omega)$. Up to a subsequence, we assume that
		\begin{align*}
			& w_n \rightharpoonup \bar w \quad \text{weakly in } H_0^1(\Omega), \\
			& w_n \to \bar w \quad \text{strongly in } L^2(\Omega),
		\end{align*}
		for some $\bar w \in H_0^1(\Omega)$. Therefore $\int_\Omega u_i \bar w dx = c_i$.
		
		When $N = 3$, we further assume that $w_n \to \bar w$ strongly in $L^4(\Omega)$. Then it follows
		\begin{align*}
			e & \leq \int_\Omega\bigg(\frac12|\nabla \bar w|^2 - \frac12 \sum_{(i,j)\in J^-}\beta_{ij}u_j^2 \bar w^2-\mu_iu_i^3 \bar w - \sum_{(i,j)\in J^+}\beta_{ij}u_j^2u_i \bar w \bigg) dx \\
			& \leq \liminf_{n \to \infty}\int_\Omega\bigg(\frac12|\nabla w_n|^2 - \frac12 \sum_{(i,j)\in J^-}\beta_{ij}u_j^2w_n^2-\mu_iu_i^3w_n - \sum_{(i,j)\in J^+}\beta_{ij}u_j^2u_i w_n \bigg) dx = e,
		\end{align*}
		showing that $e$ is attained by $\bar w$. It is standard that $\bar w$ solves \eqref{eqsolutionw} for some $\la_i \in \R$.
		
		When $N = 4$, we further assume that $w_n \to \bar w$ strongly in $L^1(\Omega)$. Then, if $\vec{u} \in S_{\vec{c}} \cap C_0^\infty(\Omega,\R^m)$, by similar arguments to $N = 3$ we obtain that $e$ is attained by a minimizer $w_i$ which solves \eqref{eqsolutionw} for some $\la_i \in \R$. For general $\vec{u} \in S_{\vec{c}}$, we take a sequence $\{\vec{u}_n\}_n \subset S_{\vec{c}} \cap C_0^\infty(\Omega,\R^m)$ such that $\vec{u}_n \to \vec{u}$ strongly in $H_0^1(\Omega,\R^m)$, where we denote $\vec{u}_n = (u_{n,1}, u_{n,2}, \cdots, u_{n,m})$. To avoid confusion, in the rest part of the proof we use notation $e_i(\vec{u})$ instead of $e$. Let $w_{i,n}$ be a minimizer of $e_i(\vec{u}_n)$ and $\la_{i,n}$ be the corresponding Lagrange multiplier. For simplicity, we denote
		\begin{align*}
			I(w,\vec{u}) = \int_\Omega\bigg(\frac12|\nabla w|^2 - \frac12 \sum_{(i,j)\in J^-}\beta_{ij}u_j^2 w^2-\mu_iu_i^3 w - \sum_{(i,j)\in J^+}\beta_{ij}u_j^2u_i w \bigg) dx.
		\end{align*}
		Next, we aim to prove that $\{I(w_{i,n},\vec{u}_n)\}_n$ is a bounded sequence. On the one hand, using \eqref{eqnonterm} and \eqref{eqnontermbeta} we obtain
		\begin{align} \label{eqlowboundofI}
			I(w_{i,n},\vec{u}_n) \geq \frac12\int_\Omega|\nabla w_{i,n}|^2dx - \frac12 \sum_{(i,j)\in J^-}\beta_{ij}u_j^2 w_{i,n}^2 - C\left( \int_\Omega|\nabla w_{i,n}|^2dx\right)^{\frac12}
		\end{align}
		where $C$ is a positive constant independent of $n$, yielding that $\inf_n I(w_{i,n},\vec{u}_n) > -\infty$. On the other hand, noticing $\int_\Omega u_{i,n}^2dx = c_i$ we have
		\begin{align} \label{equpboundofI}
			\limsup_{n \to \infty}e_i(\vec{u}_n) \leq \limsup_{n \to \infty}I(u_{i,n},\vec{u}_n) & = \lim_{n \to \infty}\int_\Omega\bigg(\frac12|\nabla u_{i,n}|^2 - \frac12\sum_{(i,j) \in J^-}\beta_{ij}u_{j,n}^2u_{i,n}^2 - \mu_iu_{i,n}^4 - \sum_{(i,j) \in J^+}\beta_{ij}u_{j,n}^2u_{i,n}^2 \bigg) dx \nonumber \\
			& = \int_\Omega\bigg(\frac12|\nabla u_{i}|^2 - \frac12\sum_{(i,j) \in J^-}\beta_{ij}u_{j}^2u_{i}^2 - \mu_iu_{i}^4 - \sum_{(i,j) \in J^+}\beta_{ij}u_{j}^2u_{i}^2 \bigg) dx,
		\end{align}
		which implies that $\sup_n I(w_{i,n},\vec{u}_n) < \infty$. Thus we prove that $\{I(w_{i,n},\vec{u}_n)\}_n$ is a bounded sequence. Then, using \eqref{eqlowboundofI} we obtain that $\{w_{i,n}\}_n$ is bounded in $H_0^1(\Omega)$ and the boundedness of $\{\la_{i,n}\}_n$ follows. Up to a subsequence, we assume that
		\begin{align*}
				& w_{i,n} \rightharpoonup w_i \quad \text{weakly in } H_0^1(\Omega), \\
				& w_{i,n} \to w_i \quad \text{strongly in } L^2(\Omega),
		\end{align*}
		for some $w_i \in H_0^1(\Omega)$ and that $\la_{i,n} \to \la_i$ for some $\la_i \in \R$. Then it is standard to check that $w_i$ solves \eqref{eqsolutionw} for $\la_i$.

	\vskip0.1in
	\emph{Uniqueness:} Take $(w,\lambda_{1})$ and $(v,\lambda_{2})$ to be solutions of
	\begin{align*}
		- \Delta w - \sum_{(i,j) \in J^-}\beta_{ij}u_j^2 w = \mu_iu_i^3 + \sum_{(i,j) \in J^+}\beta_{ij}u_j^2u_i - \la_1u_i, \quad \int_\Omega u_iw dx = c_i
	\end{align*}
	and
	\begin{align*}
		- \Delta v - \sum_{(i,j) \in J^-}\beta_{ij}u_j^2 v = \mu_iu_i^3 + \sum_{(i,j) \in J^+}\beta_{ij}u_j^2u_i - \la_2u_i, \quad \int_\Omega u_iv dx = c_i.
	\end{align*}
	Subtracting the second equation from the first one, multiplying the result by $w - v$ and integrating by parts yields
	\begin{align*}
		\int_\Omega\bigg( |\nabla (w-v)|^2 - \sum_{(i,j)\in J^-}\beta_{ij}u_j^2(w-v)^2\bigg) dx = (\la_2-\la_1)\int_\Omega u_i(w-v)dx = 0.
	\end{align*}
	Since $\beta_{ij} < 0$ for $(i,j)\in J^-$, we get $\int_\Omega|\nabla (w-v)|^2dx = 0$ and thus $w \equiv v$. Consequently, we have $\lambda_{1}=\lambda_{2}$. The proof is complete.
\end{proof}

We can now define the operator
\begin{align} \label{def:Gbeta<0}
	G: S_{\vec{c}} \to H_0^1(\Omega,\R^m), \quad \vec{u} \mapsto G(\vec{u}) = \vec{w} = (w_1,w_2,\cdots,w_m),
\end{align}
that is, for each $\vec{u}$, $w_i$, the $i^{\rm th}$ component of $G(\vec{u})$, is the unique solution of the system \eqref{eqsolutionw} when $\mu_i > 0$ and of the system \eqref{eqsolutionw<0} when $\mu_i < 0$.

\medskip
Next we state and prove four properties of the operator $G$  defined in \eqref{def:Gbeta<0}.

\begin{lemma} \label{lemmaptotangentspa}
	Assume that \eqref{assumption} holds and let $G$ be defined in \eqref{def:Gbeta<0}. Given $\vec{u} \in S_{\vec{c}}$, we have $\vec{u} - G(\vec{u}) \in T_{\vec{u}}S_{\vec{c}}$.
\end{lemma}

\begin{proof}
	The result is clear since
	\begin{align*}
		\int_\Omega u_i(u_i-w_i)dx = \int_\Omega u_i^2dx - \int_\Omega u_iw_idx = c_i -c_i = 0, \quad i = 1,2, \cdots, m.
	\end{align*}
\end{proof}

\begin{lemma} \label{lemofclassC1}
	Assume that \eqref{assumption} holds and let $G$ be defined in \eqref{def:Gbeta<0}. Then $G \in C^1(S_{\vec{c}},H_0^1(\Omega,\R^m))$.
\end{lemma}

\begin{proof}
	We just prove the case of $\mu_i > 0$ since the proof when $\mu_i < 0$ is similar. It suffices to apply the Implicit Function Theorem to the $C^1$ map
	\begin{align*}
		& \Psi : S_{\vec{c}} \times H_0^1(\Omega) \times \R \to H_0^1(\Omega) \times \R, \quad \text{where} \\
		& \Psi(\vec{u},v,\la) = \left(v - (-\Delta)^{-1}\left( \sum_{(i,j)\in J^-}\beta_{ij} u_j^2 v + \mu_i u_i^3 + \sum_{(i,j) \in J^+} \beta_{ij}u_j^2u_i- \la u_i\right), \int_\Omega u_ivdx - c_i \right).
	\end{align*}
	Note that \eqref{eqsolutionw} holds if and only if $\Psi(\vec{u},w_i,\la_i) = (0,0)$. By computing the derivative of $\Psi$ with respect to $v,\la$ at the point $(\vec{u},w_i,\la_i)$ in the direction $(\bar w, \bar \la)$, we obtain a map $\Phi: H_0^1(\Omega) \times \R \to H_0^1(\Omega) \times \R$ given by
	\begin{align*}
		\Phi(\bar w, \bar \la) & := D_{v,\la}\Psi(\vec{u},w_i,\la_i)(\bar w, \bar \la) \\
		& = \left(\bar w - (-\Delta)^{-1}\left( \sum_{(i,j)\in J^-}\beta_{ij}u_j^2 \bar w - \bar \la u_i\right), \int_\Omega u_i\bar wdx \right).
	\end{align*}
	If $\Phi(\bar w, \bar \la) = (0,0)$, then we multiply the equation
	\begin{align*}
		-\Delta \bar w - \sum_{(i,j)\in J^-}\beta_{ij}u_j^2 \bar w = - \bar \la u_i
	\end{align*}
	by $\bar w$ and obtain
	\begin{align*}
		\int_\Omega|\nabla \bar w|^2dx - \sum_{(i,j)\in J^-}\beta_{ij}\int_\Omega u_j^2 \bar w^2dx = - \bar\la\int_\Omega u_i \bar w dx = 0.
	\end{align*}
	So $\bar w \equiv 0$ and then $\bar \la u_i = 0$ in $\Omega$. Since $\int_\Omega u_i^2dx = c_i > 0$, we see that $\bar \la = 0$. Hence $\Phi$ is injective.
	
	On the other hand, for any $(f,c) \in H_0^1(\Omega) \times \R$, let $v_1, v_2 \in H_0^1(\Omega)$ be solutions of the linear problems
	\begin{align*}
		& -\Delta v_1 - \sum_{(i,j)\in J^-}\beta_{ij}u_j^2 v_1 = -\Delta f, \\
		& -\Delta v_2 - \sum_{(i,j)\in J^-}\beta_{ij}u_j^2 v_2 = u_i.
	\end{align*}
	Since $\int_\Omega u_i^2dx = c_i > 0$, so $v_2 \not\equiv 0$ and then $\int_\Omega u_iv_2dx > 0$. Let $\kappa = (c-\int_\Omega u_iv_1dx)/\int_\Omega u_iv_2dx$, then $\Phi(v_1+\kappa v_2,-\kappa) = (f,c)$. Hence $\Phi$ is surjective, that is, $\Phi$ is a bijective map. This completes the proof.
\end{proof}

\begin{lemma} \label{lemcompactoperator}
	Assume that \eqref{assumption} holds and let $G$ be defined in \eqref{def:Gbeta<0}. When $N = 3$, the operator $G$ is compact.
\end{lemma}

\begin{proof}
	As mentioned before, we just provide the proof of $\mu_i > 0$. Let $\{\vec{u}_n\}_n = \{(u_{n,1},u_{n,2},\cdots,u_{n,m})\}_n$ be a bounded sequence in $H_0^1(\Omega,\R^m)$ such that $\vec{u}_n \rightharpoonup \vec{u}$ in $H_0^1(\Omega,\R^m)$ for some $\vec{u} = (u_1,u_2,\cdots,u_m) \in H_0^1(\Omega,\R^m)$, and let $\vec{w}_n = G(\vec{u}_n) = (w_{n,1},w_{n,2},\cdots,w_{n,m})$. By \eqref{eqlowboundofI}, \eqref{equpboundofI} and similar arguments in the proof of Proposition \ref{propuniquewi} when $N = 4$, we deduce that $\{\vec{w}_n\}_n$ is an $H_0^1(\Omega,\R^m)$-bounded sequence. Let $\la_{n,i}$ be the Lagrange multiplier corresponding to $w_{n,i}$ given in Proposition \ref{propuniquewi}. It is clear that $\la_{n,i}$ is also bounded. Up to a subsequence, there exists $w_i \in H_0^1(\Omega)$ such that $w_{n,i}$ converges to $w_i$ weakly in $H_0^1(\Omega)$ and strongly in $L^2(\Omega)$ and $L^4(\Omega)$. Multiplying \eqref{eqsolutionw} by $w_{n,i} - w_i$ we see that
	\begin{align*}
		\int_\Omega \nabla w_{n,i} \nabla (w_{n,i} - w_i)dx = & \sum_{(i,j) \in J^-}\int_\Omega \beta_{ij}u_{n,j}^2w_{n,i}(w_{n,i} - w_i)dx + \sum_{(i,j) \in J^+}\int_\Omega \beta_{ij}u_{n,j}^2u_{n,i}(w_{n,i} - w_i)dx \\
		& + \int_\Omega\left(\mu_iu_{n,i}^3-\la_{n,i} u_{n,i} \right) (w_{n,i} - w_i)dx \to 0
	\end{align*}
	as $n \to \infty$ and therefore $ w_{n,i} \to w_i$ strongly in $H_0^1(\Omega)$.
\end{proof}

\begin{lemma} \label{lemG(u)beta<0}
	Assume that \eqref{assumption} holds and let $G$ be defined in \eqref{def:Gbeta<0}. Assume that $\rho$ is a fixed positive number, satisfying $2\rho < (\beta_{\max}^+)^{-1}\mathcal{C}_N^2$ if $N = 4$. Then, for any $\vec{c} = (c_1,\cdots,c_m)$ such that $c_j > 0$, $j = 1,\cdots,m$, and $(\mu_i^+ + \beta_{\max}^+)^2\rho^3 \leq \Lambda_1\mathcal{C}_N^{8}c_i$ where $\mu_i^+ := \max\{\mu_i,0\}$, we can take $\hat \delta = \hat \delta(\vec{c},\rho) > 0$ such that when $0 < \delta < \hat \delta$, $G(\vec{u}) \in (\pm\p_i)_{\delta/2}$ for any $\vec{u} \in \overline{(\pm\p_i)_\delta} \cap \B_\rho$, $i = 1,\cdots,m$. Moreover, we have $G(\vec{u}) \in \pm\p_i$ for any $\vec{u} \in \pm\p_i \cap \B_\rho$, $i = 1,\cdots,m$.
\end{lemma}

\begin{proof}
	As mentioned before, we just provide the proof of $\mu_i > 0$.
	
	Let $\vec{u} \in \overline{(\p_i)_\delta} \cap \B_\rho$ and $\vec{w} = G(\vec{u}) = (w_1,w_2,\cdots,w_m)$. For any $i \in \{1,2,\cdots,m\}$, let $\la_i$ be the Lagrange multiplier corresponding to $w_i$ given by Proposition \ref{propuniquewi}. From $\int_\Omega u_iw_i dx = c_i$ we deduce that
	\begin{align*}
		c_i \leq \left( \int_\Omega u_i^2dx\right) ^{\frac12}\left( \int_\Omega w_i^2dx\right) ^{\frac12} = \sqrt{c_i}\left( \int_\Omega w_i^2dx\right) ^{\frac12},
	\end{align*}
	that is, $\int_\Omega w_i^2dx \geq c_i$. So $\int_\Omega |\nabla w_i|^2dx \geq \Lambda_1c_i$.
	Then, since $(\mu_i + \beta_{\max}^+)^2\rho^3 \leq \Lambda_1\mathcal{C}_N^{8}c_i$, we have
	\begin{align*}
		\la_i c_i & = \mu_i\int_\Omega u_i^3w_idx - \int_\Omega|\nabla w_i|^2dx + \sum_{(i,j) \in J^-}\beta_{ij}\int_\Omega u_j^2w_i^2dx + \sum_{(i,j) \in J^+}\beta_{ij}\int_\Omega u_j^2u_iw_idx \\
		& \leq \mu_i\left( \int_\Omega u_i^4dx\right)^{\frac34}\left( \int_\Omega w_i^4dx\right)^{\frac14} - \int_\Omega|\nabla w_i|^2dx + \sum_{(i,j) \in J^+}\beta_{ij}\left( \int_\Omega u_j^4dx\right)^{\frac12}\left( \int_\Omega u_i^4dx\right)^{\frac14}\left( \int_\Omega w_i^4dx\right)^{\frac14} \\
		& \le \mu_i\mathcal{C}_N^{-4}\left( \int_\Omega |\nabla u_i|^2dx\right)^{\frac32}\left( \int_\Omega |\nabla w_i|^2dx\right)^{\frac12} - \int_\Omega|\nabla w_i|^2dx \\
		& ~~~~~~~~ + \sum_{(i,j) \in J^+}\beta_{ij}\mathcal{C}_N^{-4}\left( \int_\Omega |\nabla u_j|^2dx\right)\left( \int_\Omega |\nabla u_i|^2dx\right)^{\frac12}\left( \int_\Omega |\nabla w_i|^2dx\right)^{\frac12} \\
		& \le \left( (\mu_i + \beta_{\max}^+)\mathcal{C}_N^{-4}\rho^{\frac32} - \left( \Lambda_1 c_i\right)^{\frac12}\right) \left( \int_\Omega |\nabla w_i|^2dx\right)^{\frac12} \\
		& \le 0.
	\end{align*}
	Multiplying \eqref{eqsolutionw} by $w_i^-$ and integrating by parts yields
	\begin{align*}
		\int_\Omega |\nabla w_i^-|^2dx - \sum_{(i,j) \in J^-}\beta_{ij}\int_\Omega u_j^2(w_i^-)^2dx = \mu_i\int_\Omega u_i^3w_i^-dx + \sum_{(i,j) \in J^+}\beta_{ij} \int_\Omega u_j^2u_iw_i^-dx -\la_i\int_\Omega u_iw_i^-dx,
	\end{align*}
	where $w_i^- := \min\{w_i,0\} \leq 0$. Let $\Omega_i^- = \{x \in \Omega: w_i(x) < 0\}$ and $K = |\Omega_i^-|$. We get
	\begin{align*}
		&~ \text{dist}(\vec{w},\p_i)\|w_i^-\| \\
		\leq & ~ \int_\Omega |\nabla w_i^-|^2dx \leq \int_\Omega |\nabla w_i^-|^2dx - \sum_{(i,j) \in J^-}\beta_{ij}\int_\Omega u_j^2(w_i^-)^2dx \\
		= & ~ \mu_i\int_\Omega u_i^3w_i^-dx + \sum_{(i,j) \in J^+}\beta_{ij} \int_\Omega u_j^2u_iw_i^-dx -\la_i\int_\Omega u_iw_i^-dx \\
		\leq & ~ \mu_i\int_\Omega (u_i^-)^3w_i^-dx + \sum_{(i,j) \in J^+}\beta_{ij} \int_\Omega u_j^2u_i^-w_i^-dx -\la_i\int_\Omega u_i^-w_i^-dx \\
		\leq & ~ \mu_i\left(\int_\Omega (u_i^-)^4dx\right)^{\frac34}\left(\int_\Omega (w_i^-)^4dx\right)^{\frac14} + \beta_{\max}^+ \sum_{(i,j) \in J^+}\left(\int_{\Omega_i^-} u_j^4dx\right)^{\frac12}\left(\int_\Omega (u_i^-)^4dx\right)^{\frac14}\left(\int_\Omega (w_i^-)^4dx\right)^{\frac14} \\
		& ~~~~~~~~~ -\la_i K^{\frac12}\left(\int_\Omega (u_i^-)^4dx\right)^{\frac14}\left(\int_\Omega (w_i^-)^4dx\right)^{\frac14},
	\end{align*}
	since we assume $\mu_i > 0$ and have proved $\la_i \leq 0$, which together with the Sobolev inequality implies that
	\begin{align} \label{eqw-l4leq}
		\text{dist}(\vec{w},\p_i) \leq \mathcal{C}_N^{-1}\left(\mu_i\left(\int_\Omega (u_i^-)^4dx\right)^{\frac12} +\beta_{\max}^+ \sum_{(i,j) \in J^+}\left(\int_{\Omega_i^-} u_j^4dx\right)^{\frac12} -\la_i K^{\frac12}\right)\left(\int_\Omega (u_i^-)^4dx\right)^{\frac14}.
	\end{align}
	Since $\vec{u} \in \B_\rho$ and $G$ is continuous, we know $\vec{w}$ is uniformly bounded in $H_0^1(\Omega,\R^m)$ and then we can check that $|\la_i| < \bar \la$ for some $\bar \la > 0$ independent of $\vec{u} \in \B_\rho$. Next, we prove $K > 0$ is small enough when $\delta > 0$ is sufficiently small independent of $\vec{u}$. Note that
	\begin{align*}
		K \leq \bar K(\delta) = \sup \bigl\{|\{x \in \Omega: w_i(x) < 0\}|: \vec{u} \in \overline{(\p_i)_\delta} \cap \B_\rho,~ \vec{w} = G(\vec{u}) = (w_1, \cdots, w_m)\bigr\}.
	\end{align*}
	It suffices to prove that $\bar K(\delta) \to 0$ as $\delta \to 0^+$. Arguing by contradiction, we suppose that there exists $\delta_n \to 0^+$ such that $\lim_{n \to \infty}\bar K(\delta_n) > 0$. That is, we can take $\vec{u}_n \in \overline{(\p_i)_{\delta_n}} \cap \B_\rho$ such that $\lim_{n \to \infty}|\{x \in \Omega: w_{n,i}(x) < 0\}| > 0$, where $w_{n,i}$ is the $i^{\rm th}$ component of $\vec{w}_n = G(\vec{u}_n)$. Up to a subsequence, we assume that
	\begin{align*}
		\vec{u}_n \to \vec{u}^* = (u_1^*,\cdots,u_m^*) \in \p_i, \quad \text{weakly in } H_0^1(\Omega,\R^m), \text{ and strongly in } L^2(\Omega,\R^m).
	\end{align*}
	Note that $\vec{u}^* \in \B_\rho$ and we have checked that $\la_{\vec{u}^*,i} \le 0$. Moreover, since $\mu_i > 0$, $u_i^* \geq 0$, up to a subsequence we get
	\begin{align*}
		w_{n,i} \to \left( -\Delta - \sum_{(i,j) \in J^-}\beta_{ij}(u_j^*)^2\right) ^{-1}\bigg(\mu_i(u_i^*)^3+\sum_{(i,j) \in J^+}\beta_{ij}(u_j^*)^2u_i^* - \la_{\vec{u}^*,i}u_i^*\bigg) > 0 \quad \text{strongly in } L^2(\Omega,\R^m),
	\end{align*}
	where the last "$ > 0$" follows by the strong maximum principle, and therefore we can conclude that $\lim_{n \to \infty}\{|x \in \Omega: w_{n,i}(x) < 0\}| = 0$ since $\Omega$ is bounded, contradicting $\lim_{n \to \infty}\{|x \in \Omega: w_{n,i}(x) < 0\}| > 0$. This contradiction enables us to take sufficiently small $K > 0$ if $\delta > 0$ is small enough. Observe that
	\begin{align} \label{eqdistwu}
		\left(\int_\Omega |u_i^-|^4dx\right)^{\frac14} & = \inf\bigl\{\|\vec{u} - \vec{v}\|_{L^4(\Omega,\R^m)}: \vec{v} \in \p_i\bigr\} \nonumber \\
		& \leq \mathcal{C}_N^{-1}\inf\bigl\{\|\vec{u} - \vec{v}\|: \vec{v} \in \p_i\bigr\} = \mathcal{C}_N^{-1}\text{dist}(\vec{u},\p_i).
	\end{align}
	When $N = 3$, by H\"older inequality and Sobolev inequality we get
	\begin{align*}
		\sum_{(i,j)\in J^+}\left(\int_{\Omega_i^-}u_j^4dx\right)^{\frac12} \leq K^{\frac16}\sum_{(i,j)\in J^+}\left(\int_{\Omega_i^-}u_j^6dx\right)^{\frac13} \leq K^{\frac16}\mathcal{C}_{N,6}^{-2}\sum_{(i,j)\in J^+}\int_\Omega|\nabla u_j|^2dx < K^{\frac16}\mathcal{C}_{N,6}^{-2} \rho
	\end{align*}
	where $\mathcal{C}_{N,6}$ is the Sobolev best constant of $H_0^1(\Omega)\hookrightarrow L^6(\Omega)$ defined by
	\begin{equation*}
		\mathcal{C}_{N,6} :=\inf_{u \in H_0^1(\Omega) \setminus \{0\} } \cfrac{\left( \int_{\Omega} |\nabla u|^2 dx\right) ^{1/2}}{\left(\int_{\Omega} |u|^6dx\right)^{1/6}}.
	\end{equation*}
	Then, by \eqref{eqw-l4leq} and \eqref{eqdistwu} we have
	\begin{align*}
		\text{dist}(\vec{w},\p_i) \leq \mathcal{C}_N^{-2}\left( \mu_{\max}^+ \mathcal{C}_N^{-2}\delta^2 + K^{\frac16}\mathcal{C}_{N,6}^{-2} \rho \beta_{\max}^+ + \bar \la K^{\frac12} \right)\delta,
	\end{align*}
	and therefore, for small $\delta > 0$, we conclude that $\vec{w} \in (\p_i)_{\delta/2}$. When $N = 4$, since $2\rho < (\beta_{\max}^+)^{-1}\mathcal{C}_N^2$ we obtain
	\begin{align*}
		\beta_{\max}^+\sum_{(i,j)\in J^+}\left(\int_{\Omega_i^-}u_j^4dx\right)^{\frac12} \leq \beta_{\max}^+\mathcal{C}_N^{-2} \sum_{(i,j)\in J^+}\int_\Omega|\nabla u_j|^2dx \leq \beta_{\max}^+\mathcal{C}_N^{-2} \rho < \frac12.
	\end{align*}
	Then, by \eqref{eqw-l4leq} and \eqref{eqdistwu} we have
	\begin{align*}
		\text{dist}(\vec{w},\p) \leq \mathcal{C}_N^{-2}\left( \mu_{\max}^+ \mathcal{C}_N^{-2}\delta^2 + \beta_{\max}^+ \mathcal{C}_N^{-2} \rho + \bar \la K^{\frac12} \right)\delta,
	\end{align*}
	and therefore, for small $\delta > 0$, we conclude that $\vec{w} \in (\p_i)_{\delta/2}$. Moreover, using \eqref{eqw-l4leq} we get that $G(\vec{u}) \in \p_i$ for any $\vec{u} \in \p_i \cap \B_\rho$.
	Similar argument works for the "$-$" sign and we complete the proof.
\end{proof}

Now let us define a map
\begin{align} \label{def:Vbeta<0}
	V : S_{\vec{c}} \to H_0^1(\Omega,\R^m), \quad V(\vec{u}) = \vec{u} - G(\vec{u}).
\end{align}
One can check that $V$ is exactly the gradient of $E$ constrained to $S_{\vec{c}}$ if $\mu_i > 0$ and $\beta_{ij} > 0 ~ \forall i \neq j$. Observe that, for $\vec{u} \in S_{\vec{c}}$,
\begin{align*}
	V(\vec{u}) = 0 \quad \Leftrightarrow \quad \vec{u} \text{ solves } \eqref{eq:mainequation}.
\end{align*}

\begin{lemma}[Palais-Smale condition] \label{lempsconbeta>0}
	Assume that \eqref{assumption} holds and let $V$ be defined in \eqref{def:Vbeta<0}, and let $\{\vec{u}_n\} \subset \B_\rho$ where $\rho$ satisfies \eqref{defi of rho} be such that
	\begin{align*}
		E(\vec{u}_n) \to c < \begin{cases}
			\infty, & \text{ when } N = 3,\\
			\frac{1}{4}( \mu_{\max}^++\beta_{\max}^+)^{-1}\mathcal{C}_N^4, &\text{ when } N=4,
		\end{cases} \quad \text{and} \quad V(\vec{u}_n) \to 0 \quad \text{strongly in } H_0^1(\Omega,\R^m).
	\end{align*}
	Then, up to a subsequence, there exists $\vec{u} \in \overline{	\B_\rho}\subset S_{\vec{c}}$ such that $\vec{u}_n \to \vec{u}$ strongly in $H_0^1(\Omega,\R^m)$ and $V(\vec{u}) = 0$.
\end{lemma}

\begin{proof}
	Observe that the sequence  $\{\vec{u}_n\}=\{\sbr{ u_{n,1},\ldots, u_{n,m}}\}$ is bounded in $H_0^1(\Omega,\R^m)$ as it is contained in the bounded set $\B_\rho \subset S_{\vec{c}} $. Let $ \lambda_{n,i}$ be the corresponding Lagrange multiplier to $ u_{n,i}$ given in Proposition \ref{propuniquewi}. Then,  it is straightforward to verify that the sequence  $\lbr{ \lambda_{n,i}}$ is bounded in $\R$.   Therefore,  up to subsequence, we may assume that there exist $\vec{u}\in \B_\rho \subset S_{\vec{c}}$  and $ \lambda_i\in \R$ such that
	\begin{equation}\label{m6}
		\begin{aligned}
			& u_{n,i} \rightharpoonup u_i \quad \text{ weakly in } H_0^1(\Omega), \\
			& u_{n,i} \rightharpoonup u_i \quad \text{ weakly in } L^{2^\ast}(\Omega),\\
			& u_{n,i} \to u_i\quad \text{ strongly in } L^r(\Omega) \text{ for } 2 \leq r <2^\ast,\\
			& u_{n,i} \to u_i \quad \text{ almost everywhere in } \Omega.
			\\& \lambda_{n,i} \to \lambda_{i}  \quad \text{   in } \R,
		\end{aligned}
	\end{equation}
	where $2^*=2N/(N-2)$.
	\medbreak
	For any $i\in \lbr{1,\ldots,m}$, let $\tilde{\lambda}_{i}$ be  the  Lagrange multiplier associated with  $u_i$, then we claim that $ \tilde{\lambda}_{i} =\lambda_{i}$. Indeed, for any  $ v\in  H_0^1(\Omega)$ satisfying that $\int_{\Omega}  u_{i}v dx \neq 0 $,   then we can see from the fact $V(\vec{u}_n) \to 0$ that
	\begin{equation}
		\begin{aligned}
			\lambda_{n,i}  \int_{\Omega}  u_{n,i}v dx=-\int_{\Omega}  \nabla u_{n,i}  \cdot \nabla v dx +  \mu_i \int_{\Omega} u_{n,i}^3 vdx + \sum_{j \neq i} \beta_{ij} \int_{\Omega} u_{n,i} u_{n,j}^2 vdx +o_{n}(1).
		\end{aligned}
	\end{equation}
	Letting $n\to \infty$, we deduce from  \eqref{m6} that
	\begin{equation}\label{4.12}
			\begin{aligned}
				\lambda_{i}  \int_{\Omega}  u_{i}v dx=-\int_{\Omega}  \nabla u_{i}  \cdot \nabla v dx +  \mu_i \int_{\Omega} u_{i}^3 vdx +   \sum_{j \neq i} \beta_{ij} \int_{\Omega} u_{i} u_{j}^2 vdx =\tilde{\lambda}_{i}  \int_{\Omega}  u_{i}v dx,
			\end{aligned}
	\end{equation}
	which implies that $ \tilde{\lambda}_{i} =\lambda_{i}$. Therefore, by using \eqref{4.12} and the fact  $V(\vec{u}_n) \to 0$  again,  it is  easy to  verify that $V(\vec{u}) = 0$.
	\medbreak
	It remains to prove that $\vec{u}_n \to \vec{u}$ strongly in $H_0^1(\Omega,\R^m)$.  When $N = 3$, this is Sobolev subcritical case, and it is classical to prove that $\vec{u}_n \to \vec{u}$ strongly in $H_0^1(\Omega,\R^m)$ by using the compact embedding of $H_0^1(\Omega) \hookrightarrow L^2(\Omega)$ and of $H_0^1(\Omega) \hookrightarrow L^4(\Omega)$.
	\medbreak
	Next we consider the Sobolev critical case when $N = 4$. In such case, the embedding  $H_0^1(\Omega) \hookrightarrow L^4(\Omega)$ is not compact, which is much more complicated.   Let $ \vec{\sigma}_{n}=\sbr{   \sigma_{n,1},\ldots, \sigma_{n,m}}=\vec{u}_n  - \vec{u}$. Then we aim to prove $  \vec{\sigma}_{n} \to 0$ strongly in $H_0^1(\Omega,\R^m)$  by  using a contradiction argument. Without loss of generality,  we  may assume that there exists some  $ i\in \lbr{1,\ldots,m}$, such that
	\begin{equation}\label{m5}
		\int_{\Omega} | \nabla \sigma_{n,i}|^2 dx \to l >0, \quad n\to\infty.
	\end{equation}
	Recall that $V(\vec{u}_n) \to 0$ and  $V(\vec{u}) = 0$,  we have
	\begin{equation}
		\int_{\Omega}  |\nabla u_{n,i}|^2    dx+ \lambda_{n,i}  c_i  =  \mu_i \int_{\Omega} u_{n,i}^4  dx + \sum_{j \neq i} \beta_{ij} \int_{\Omega} u_{n,i}^2  u_{n,j}^2 dx +o_{n}(1),
	\end{equation}
	\begin{equation}
		\int_{\Omega}  |\nabla u_{i}|^2    dx+ \lambda_{i}  c_i  =  \mu_i \int_{\Omega} u_{i}^4  dx +  \sum_{j \neq i} \beta_{ij} \int_{\Omega} u_{i}^2  u_{j}^2 dx .
	\end{equation}
	By using the  Br\'ezis-Lieb Lemma (see \cite{BL}  for the case of one component and \cite[p.447]{Zou2015} for  the case of two components), we have
	\begin{equation} \label{m4}
		\int_{\Omega} u_{n,i}^4 dx=\int_{\Omega}\sigma_{n,i}^4dx+	\int_{\Omega} u_{i}^4dx+o_{n}(1),\quad \int_{\Omega}  u_{n,i} ^2 u_{n,j} ^2dx=\int_{\Omega} \sigma_{n,i}^2 \sigma_{n,j} ^2 dx+	\int_{\Omega}  u_{i} ^2 u_{j} ^2dx+o_{n}(1).
	\end{equation}
	Therefore, we obtain
	\begin{equation}
			\begin{aligned}
				\int_{\Omega}  |\nabla \sigma_{n,i}|^2    dx &=\mu_i \int_{\Omega} \sigma_{n,i}^4  dx +  \sum_{j \neq i} \beta_{ij} \int_{\Omega} \sigma_{n,i}^2  \sigma_{n,j}^2 dx+o_{n}(1) \\
				& \leq \mu_i \int_{\Omega} \sigma_{n,i}^4  dx +\beta_{\max}^+ \sbr{\int_{\Omega} \sigma_{n,i}^4dx}^{\frac{1}{2}}\sum_{ j \neq i}		\sbr{\int_{\Omega} \sigma_{n,j}^4dx}^{\frac{1}{2}}+o_{n}(1)\\			
				&\leq  \mu_{\max}^+ \mathcal{C}_N^{-4}\sbr{ \int_{\Omega}  |\nabla \sigma_{n,i}|^2  dx}^2 +\beta_{\max}^+\mathcal{C}_N^{-4}\int_{\Omega}  |\nabla \sigma_{n,i}|^2dx \sum_{j \neq i}	\int_{\Omega}  |\nabla \sigma_{n,j}|^2dx+o_{n}(1).
			\end{aligned}
	\end{equation}
	If $\mu_{\max}^+ = 0$, together with \eqref{defi of rho} we get
	\begin{align*}
		0 < (1-2\beta_{\max}^+ \rho \mathcal{C}_N^{-4})l \le 0,
	\end{align*}
	which is a contradiction.
	When $\mu_{\max}^+ > 0$, observing that $ \|\vec{\sigma}_n\|^2 \le 2\rho$, and by	letting $ n\to \infty$, we have
	\begin{equation}
		l\geq (1-2\beta_{\max}^+ \rho \mathcal{C}_N^{-4}) \sbr{\mu_{\max}^+}^{-1}\mathcal{C}_N^{4} \geq \sbr{\mu_{\max}^++\beta_{\max}^+}^{-1}\mathcal{C}_N^{4},
	\end{equation}
	by the choice of $\rho$, see \eqref{defi of rho}.
	Then we have
	\begin{equation}
		\begin{aligned}
			E(\vec{\sigma}_n)= \frac{1}{4}\sum_{i= 1}^m  \int_{\Omega}  |\nabla \sigma_{n,i}|^2 +o_{n}(1) \geq\frac{1}{4}  \int_{\Omega}  |\nabla \sigma_{n,i}|^2 dx +o_{n}(1) \geq  \frac{1}{4} \sbr{\mu_{\max}^++\beta_{\max}^+}^{-1}\mathcal{C}_N^{4} +o_{n}(1)  .
		\end{aligned}
	\end{equation}
	Moreover, since
	\[\sum_{i=1}^m\int_{\Omega}\lvert\nabla u_{i}\rvert^2dx\leq\liminf_{n \to \infty}\sum_{i=1}^m\int_{\Omega}\lvert\nabla u_{n,i}\rvert^2dx\leq\rho,\]
	we deduce from Lemma \ref{lemma 2.1} that  $E(\vec{u}) \geq M_0\nm{\vec{u}}\geq 0$.  Then  by using \eqref{m4} again, we have
	\begin{equation}
		\begin{aligned}
			E(\vec{u}_n) =E(\vec{u})+  E(\vec{\sigma}_n)+o_n(1) \geq   \frac{1}{4} \sbr{\mu_{\max}^++\beta_{\max}^+}^{-1}\mathcal{C}_N^{4}+o_n(1) .
		\end{aligned}
	\end{equation}
	However, this is impossible since 	$E(\vec{u}_n) \to c < \frac{1}{4} \sbr{\mu_{\max}^++\beta_{\max}^+}^{-1}\mathcal{C}_N^{4}$, and we get a contradiction. Therefore, $  \vec{\sigma}_{n} \to 0$ strongly in $H_0^1(\Omega,\R^m)$  and the proof is completed.
\end{proof}

\vskip 0.15in

By Lemma \ref{lemmaptotangentspa}, we have $V(\vec{u}) \in T_{\vec{u}}S_{\vec{c}}$ for any $\vec{u} \in S_{\vec{c}}$. Next we prove that $V$ is a pseudogradient for $E$ constrained to $S_{\vec{c}}$.

\begin{lemma} \label{pseudo}
	Assume that \eqref{assumption} holds and let $V$ be defined in \eqref{def:Vbeta<0}. We have
	\begin{align*}
		\langle \nabla E|_{S_{\vec{c}}}(\vec{u}),V(\vec{u}) \rangle \geq \|V(\vec{u})\|^2 \quad \text{for any } \vec{u} \in S_{\vec{c}}.
	\end{align*}
\end{lemma}

\begin{proof}
	Let $\vec{u} = (u_1,u_2,\cdots,u_m) \in S_{\vec{c}}$ and $\vec{w} = G(\vec{u}) = (w_1,w_2,\cdots,w_m)$. Note that $\int_\Omega u_i(u_i-w_i)dx = 0$. We have
	\begin{align*}
		& \langle \nabla E|_{S_{\vec{c}}}(\vec{u}),V(\vec{u}) \rangle \\
		= ~ & \sum_{i=1}^m \int_\Omega \bigg(\nabla u_i \nabla (u_i-w_i) - \mu_i u_i^3(u_i-w_i) - \sum_{j\neq i}\beta_{ij}u_j^2u_i(u_i-w_i)\bigg)dx \\
		\geq ~ & \sum_{i=1}^m \int_\Omega \bigg(\nabla u_i \nabla (u_i-w_i) - \mu_i u_i^3(u_i-w_i) - \sum_{(i,j) \in J^+}\beta_{ij}u_j^2u_i(u_i-w_i) - \sum_{(i,j) \in J^-}\beta_{ij}u_j^2w_i(u_i-w_i)\bigg)dx \\
		\ge ~ & \sum_{i=1}^m \int_\Omega |\nabla (u_i-w_i)|^2dx \\
		= ~ & \|V(\vec{u})\|^2.
	\end{align*}
	The proof is completed.
\end{proof}

With $V$ we can now construct a decreasing flow $\eta(t,\cdot): S_{\vec{c}} \to S_{\vec{c}}$ as
\begin{align} \label{eqflowbeta<0}
	\left\{
	\begin{aligned}
		& \frac{\partial}{\partial t}\eta(t,\vec{u}) = -W(\eta(t,\vec{u})), \\
		& \eta(0,\vec{u}) = \vec{u},
	\end{aligned}
	\right.
\end{align}
where 
\begin{align}
	& U_1 := \bigl\{\vec{u} \in S_{\vec{c}}: E(\vec{u}) \geq b + 2\epsilon \text{ or } E(\vec{u}) \le a - 2\epsilon\bigr\}, \\
	& U_2 := \bigl\{\vec{u} \in S_{\vec{c}}: a -\epsilon \le E(\vec{u}) \leq b+ \epsilon\bigr\},\\
	& h(\vec{u}):= \frac{\text{dist}(\vec{u},U_1)}{\text{dist}(\vec{u},U_2)+\text{dist}(\vec{u},U_1)},
\end{align}
for $a \le b$, $\ep > 0$, and
\begin{align*}
	W(\vec{u}) = h(\vec{u})V(\vec{u}).
\end{align*}

\begin{proposition} \label{propconmapofAbeta<0}
	Assume that \eqref{assumption} holds. Let $\rho$ be a fixed positive number satisfying $2\rho < (\beta_{\max}^+)^{-1}\mathcal{C}_N^2$ if $N = 4$, and let $\vec{c} = (c_1,\cdots,c_m)$ such that $c_i > 0$ and $(\mu_i^+ + \beta_{\max}^+)^2\rho^3 \leq \Lambda_1\mathcal{C}_N^{8}c_i$, $i = 1,\cdots,m$. Let $\delta \in (0,\hat \delta)$ where $\hat \delta > 0$ is given in Lemma \ref{lemG(u)beta<0}.
	
	For $a \le b < M_1$ and $0 <2\ep < M_1-b$ such that
	\begin{align} \label{eqVgeepbeta<0}
		\|V(\vec{u})\| \geq \sqrt{2\ep}, \quad \forall \vec{u} \in (E|_{S_{\vec{c}}})^{-1}[a-\ep,b+\ep] \cap S_d^*(\delta) \cap \sbr{\cap_{i=d+1}^m \p_i} \cap \B_\rho^{M_1},
	\end{align}
	where $V(u)$ is the pseudogradient for $E$ constrained to $S_{\vec{c}}$ defined in \eqref{def:Vbeta<0}, the flow $\eta(t,\vec{u})$, defined in \eqref{eqflowbeta<0}, globally exists for all $t \geq 0$ if $\vec{u} \in \B_\rho^{M_1}$ and the following hold:
	\begin{itemize}
		\item[(i)] $\eta(t,\vec{u}) = \vec{u}$ for all $t > 0$ if $\vec{u} \in U_1$;
		\item[(ii)] $E(\eta(\frac{b-a+2\ep}{2\ep},\vec{u})) \leq a -\ep$ or $\eta(\frac{b-a+2\ep}{2\ep},\vec{u}) \in \overline{\PR_\delta^{(d)}}$ if $\vec{u} \in \sbr{\cap_{i=d+1}^m \p_i} \cap \B_\rho$ with $E(\vec{u}) \leq b + \ep$.
		\item[(iii)] $\eta(t,\vec{u}) \in \overline{\PR_\delta^{(d)}} \cap \B_\rho^{M_1}$ for any $t \ge 0$, $\vec{u} \in \overline{\PR_\delta^{(d)}} \cap \B_\rho^{M_1}$.
		\item[(iv)] $\eta(t,\vec{u}) \in \sbr{\cap_{i=d+1}^m \p_i} \cap \B_\rho^{M_1}$ for any $t \ge 0$, $\vec{u} \in \sbr{\cap_{i=d+1}^m \p_i} \cap \B_\rho^{M_1}$.
	\end{itemize}
\end{proposition}


\begin{proof}
	Note that $\B_\rho^{M_1} \subset \B_\rho$ is a bounded set in $H_0^1(\Omega,\R^m)$. It is standard that $\eta(t,\vec{u})$ locally exists if $\vec{u} \in \B_\rho^{M_1}$ since $W$ is Lipschitz continuous. Fix $\vec{u} \in \B_\rho^{M_1}$. We first prove $\eta(t,\vec{u}) \in \B_\rho^{M_1}$ for all $t \in [0,T_{\vec{u}})$, where $T_{\vec{u}}$ is the maximal time such that $\eta(t,\vec{u})$ exists. In fact, since
	\begin{align*}
		dE(\eta(t,\vec{u})) = \langle \nabla E|_{S_{\vec{c}}}(\eta(t,\vec{u})), \frac{\partial}{\partial t}\eta(t,\vec{u}) \rangle = -h(\eta(t,\vec{u}))\|V(\eta(t,\vec{u}))\|^2 \leq 0,
	\end{align*}
	we have
	\begin{align*}
		E(\eta(t,\vec{u})) \leq E(\vec{u}) < M_1
	\end{align*}
	for $t \in [0,T_{\vec{u}})$. Hence, if $\eta(t_{\vec{u}},\vec{u}) \notin \B_\rho^{M_1}$ for some $0 < t_{\vec{u}} < T_{\vec{u}}$, then $\eta(t_{\vec{u}},\vec{u}) \notin \B_{\rho}$. By continuity, there exists $t_{\vec{u}}' \in (0,t_{\vec{u}}]$ such that $\eta(t_{\vec{u}}',\vec{u}) \in \partial \B_{\rho}$. By Lemma \ref{lemma 2.1}, we have $E(\eta(t_{\vec{u}}',\vec{u})) \geq M_1$, contradicting
	\begin{align*}
		E(\eta(t_{\vec{u}}',\vec{u})) \leq E(\vec{u}) < M_1.
	\end{align*}
	This contradiction shows that $\B_\rho^{M_1}$ is an invariant subset with respect to the flow $\eta$.
	
	To prove $\eta(t,\vec{u})$ globally exists if $\vec{u} \in \B_\rho^{M_1}$, arguing by contradiction we suppose $T_{\vec{u}} < \infty$ for some $\vec{u} \in \B_\rho^{M_1}$. On the one hand, it is necessary that
	\begin{align*}
		\limsup_{t \to T_{\vec{u}}}\|W(\eta(t,\vec{u}))\| \to \infty.
	\end{align*}
	On the other hand, since $\eta(t,\vec{u}) \in \B_\rho^{M_1} \subset \B_{\rho}$ for all $0 < t < T_{\vec{u}}$, $\eta(t,\vec{u})$ is uniformly bounded in $H_0^1(\Omega,\R^m)$ and so $\|W(\eta(t,\vec{u}))\|$ is uniformly bounded with respect to $t$. This is a contradiction.
	
	In the following, observe that $W(\vec{u}) = 0$ if $\vec{u} \in U_1$. Then (i) follows. Next we claim that $\eta(t,\vec{u}) \in \overline{(\pm\PR_i)_\delta} \cap \B_\rho^{M_1}$ for any $t \ge 0$ and $\vec{u} \in \overline{(\pm\PR_i)_\delta} \cap \B_\rho^{M_1}$. Let $\vec{u} \in \overline{(\pm\PR_i)_\delta} \cap \B_\rho^{M_1}$. We have proved that $\eta(t,\vec{u}) \in \B_\rho^{M_1}$ for all $t \ge 0$ and it suffices to show  $\eta(t,\vec{u}) \in \overline{(\pm\PR_i)_\delta}$ for $t \ge 0$. We use Corollary \ref{corbm} with $B_{\vec{c}} = \overline{(\pm\PR_i)_\delta} \cap S_{\vec{c}}$ and $B^* = \B_\rho^{M_1}$. The fact that \eqref{eqbm} is satisfied is a consequence of Lemma \ref{lemconvex} used with $\tilde{B} = \overline{(\pm\PR_i)_\delta}$. Indeed, we know from Lemma \ref{lemG(u)beta<0} that $G(\vec{u}) \in (\pm\p_i)_\delta$ if $\vec{u} \in \overline{(\pm\PR_i)_\delta} \cap B_\rho^{M_1}$ and thus \eqref{eqd=0atapoint}, or equivalently \eqref{eqbm}, holds for any $\vec{u} \in \overline{(\pm\PR_i)_\delta} \cap \B_\rho^{M_1} = B_{\vec{c}} \cap \B_\rho^{M_1}.$  Applying Corollary \ref{corbm} we obtain that  $\eta(t,\vec{u}) \in \overline{(\pm\PR_i)_\delta} \cap \B_\rho^{M_1}$ for all $0 < t < T_{\vec{u}}$. Moreover, we have shown that $T_{\vec{u}} = \infty$, which completes the proof of the claim. Then (iii) follows. Similarly, we can prove $\eta(t,\vec{u}) \in \pm\PR_i \cap \B_\rho^{M_1}$ for any $t \ge 0$ and $\vec{u} \in \pm\PR_i \cap \B_\rho^{M_1}$ and then validate (iv).
	
	Finally, we prove (ii). Let $\vec{u} \in \sbr{\cap_{i=d+1}^m \p_i} \cap \B_\rho$ with $E(\vec{u}) \leq b + \ep$. Note that $\vec{u} \in \B_\rho^{M_1}$ since $b + \ep < M_1$. Thus, if $\eta(t,\vec{u}) \in \overline{\PR^{(d)}_\delta}$ for some $t \in [0,\frac{b-a+2\ep}{2\ep}]$, we have $\eta(\frac{b-a+2\ep}{2\ep},\vec{u}) \in \overline{\PR^{(d)}_\delta}$. Moreover, if $E(\eta(t,\vec{u})) \leq a -\ep$ for some $t \in [0,\frac{b-a+2\ep}{2\ep}]$, then
	\begin{align*}
		E(\eta(\frac{b-a+2\ep}{2\ep},\vec{u})) \leq E(\eta(t,\vec{u})) \leq a -\ep.
	\end{align*}
	So, to arrive at the desired conclusion, it is sufficient to find a contradiction if we assume that $E(\eta(t,\vec{u})) > a - \ep$ and $\eta(t,\vec{u}) \in S_d^*(\delta)$ for all $t \in [0,\frac{b-a+2\ep}{2\ep}]$. In this case, we get
	\begin{align*}
		\eta(t,\vec{u}) \in (E|_{S_{\vec{c}}})^{-1}[a-\ep,b+\ep] \cap S_d^*(\delta) \cap \sbr{\cap_{i=d+1}^m \p_i} \cap \B_\rho^{M_1}, \quad \forall t\in[0,\frac{b-a+2\ep}{2\ep}],
	\end{align*}
	and by \eqref{eqVgeepbeta<0} we obtain
	\begin{align*}
		\|V(\eta(t,\vec{u}))\| \geq \sqrt{2\ep}, \quad \forall t\in[0,\frac{b-a+2\ep}{2\ep}].
	\end{align*}
	Then, we have
	\begin{align*}
		E(\eta(\frac{b-a+2\ep}{2\ep},\vec{u})) & = E(\vec{u}) + \int_0^{\frac{b-a+2\ep}{2\ep}}dE(\eta(t,\vec{u})) \\
		& = E(\vec{u}) + \int_0^{\frac{b-a+2\ep}{2\ep}}\langle \nabla E|_{S_{\vec{c}}}(\eta(t,\vec{u})), \frac{\partial}{\partial t}\eta(t,\vec{u}) \rangle dt \\
		& \leq b + \ep - (b-a+2\ep) = a - \ep
	\end{align*}
	and find a contradiction with $E(\eta(\frac{b-a+2\ep}{2\ep},\vec{u})) > a - \ep$, which completes the proof.
\end{proof}

\begin{remark}
	{\rm Proposition \ref{propconmapofAbeta<0} includes the case of $d = m$, in which case (ii) reads as
	\begin{itemize}
		\item $E(\eta(\frac{b-a+2\ep}{2\ep},\vec{u})) \leq a -\ep$ or $\eta(\frac{b-a+2\ep}{2\ep},\vec{u}) \in \overline{\PR_\delta}$ if $\vec{u} \in \B_\rho$ with $E(\vec{u}) \leq b + \ep$.
	\end{itemize}}
\end{remark}

\section{Sign-changing normalized  solutions} \label{secsign-changing}

We complete the proof of Theorem \ref{thm1.1} in this section. To do it, we first prove that the minimax value $m_{\vec{c},\rho,k}^\delta$, constructed in Section \ref{seclink}, is indeed a critical value.

\begin{theorem} \label{thmbeta>0}
	Assume that \eqref{assumption} holds. Let $\rho$ be a fixed positive number satisfying \eqref{defi of rho} and further satisfying $2\rho < (\beta_{\max}^+)^{-1}\mathcal{C}_N^2$ if $N = 4$, and let $\vec{c} = (c_1,\cdots,c_m)$ such that \eqref{eqci1}, \eqref{eqci2} hold true, $c_i > 0$ and $(\mu_i^+ + \beta_{\max}^+)^2\rho^3 \leq \Lambda_1\mathcal{C}_N^{8}c_i$, $i = 1,\cdots,m$. Let $\delta \in (0,\min\{\delta_0,\hat \delta\})$ where $\delta_0, \hat \delta > 0$ are given in Lemma \ref{lem>0} and Lemma \ref{lemG(u)beta<0} respectively. Then, \eqref{eq:mainequation}  has a sign-changing  normalized  solution at the level $m_{\vec{c},\rho,k}^\delta$.
\end{theorem}

\begin{proof}
	Observe that the value $m_{\vec{c},\rho,k}^\delta$ is well defined and finite by Proposition \ref{propwelldefined}.
	First, we apply Proposition \ref{propconmapofAbeta<0} to find a Palais-Smale sequence in $\B_\rho^{M_1} \cap S^*(\delta)$ at the level $m_{\vec{c},\rho,k}^\delta$, that is, a sequence $\{\vec{u}_n\} \subset \B_\rho^{M_1} \cap S^*(\delta)$ satisfying
	\begin{align*}
		E(\vec{u}_n) \to m_{\vec{c},\rho,k}^\delta \quad \text{and} \quad V(\vec{u}_n) \to 0 \text{ strongly in } H_0^1(\Omega,\R^m), \quad \text{as } n \to \infty.
	\end{align*}
	With $a = b = m_{\vec{c},\rho,k}^\delta$ and $\ep > 0$ sufficiently small, we claim the existence of $\vec{u}_\ep \in (E|_{S_{\vec{c}}})^{-1}[b-\ep,b+\ep] \cap S^*(\delta) \cap \B_\rho^{M_1}$ such that $\|V(\vec{u}_\ep)\| < \sqrt{2\ep}$. By contradiction, \eqref{eqVgeepbeta<0} holds true, and by Proposition \ref{propconmapofAbeta<0}, the flow $\eta(t,\vec{u})$, defined in \eqref{eqflowbeta<0}, globally exists for all $t \geq 0$ if $\vec{u} \in \B_\rho^{M_1}$ and the following hold:
	\begin{itemize}
		\item[(i)] $\eta(t,\vec{u}) = \vec{u}$ for all $t > 0$ if $\vec{u} \in U_1$;
		\item[(ii)] $E(\eta(1,\vec{u})) \leq b -\ep$ or $\eta(1,\vec{u}) \in \overline{\PR_\delta}$ if $\vec{u} \in \B_\rho$ with $E(\vec{u}) \leq b + \ep$.
		\item[(iii)] $\eta(t,\vec{u}) \in \overline{\PR_\delta} \cap \B_\rho^{M_1}$ for any $t \ge 0$, $\vec{u} \in \overline{\PR_\delta} \cap \B_\rho^{M_1}$.
	\end{itemize}
	According to the definition of $m_{\vec{c},\rho,k}^\delta$, we can take $A \in \mathcal{L}_{\rho,k}$ such that $\sup_{\vec{u} \in A\cap S^*(\delta)}E(\vec{u}) \le m_{\vec{c},\rho,k}^\delta + \ep$. Note that $A \subset \B_\rho^{M_1}$ by the definition of $\mathcal{L}_{\rho,k}$. Then, either $E(\vec{u}) \leq m_{\vec{c},\rho,k}^\delta - \ep$ or $\vec{u} \in \overline{\PR_\delta}$ for all $\vec{u} \in \eta(1,A)$. We will prove $\eta(1,A) \in \mathcal{L}_{\rho,k}$, then
	\[\sup_{\vec{u}\in\eta(1,A)\cap S^*(\delta)}E(\vec{u}) \le m_{\vec{c},\rho,k}^{\delta}-\varepsilon < m_{\vec{c},\rho,k}^{\delta},\]
	which contradicts the definition of $m_{\vec{c},\rho,k}^\delta$ and completes the proof of our claim. Indeed, $\eta(1,A) \subset \B_\rho^{M_1}$, and it needs to show that $\eta(1,A)$ is linked to $\mathbb{S}_{k}^\bot$. By Proposition \ref{propwelldefined}, we have
	\begin{align*}
		m_{\vec{c},\rho,k}^\delta > \sup_{\partial \mathbb{M}_{k+1}}E(\vec{u}).
	\end{align*}
	With $\ep > 0$ small enough we get $m_{\vec{c},\rho,k}^\delta - 2\ep > \sup_{\partial \mathbb{M}_{k+1}}E(\vec{u}),$ that is, $ \partial \mathbb{M}_{k+1} \subset U_1$. By property (i), we conclude that $\eta(1,\cdot)|_{\partial \mathbb{M}_{k+1}} = \textbf{id}$ and $\partial \mathbb{M}_{k+1} \subset \eta(1,A)$. Moreover, for any $h \in C(\eta(1,A),S_{\vec{c}})$, $h(\eta(1,\cdot)) \in C(A,S_{\vec{c}})$ and so $h(\eta(1,A)) \cap \mathbb{S}_k^{\bot} \neq \emptyset$.
	
	Now, with $\ep_n \to 0^+$ we obtain a sequence $\{\vec u_n\} = \{\vec u_{\ep_n}\} \subset S^*(\delta) \cap \B_\rho^{M_1}$ such that
	\begin{align*}
		m_{\vec{c},\rho,k}^\delta - \ep_n \le E(\vec u_n) \le m_{\vec{c},\rho,k}^\delta + \ep_n \quad \text{and} \quad \|V(\vec u_n)\| < \sqrt{2\ep_n},
	\end{align*}
	succeeding to find a Palais-Smale sequence of $E$ constrained to $S_{\vec c}$ in $\B_\rho^{M_1} \cap S^*(\delta)$ at the level $m_{\vec{c},\rho,k}^\delta$.
	
	By Lemma \ref{lemma 2.1} and \eqref{defi of rho} we have
	\begin{align*}
		M_1 < \frac12 \rho < \frac{1}{4}( \mu_{\max}^++\beta_{\max}^+)^{-1}\mathcal{C}_N^4.
	\end{align*}
	Then, using Lemma \ref{lempsconbeta>0}, up to a subsequence, there exists $\vec{u} \in S_{\vec{c}}$ such that $\vec{u}_n \to \vec{u}$ strongly in $H_0^1(\Omega,\R^m)$ and $V(\vec{u}) = 0$. Clearly, $E(\vec{u}) = m_{\vec{c},\rho,k}^\delta$. Moreover, by $\vec{u} \in \overline{S^*(\delta)}$ it follows that $\vec{u}$ is sign-changing. Finally, by Lagrange multiplier principle, $\vec{u}$ is a solution of \eqref{eq:mainequation} with $\la_i, i =1,\cdots,m,$ being Lagrange multipliers.
\end{proof}



Now we can complete the proof of Theorem \ref{thm1.1}.

\begin{proof}[Proof of Theorem \ref{thm1.1}]
	Let $j$ be a positive integer. Since $\Lambda_{k} \to \infty$ as $k \to \infty$, we can take $0 < \Lambda_1 = \Lambda_{k_1} < \Lambda_{k_2} < \Lambda_{k_3} < \cdots \Lambda_{k_{j}}$ such that $\Lambda_{k_l} < \Lambda_{k_l + 1}$ for $l = 1,2,\cdots,j$.
	\medbreak
	By Remark \ref{rmktildeck}, there exists $\tilde c_j > 0$ such that for $\vec{c} = (c_1,\cdots,c_m)$ with $c_i > 0$, $i = 1,\cdots,m$, and $\frac{\max_{1 \le i \le m} c_i^2}{\min_{1 \le i \le m} c_i} < \tilde c_j$, we can always take $\rho > 0$ such that \eqref{defi of rho}, \eqref{eqci1} and \eqref{eqci2} hold true for $k = k_1,\cdots,k_j$. Reasoning as Remark \ref{rmktildeck}, we can further assume that $2\rho < (\beta_{\max}^+)^{-1}\mathcal{C}_N^2$ if $N = 4$ and $(\mu_i^+ + \beta_{\max}^+)^2\rho^3 \leq \Lambda_1\mathcal{C}_N^{8}c_i$, $i = 1,2,\cdots,m$, when $N \leq 4$. Then, by Theorem \ref{thmbeta>0}, for sufficiently small $\delta > 0$, \eqref{eq:mainequation}  has sign-changing normalized   solutions at levels $m_{\vec{c},\rho,k_1}^\delta, m_{\vec{c},\rho,k_2}^\delta, \cdots, m_{\vec{c},\rho,k_j}^\delta$. Next, we will prove that $m_{\vec{c},\rho,k_1}^\delta < m_{\vec{c},\rho,k_2}^\delta < \cdots < m_{\vec{c},\rho,k_j}^\delta$, which enables us to arrive at the desired conclusion. On the one hand, by \eqref{eqvaluesepe} we have
	\begin{align} \label{eqm>}
		m_{\vec{c},\rho,k_l}^\delta \geq \inf_{\B_\rho   ^{M_1}\cap\mathbb{S}_{k_l}^\bot}E(\vec{u}) > \sup_{\partial \mathbb{M}_{k_l+1}}E(\vec{u}) \geq \sup_{\mathbb{S}_{k_l}}E(\vec{u}).
	\end{align}
	On the other hand, by Lemma \ref{lemnemp} and the definition of $m_{\vec{c},\rho,k_l}^\delta$ we get
	\begin{align} \label{eqm<}
		m_{\vec{c},\rho,k_l}^\delta \leq \sup_{\mathbb{M}_{k_l+1}}E(\vec{u}) = \sup_{\mathbb{S}_{k_l+1}}E(\vec{u}).
	\end{align}
	Combining \eqref{eqm>} and \eqref{eqm<} we conclude that
	\begin{align*}
		& \ m_{\vec{c},\rho,k_1}^\delta \leq \sup_{\mathbb{S}_{k_1+1}}E(\vec{u}) \leq \sup_{\mathbb{S}_{k_2}}E(\vec{u}) \\
		< \ & m_{\vec{c},\rho,k_2}^\delta \leq \sup_{\mathbb{S}_{k_2+1}}E(\vec{u}) \leq \sup_{\mathbb{S}_{k_3}}E(\vec{u}) \\
		< \ & \cdots \\
		< \ &  m_{\vec{c},\rho,k_{j-1}}^\delta \leq \sup_{\mathbb{S}_{k_{j-1}+1}}E(\vec{u}) \leq \sup_{\mathbb{S}_{k_j}}E(\vec{u}) \\
		< \ & m_{\vec{c},\rho,k_j}^\delta.
	\end{align*}
	The proof is complete.
\end{proof}

\section{Semi-nodal  normalized solutions} \label{secsemi-nodal}

In this section, we prove Theorem \ref{thm1.2}. For convenience, we introduce the following terminology.  For a given integer  $d$ with $1 \le d \le m-1$, a solution $\vec{u}$ is called $(d,m-d)$-semi-nodal if the components $u_1,\cdots,u_d$ change sign and the components $u_{d+1},\cdots,u_m$ are positive.  
Following the strategies in Section \ref{secsign-changing}, we first prove that the minimax value $m_{\vec{c},\rho,k,d}^\delta$, constructed in Section \ref{seclink}, is indeed a critical value.

\begin{theorem} \label{thmbeta>0seminodal}
	Assume that \eqref{assumption} holds. Let $\rho$ be a fixed positive number satisfying \eqref{defi of rho} and further satisfying $2\rho < (\beta_{\max}^+)^{-1}\mathcal{C}_N^2$ if $N = 4$, and let $\vec{c} = (c_1,\cdots,c_m)$ such that \eqref{eqci5.1}, \eqref{eqci5.2} hold true, $c_i > 0$ and $(\mu_i^+ + \beta_{\max}^+)^2\rho^3 \leq \Lambda_1\mathcal{C}_N^{8}c_i$, $i = 1,\cdots,m$. Let $\delta \in (0,\min\{\delta_{0,d},\hat \delta\})$ where $\delta_{0,d}, \hat \delta > 0$ are given in Lemma \ref{lem>0 semi-nodal} and Lemma \ref{lemG(u)beta<0} respectively. Then, \eqref{eq:mainequation}  has a $(d,m-d)$-semi-nodal  normalized solution at the level $m_{\vec{c},\rho,k,d}^\delta$.
\end{theorem}

\begin{proof}
	The value $m_{\vec{c},\rho,k,d}^\delta$ is well defined and finite by Proposition \ref{propwelldefined semi-nodal}.
	We apply Proposition \ref{propconmapofAbeta<0} to find a Palais-Smale sequence $\{\vec{u}_{n}\}$ in $\sbr{\cap_{i=d+1}^m \p_i}\cap\B_\rho^{M_1} \cap S_{d}^*(\delta)$ at the level $m_{\vec{c},\rho,k,d}^\delta$, that is
	\begin{align*}
		E(\vec{u}_n) \to m_{\vec{c},\rho,k,d}^\delta \quad \text{and} \quad V(\vec{u}_n) \to 0 \text{ strongly in } H_0^1(\Omega,\R^m), \quad \text{as } n \to \infty.
	\end{align*}
	With $a = b = m_{\vec{c},\rho,k,d}^\delta$ and $\ep > 0$ sufficiently small, we claim the existence of $\vec{u}_\ep \in (E|_{S_{\vec{c}}})^{-1}[b-\ep,b+\ep] \cap S_{d}^*(\delta) \cap \sbr{\cap_{i = d+1}^m \p_{i}} \cap \B_\rho^{M_1}$ such that $\|V(\vec{u}_\ep)\| < \sqrt{2\ep}$. By contradiction, \eqref{eqVgeepbeta<0} holds true, and by Proposition \ref{propconmapofAbeta<0}, the flow $\eta(t,\vec{u})$, defined in \eqref{eqflowbeta<0}, globally exists for all $t \geq 0$ if $\vec{u} \in \B_\rho^{M_1}$ and the following hold:
	\begin{itemize}
		\item[(i)] $\eta(t,\vec{u}) = \vec{u}$ for all $t > 0$ if $\vec{u} \in U_1$;
		\item[(ii)] $E(\eta(1,\vec{u})) \leq b -\ep$ or $\eta(1,\vec{u}) \in\overline{\PR_\delta^{(d)}}$ if $\vec{u} \in \sbr{\cap_{i=d+1}^m \p_i} \cap \B_\rho$ with $E(\vec{u}) \leq b + \ep$.
		\item[(iii)] $\eta(t,\vec{u}) \in \overline{\PR_\delta^{(d)}} \cap \B_\rho^{M_1}$ for any $t \ge 0$, $\vec{u} \in \overline{\PR_\delta^{(d)}} \cap \B_\rho^{M_1}$.
		\item[(iv)] $\eta(t,\vec{u}) \in \sbr{\cap_{i=d+1}^m \p_i} \cap \B_\rho^{M_1}$ for any $t \ge 0$, $\vec{u} \in \sbr{\cap_{i=d+1}^m \p_i} \cap \B_\rho^{M_1}$.
	\end{itemize}
	According to the definition of $m_{\vec{c},\rho,k,d}^\delta$, we can take $A \in \mathcal{L}_{\rho,k,d}$ such that $\sup_{\vec{u} \in A\cap S_{d}^*(\delta)}E(\vec{u}) \le m_{\vec{c},\rho,k,d}^\delta + \ep$. Note that $A \subset \sbr{\cap_{i=d+1}^m \p_i} \cap \B_\rho^{M_1}$ by the definition of $\mathcal{L}_{\rho,k,d}$. Then, either $E(\vec{u}) \leq m_{\vec{c},\rho,k,d}^\delta - \ep$ or $\vec{u} \in \overline{\PR_\delta^{(d)}}$ for all $\vec{u} \in \eta(1,A)$. We prove $\eta(1,A) \in \mathcal{L}_{\rho,k,d}$, which contradicts the definition of $m_{\vec{c},\rho,k,d}^\delta$ and completes the proof of our claim. Indeed, $\eta(1,A) \subset \sbr{\cap_{i=d+1}^m \p_i} \cap \B_\rho^{M_1}$ by (iv), and it needs to show that $\eta(1,A)$ is partially $d$-linked to $\mathbb{S}_{k,d}^\bot$. By Proposition  \ref{propwelldefined semi-nodal}, we have
	\begin{align*}
		m_{\vec{c},\rho,k,d}^\delta > \sup_{\partial \mathbb{M}_{k+1,d}\times\lbr{(\sqrt{c_{d+1}\phi_{1}},\cdots,\sqrt{c_{m}}\phi_{1})}}E(\vec{u}).
	\end{align*}
	With $\ep > 0$ small enough we get
	\[m_{\vec{c},\rho,k,d}^\delta - 2\ep > \sup_{\partial \mathbb{M}_{k+1,d}\times\lbr{(\sqrt{c_{d+1}}\phi_{1},\cdots,\sqrt{c_{m}}\phi_{1})}}E(\vec{u}),\]
	that is, $\partial \mathbb{M}_{k+1,d}\times\lbr{(\sqrt{c_{d+1}}\phi_{1},\cdots,\sqrt{c_{m}}\phi_{1})}\subset U_1$. By property (i), we conclude that
	\[\eta(1,\cdot)|_{\partial \mathbb{M}_{k+1,d}\times\lbr{(\sqrt{c_{d+1}}\phi_{1},\cdots,\sqrt{c_{m}}\phi_{1})}} = \textbf{id}\quad\mbox{and}\quad\partial \mathbb{M}_{k+1,d}\times\lbr{(\sqrt{c_{d+1}}\phi_{1},\cdots,\sqrt{c_{m}}\phi_{1})}\subset\eta(1,A).\]
	Moreover, for any $h \in C(\eta(1,A),S_{\vec{c}})$, $h(\eta(1,\cdot)) \in C(A,S_{\vec{c}})$ and so $T_{d}\circ h(\eta(1,A)) \cap \mathbb{S}_{k,d}^{\bot} \neq \emptyset$.
	
	Now, with $\ep_n \to 0^+$ we obtain a sequence $\{\vec u_n\} = \{\vec u_{\ep_n}\} \subset S_{d}^*(\delta)\cap\sbr{\cap_{i = d+1}^m \p_{i}} \cap \B_\rho^{M_1}$ such that
	\begin{align*}
		m_{\vec{c},\rho,k,d}^\delta - \ep_n \le E(\vec u_n) \le m_{\vec{c},\rho,k,d}^\delta + \ep_n \quad \text{and} \quad \|V(\vec u_n)\| < \sqrt{2\ep_n},
	\end{align*}
	succeeding to find a Palais-Smale sequence of $E$ constrained to $S_{\vec c}$ in $\sbr{\cap_{i=d+1}^m \p_{i}}\cap\B_\rho^{M_1} \cap S_{d}^*(\delta)$ at the level $m_{\vec{c},\rho,k,d}^\delta$.
	
	By Lemma \ref{lemma 2.1} and \eqref{defi of rho} we have
	\begin{align*}
		M_1 < \frac12 \rho < \frac{1}{4}( \mu_{\max}^+ + \beta_{\max}^+)^{-1}\mathcal{C}_N^4.
	\end{align*}
	Then, using Lemma \ref{lempsconbeta>0}, up to a subsequence, there exists $\vec{u} \in S_{\vec{c}}$ such that $\vec{u}_n \to \vec{u}$ strongly in $H_0^1(\Omega,\R^m)$ and $V(\vec{u}) = 0$. Clearly, $E(\vec{u}) = m_{\vec{c},\rho,k,d}^\delta$. By $\vec{u} \in \overline{S_{d}^*(\delta)}$ it follows that $u_{1},\cdots,u_{d}$ are sign-changing. Moreover, since $\vec{u}\in\sbr{\cap_{i=d+1}^m \p_{i}}$, we get that $u_{d+1},\cdots,u_{m}$ are nonnegative. By Lagrange multiplier principle, $\vec{u}$ is a solution of \eqref{eq:mainequation}  with $\la_i, i =1,\cdots,m,$ being Lagrange multipliers. Then the strong maximum principle implies $u_{d+1},\cdots,u_{m}$ are positive in $\Omega$. Now, we obtain $(d,m-d)$-semi-nodal  normalized solutions and complete the proof.
\end{proof}

Now we can complete the proof of Theorem \ref{thm1.2}.

\begin{proof}[Proof of Theorem \ref{thm1.2}]
	Let $j$ be a positive integer and take $0 < \Lambda_1 = \Lambda_{k_1} < \Lambda_{k_2} < \Lambda_{k_3} < \cdots \Lambda_{k_{j}}$ such that $\Lambda_{k_l} < \Lambda_{k_l + 1}$ for $l = 1,2,\cdots,j$.
	\medbreak
	By Remark \ref{rmktildeck semi-nodal}, there exists $\tilde c_{j,d} > 0$ such that for $\vec{c} = (c_1,\cdots,c_m)$ with $c_i > 0$, $i = 1,\cdots,m$, and $\frac{\max_{1 \le i \le m} c_i^2}{\min_{1\le i \le m} c_i} < \tilde c_{j,d}$, we can always take $\rho > 0$ such that \eqref{defi of rho}, \eqref{eqci5.1} and \eqref{eqci5.2} hold true for $k = k_1,\cdots,k_j$. We can further assume that $2\rho < (\beta_{\max}^+)^{-1}\mathcal{C}_N^2$ if $N = 4$ and $(\mu_i^+ + \beta_{\max}^+)^2\rho^3 \leq \Lambda_1\mathcal{C}_N^{8}c_i$, $i = 1,2,\cdots,m$, when $N \leq 4$. Then, by Theorem \ref{thmbeta>0seminodal}, for sufficiently small $\delta > 0$, \eqref{eq:mainequation}  has $(d,m-d)$-semi-nodal  normalized solutions at levels $m_{\vec{c},\rho,k_1,d}^\delta, m_{\vec{c},\rho,k_2,d}^\delta, \cdots, m_{\vec{c},\rho,k_j,d}^\delta$. Next, we will prove that $m_{\vec{c},\rho,k_1,d}^\delta < m_{\vec{c},\rho,k_2,d}^\delta < \cdots < m_{\vec{c},\rho,k_j,d}^\delta$, which enables us to arrive at the desired conclusion. By \eqref{eqestiofm semi-nodal}, we have
	\begin{align} \label{eqm>seminodal}
		m_{\vec{c},\rho,k_l,d}^\delta > \sup_{\partial \mathbb{M}_{k_{l}+1,d} \times \lbr{(\sqrt{c_{d+1}}\phi_1, \cdots, \sqrt{c_{m}}\phi_1)}}E(\vec{u}) \geq \sup_{\mathbb{S}_{k_{l},d}\times \lbr{(\sqrt{c_{d+1}}\phi_1, \cdots, \sqrt{c_{m}}\phi_1)}}E(\vec{u}).
	\end{align}
	and
	\begin{align} \label{eqm<seminodal}
		m_{\vec{c},\rho,k_l,d}^\delta \leq \sup_{\mathbb{M}_{k_{l}+1,d} \times \lbr{(\sqrt{c_{d+1}}\phi_1, \cdots, \sqrt{c_{m}}\phi_1)}}E(\vec{u}) = \sup_{\mathbb{S}_{k_l+1,d}\times \lbr{(\sqrt{c_{d+1}}\phi_1, \cdots, \sqrt{c_{m}}\phi_1)}}E(\vec{u}).
	\end{align}
	Combining \eqref{eqm>seminodal} and \eqref{eqm<seminodal} we conclude that
	\begin{align*}
		& \ m_{\vec{c},\rho,k_1,d}^\delta \leq \sup_{\mathbb{S}_{k_1+1,d}\times \lbr{(\sqrt{c_{d+1}}\phi_1, \cdots, \sqrt{c_{m}}\phi_1)}}E(\vec{u}) \leq \sup_{\mathbb{S}_{k_2,d}\times \lbr{(\sqrt{c_{d+1}}\phi_1, \cdots, \sqrt{c_{m}}\phi_1)}}E(\vec{u}) \\
		< \ & m_{\vec{c},\rho,k_2,d}^\delta \leq \sup_{\mathbb{S}_{k_2+1,d}\times \lbr{(\sqrt{c_{d+1}}\phi_1, \cdots, \sqrt{c_{m}}\phi_1)}}E(\vec{u}) \leq \sup_{\mathbb{S}_{k_3,d}\times \lbr{(\sqrt{c_{d+1}}\phi_1, \cdots, \sqrt{c_{m}}\phi_1)}}E(\vec{u}) \\
		< \ & \cdots \\
		< \ &  m_{\vec{c},\rho,k_{j-1},d}^\delta \leq \sup_{\mathbb{S}_{k_{j-1}+1,d}\times \lbr{(\sqrt{c_{d+1}}\phi_1, \cdots, \sqrt{c_{m}}\phi_1)}}E(\vec{u}) \leq \sup_{\mathbb{S}_{k_j,d}\times \lbr{(\sqrt{c_{d+1}}\phi_1, \cdots, \sqrt{c_{m}}\phi_1)}}E(\vec{u}) \\
		< \ & m_{\vec{c},\rho,k_j,d}^\delta.
	\end{align*}
	The proof is complete.
\end{proof}

\section{Limit behaviors as $\vec{c}\to\vec{0}^+$ and the bifurcation phenomenon}\label{limit}

\subsection{Proof of Theorem \ref{limitbehaviorhsignchanging}}

In this subsection we are devoted to prove Theorem \ref{limitbehaviorhsignchanging}. Let $(\lambda_{1},\cdots,\lambda_{m})=(\Lambda_{k_{1}},\cdots,\Lambda_{k_{m}})$ where $k_{1},\cdots,k_{m}$ are positive integers. We firstly show that $(\lambda_{1},\cdots,\lambda_{m})=(\Lambda_{k_{1}},\cdots,\Lambda_{k_{m}})$ is a nontrivial bifurcation point of system \eqref{eq:mainequation} when $k_i \ge 2$ for each $i \in \{1,\cdots,m\}$.

\begin{proposition} \label{propbifurcationsign-changing}
	Assume that \eqref{assumption} holds. Let $(\lambda_{1},\cdots,\lambda_{m})=(\Lambda_{k_{1}},\cdots,\Lambda_{k_{m}})$ where $k_{1},\cdots,k_{m}$ are positive integers no less than $2$. Then $(\lambda_{1},\cdots,\lambda_{m})$ is a nontrivial bifurcation point of system \eqref{eq:mainequation}.
\end{proposition}

\begin{proof}
	We just provide the proof with $k_{1}=\cdots=k_{m}\geq 2$ by using Theorem \ref{thmbeta>0}. Modifying the definitions of $\mathcal{H}_{k}$ and $\mathcal{H}_{k}^{\bot}$ in \eqref{defi of Hk} as
	\[\mathcal{H}_{\vec{k}}:=H_{k_{1}}\times\cdots\times H_{k_{m}}\quad\mbox{and}\quad\mathcal{H}_{\vec{k}}^{\bot}=H_{k_{1}}^{\bot}\times\cdots\times H_{k_{m}}^{\bot},\]
	and adapting the proof strategy when $k_{1}=\cdots=k_{m}\geq 2$, we can complete proof of the general case with $k_{1},\cdots,k_{m} \ge 2$.
	
	Let $k = k_{1}=\cdots=k_{m}\geq 2$. Without loss of generality, we assume that $\Lambda_{k-1} < \Lambda_k$. In this proof we always suppose for $\vec{c} = (c_1,\cdots,c_m) \in (0,\infty)^m$ that
	\begin{equation} \label{eqci2leq}
		\sbr{\max_{1 \le i \le m} c_{i}^2}\sbr{\min_{1 \le i \le m} c_{i}}^{-1} <\hat{c}_{k},
	\end{equation}
	so that we can take $\rho > 0$ with assumptions in Theorem \ref{thmbeta>0} satisfied for $k-1$. Then, there exists a sign-changing solution $\vec{u}_{\vec{c}}$ with Lagrange multipliers $\vec{\lambda}_{\vec{c}}$ at level $m_{\vec{c},\rho,k-1}^{\delta}$ for sufficiently small $\delta > 0$. For $i=1,\cdots,m$, let $v_{\vec{c},i}:=\frac{1}{\sqrt{c_{i}}}u_{\vec{c},i}$. In the following we further assume that there exists a constant $\overline{c}>1$ such that
	\begin{equation}\label{framaxcimincileq}
		\sbr{\max_{1 \le i \le m} c_{i}}\sbr{\min_{1 \le i \le m} c_{i}}^{-1}\leq\overline{c}.
	\end{equation}
	
	\vskip0.1in
	\emph{Step 1:} $\{\vec{v}_{\vec{c}}\}$ and $\{\vec{\lambda}_{\vec{c}}\}$ are bounded in $H_{0}^1(\Omega,\mathbb{R}^m)$ and $\mathbb{R}^m$ respectively as $\vec{c}\to\vec{0}^+$.
	
	By \eqref{e7}, \eqref{e9} and Proposition \ref{propwelldefined}, we know
	\begin{equation*}
		E(\vec{u}_{\vec{c}})=m_{\vec{c},\rho,k-1}^{\delta}\leq\sup_{\mathbb{S}_{k}}E\leq\frac{1}{4}\mbr{2- \sbr{  \mu_{\min}^- + \beta_{\min}^-} \mathcal{C}_N^{-4} \rho }\sum_{i= 1}^m  {c_i }\Lambda_{k}
	\end{equation*}
	and
	\begin{equation*}
		E(\vec{u}_{\vec{c}})\geq\inf_{\B_\rho   ^{M_1}\cap\mathbb{S}_{k-1}^\bot}E\geq
		\frac 14 \sum_{i=1}^m \mbr{  \sbr{2- \beta_{\max}^+\mathcal{C}_N^{-4}\rho}c_i\Lambda_{k}  -\mu_{\max}^+\mathcal{C}_N^{-4}  \sbr{c_i\Lambda_{k} }^2}.
	\end{equation*}
	By the choice of $\rho$ (see Remark \ref{rmktildeck}), we know $\rho\to 0$ as $\vec{c}\to\vec{0}^+$. This implies
	\begin{equation}\label{Evecuveccl}
		E(\vec{u}_{\vec{c}})=\Big(\frac{1}{2}+o_{\vec{c}}(1)\Big)\sum_{i=1}^mc_{i}\Lambda_{k}\quad\mbox{as}\ \vec{c}\to\vec{0}^+.
	\end{equation}
	Using the Sobolev inequality, we have
	\begin{equation*}
		E(\vec{u}_{\vec{c}})\leq\frac{1}{4}\mbr{2- \sbr{  \mu_{\min}^- + \beta_{\min}^-} \mathcal{C}_N^{-4} \rho }\sum_{i=1}^m\int_\Omega |\nabla u_{\vec{c},i}|^2dx
	\end{equation*}
	and
	\begin{equation}
		E(\vec{u}_{\vec{c}})\geq\frac14 \sbr{2-\beta_{\max}^+\mathcal{C}_{N}^{-4}\rho}\sum_{i=1}^m\int_\Omega |\nabla u_{\vec{c},i}|^2dx -  \frac14 \mu_{\max}^+\mathcal{C}_N^{-4}\sum_{i=1}^m\sbr{\int_\Omega |\nabla u_{\vec{c},i}|^2dx}^2.
	\end{equation}
	Since
	\[\sum_{i=1}^m\int_{\Omega}\lvert\nabla u_{\vec{c},i}\rvert^2dx<\rho,\]
	we obtain
	\[E(\vec{u}_{\vec{c}})=\Big(\frac{1}{2}+o_{\vec{c}}(1)\Big)\sum_{i=1}^m\int_{\Omega}\lvert\nabla u_{\vec{c},i}\rvert^2dx\quad\mbox{as}\ \vec{c}\to\vec{0}^+,\]
	which combining \eqref{Evecuveccl} implies
	\begin{equation}\label{limvecctovec0}
		\lim_{\vec{c}\to\vec{0}^+}\frac{\sum_{i=1}^m\int_{\Omega}\lvert\nabla u_{\vec{c},i}\rvert^2dx}{\sum_{i=1}^mc_{i}}=\Lambda_{k}.
	\end{equation}
	Therefore, by \eqref{framaxcimincileq}, we know
	\[\sum_{i=1}^m\int_{\Omega}\lvert\nabla v_{\vec{c},i}\rvert^2dx=\sum_{i=1}^m\frac{\int_{\Omega}\lvert\nabla u_{\vec{c},i}\rvert^2dx}{c_{i}}\]
	is bounded, that is $\{\vec{v}_{\vec{c}}\}$ is bounded in $H_{0}^1(\Omega,\mathbb{R}^m)$.
	
	For $i=1,\cdots,m$, direct calculations show that $\vec{v}_{\vec{c}}$ satisfies
	\begin{equation}\label{vecvveccji}
		-\Delta v_{\vec{c},i}+\lambda_{\vec{c},i}v_{\vec{c},i}=\mu_{i}c_{i}v_{\vec{c},i}^3+\sum_{j \neq i}\beta_{ij}c_{j}v_{\vec{c},j}^2 v_{\vec{c},i}\quad \mbox{in}\ \Omega,
	\end{equation}
	which implies $\{\lambda_{\vec{c},i}\}$ is bounded, since $\lVert v_{\vec{c},i}\rVert_{L^2(\Omega)}=1$.

	\vskip0.1in
	\emph{Step 2:}
	For $i=1,\cdots,m$, $v_{\vec{c},i}$ converges strongly to some eigenfunction of $-\Delta$ on $H_{0}^1(\Omega)$ and $-\lambda_{\vec{c},i}$ converges to the corresponding eigenvalue up to a subsequence as $\vec{c}\to\vec{0}^+$.
	
	By Step 1, up to a subsequence we assume that
	\begin{align}
		\begin{aligned}
			&v_{\vec{c},i}\rightharpoonup v_{i}\quad\mbox{in}\ H_{0}^1(\Omega),\\
			&v_{\vec{c},i}\to v_{i}\quad\mbox{in}\ L^2(\Omega),\\
			&v_{\vec{c},i}\to v_{i}\quad\mbox{almost everywhere in}\ \Omega,\\
			&-\lambda_{\vec{c},i} \to \lambda_{i}\quad\mbox{in} \ \mathbb{R},
		\end{aligned}
	\end{align}
	as $\vec{c}\to\vec{0}^+$. Since $\lVert v_{i}\rVert_{L^2(\Omega)}=1$, we know $v_{i}\neq 0$. Obviously,
	\[c_{i}\int_{\Omega}v_{\vec{c},i}^4dx\to 0\ \ \mbox{and}\ \ \sum_{j \neq i}c_{j}\int_{\Omega}v_{\vec{c},i}^2v_{\vec{c},j}^2dx\to 0\quad\mbox{as}\ \vec{c}\to\vec{0}^+.\]
	Thus, from \eqref{vecvveccji}, we know $v_{\vec{c},i}\to v_{i}$ in $H_{0}^1(\Omega)$ and $v_{i}$ satisfies
	\[-\Delta v_{i} = \lambda_{i}v_{i}\quad\mbox{in}\ \Omega.\]
	This shows that $\lambda_{i}$ is an eigenvalue of $-\Delta$ on $H_{0}^1(\Omega)$ and $v_{i}$ is the corresponding eigenfunction.
	
	\vskip0.1in
	\emph{Step 3:} Completion.
	
	Observe that $\lambda_{i}$ in Step 2 depends on the sequence $\{\vec{c}\}_{\vec{c} \to\vec{0}^+}$. To complete the proof, it suffices to show that there exists $\{\vec{c}_n\}$ such that \eqref{eqci2leq} and \eqref{framaxcimincileq} hold, $\vec{c}_n \to\vec{0}^+$, and $\la_1 = \cdots = \la_m = \Lambda_k$. Arguing by contradiction, we suppose that such sequence $\{\vec{c}_n\}$ does not exist. Let
	\begin{equation} \label{def: set T}
		T := \lbr{(t_1,\cdots,t_m) \in \R^m: t_1, \cdots, t_m > 0, ~\sum_{i =1}^mt_i = 1,~ \sbr{\max_{1 \le i \le m} t_{i}}\sbr{\min_{1 \le i \le m} t_{i}}^{-1}\leq\overline{c}}.
	\end{equation}
	For any $\vec{t} \in T$ and $0 < r < \hat{c}_k/\overline{c}$, we have that $\vec{c} = r\vec{t}$ satisfies \eqref{eqci2leq} and \eqref{framaxcimincileq}. Now we fix $\vec{t} = (t_1,\cdots,t_m) \in T$ and take $\vec{c} = r\vec{t}$ with $r \to 0$ such that $\vec{c} \to \vec{0}^+$. By Step 2, we assume that $-\lambda_{\vec{c},i}\to\Lambda_{j_{i}}$ up to a subsequence as $\vec{c}\to\vec{0}^+$. By the assumption based on contradiction arguments, it is impossible that $\Lambda_{j_{1}} = \cdots = \Lambda_{j_{m}} = \Lambda_k$.
	From \eqref{limvecctovec0}, we obtain that
	\begin{align}
		\lim_{\vec{c}\to\vec{0}^+}\frac{\sum_{i=1}^m\int_{\Omega}\lvert\nabla u_{\vec{c},i}\rvert^2dx}{\sum_{i=1}^mc_{i}}&=\lim_{\vec{c}\to\vec{0}^+}\sum_{i=1}^mt_{i}\frac{\int_{\Omega}\lvert\nabla u_{\vec{c},i}\rvert^2dx}{c_{i}}\\
		&=\lim_{\vec{c}\to\vec{0}^+}\sum_{i=1}^mt_{i}\int_{\Omega}\lvert\nabla v_{\vec{c},i}\rvert^2dx\\
		&=\sum_{i=1}^mt_{i}\Lambda_{j_{i}}=\Lambda_{k},
	\end{align}
	that is
	\[\sum_{i=1}^mt_{i}(\Lambda_{k}-\Lambda_{j_{i}})=0.\]
	This implies $(t_{1},\cdots,t_{m})$ is orthogonal to $(\Lambda_{k}-\Lambda_{j_{1}},\cdots,\Lambda_{k}-\Lambda_{j_{m}})$. Let $A$ denote the set consisting all $(\Lambda_{k}-\Lambda_{j_{1}},\cdots,\Lambda_{k}-\Lambda_{j_{m}})$ when $\vec{t}$ varies in $T$. We have assumed that $\vec{0}\notin A$. Thus for each $\vec{a}\in A$, all vectors that orthogonal to $\vec{a}$ form an $(m-1)$-dimensional hyperplane passing through the origin. It can be shown that the intersection of this hyperplane and $T$ is of $(m-2)$-dimension. Denote this intersection by $T_{\vec{a}}$, then dim$T_{\vec{a}}=m-2$ and hence the $(m-1)$-dimensional measure of $T_{\vec{a}}$ is zero. Since $\{\Lambda_{k}\}_{k\geq 1}$ is countable, $A$ is countable and hence the $(m-1)$-dimensional measure of $\cup_{\vec{a} \in A}T_{\vec{a}}$ is zero. However, $T = \cup_{\vec{a} \in A}T_{\vec{a}}$ and $T$ has a positive $(m-1)$-dimensional measure. We have reached a contradiction and complete the proof.
	
\end{proof}

Next we consider $(\lambda_{1},\cdots,\lambda_{m})=(\Lambda_{k_{1}},\cdots,\Lambda_{k_{m}})$ when there exists $i \in \{1,\cdots,m\}$ such that $k_i = 1$. Rearranging $\lambda_{1}, \cdots, \lambda_{m}$ if necessary, we assume that $(\lambda_{1},\cdots,\lambda_{m})=(\Lambda_{k_{1}},\cdots,\Lambda_{k_{d}},\Lambda_1,\cdots,\Lambda_1)$ with $k_{i} \ge 2$ for $i \in \{1,\cdots,d\}$. The case $d = 0$ is allowed, which means that $(\lambda_{1},\cdots,\lambda_{m})=(\Lambda_1,\cdots,\Lambda_1)$.

\begin{proposition} \label{propbifurcationsemi-nodal}
	Assume that \eqref{assumption} holds. Let $(\lambda_{1},\cdots,\lambda_{m})=(\Lambda_{k_{1}},\cdots,\Lambda_{k_{d}},\Lambda_1,\cdots, \Lambda_1)$ where $d$ is an integer with $0 \le d \le m-1$ and $k_{1},\cdots,k_{d}$ are positive integers no less than $2$. Then $(\lambda_{1},\cdots,\lambda_{m})$ is a nontrivial bifurcation point of system \eqref{eq:mainequation}.
\end{proposition}

\begin{proof}
	We first assume that $1 \le d \le m-1$. Similar to the proof of Proposition \ref{propbifurcationsign-changing}, we just provide the proof with $k_{1}=\cdots=k_{d}\geq 2$ by using Theorem \ref{thmbeta>0seminodal}.
	
	Let $k = k_{1}=\cdots=k_{d}\geq 2$. Without loss of generality, we assume that $\Lambda_{k-1} < \Lambda_k$. In this proof we always suppose for $\vec{c} = (c_1,\cdots,c_m) \in (0,\infty)^m$ that
	\begin{equation} \label{eqci2leqofd}
		\sbr{\max_{1 \le i \le m} c_{i}^2}\sbr{\min_{1 \le i \le m} c_{i}}^{-1} <\hat{c}_{k,d},
	\end{equation}
	so that we can take $\rho > 0$ with assumptions in Theorem \ref{thmbeta>0seminodal} satisfied for $k-1$. Then there exists a $(d,m-d)$-semi-nodal normalized  solution $\vec{u}_{\vec{c}}$ of \eqref{eq:mainequation}  with Lagrange multipliers $\vec{\lambda}_{\vec{c}}$ at level $m_{\vec{c},\rho,k-1,d}^\delta$ for sufficiently small $\delta >0$. For $i=1,\cdots,m$, let $v_{\vec{c},i}:=\frac{1}{\sqrt{c_{i}}}u_{\vec{c},i}$. In the following we further assume that \eqref{framaxcimincileq} holds true.
	
	\vskip0.1in
	\emph{Step 1:} $\{\vec{v}_{\vec{c}}\}$ and $\{\vec{\lambda}_{\vec{c}}\}$ are bounded in $H_{0}^1(\Omega,\mathbb{R}^m)$ and $\mathbb{R}^m$ respectively as $\vec{c}\to\vec{0}^+$.
	
	By \eqref{ss1}, \eqref{Eveculeq}, \eqref{qw1} and Proposition \ref{propwelldefined semi-nodal}, we know
	\begin{align}
		E(\vec{u}_{\vec{c}}) = m_{\vec{c},\rho,k-1,d}^{\delta} & \leq \sup_{\mathbb{S}_{k,d}\times \lbr{(\sqrt{c_{d+1}}\phi_1, \cdots, \sqrt{c_{m}}\phi_1)}}E \\
		& \le \frac{1}{4}\mbr{2- \sbr{ 	\mu_{\min}^-   +  \beta_{\min}^-} \mathcal{C}_N^{-4} \rho } \left( \Lambda_{k}\sum_{i= 1}^d c_i +  \Lambda_{1} \sum_{i=d+1}^m c_i \right) ,
	\end{align}
	and
	\begin{align}
		E(\vec{u}_{\vec{c}})&\geq\inf_{  	\B_\rho \cap \sbr{\mathbb{S}_{k-1,d}^\bot \times S_{\vec{c},m-d}}}E\\
		&\geq\frac 14 \sum_{i=1}^d \mbr{  \sbr{2- \beta_{\max}^+\mathcal{C}_N^{-4}\rho}c_i\Lambda_{k}  - \mu_{\max}^+ \mathcal{C}_N^{-4}  \sbr{c_i\Lambda_{k} }^2}\\
		&\qquad+ \frac 14 \sum_{i=d+1}^m \mbr{  \sbr{2-\beta_{\max}^+\mathcal{C}_N^{-4}\rho}c_i\Lambda_{1}  - \mu_{\max}^+ \mathcal{C}_N^{-4}  \sbr{c_i\Lambda_{1} }^2}.
	\end{align}
	Similar to Step 1 in proving Proposition \ref{propbifurcationsign-changing}, we have
	\begin{equation}\label{Evecuveccjd}
		E(\vec{u}_{\vec{c}})=\Big(\frac{1}{2}+o_{\vec{c}}(1)\Big)\Big(\sum_{i=1}^dc_{i}\Lambda_{k}+\sum_{i=d+1}^mc_{i}\Lambda_{1}\Big)\quad\mbox{as}\ \vec{c}\to\vec{0}^+.
	\end{equation}
	Moreover,
	\[E(\vec{u}_{\vec{c}})=\Big(\frac{1}{2}+o_{\vec{c}}(1)\Big)\sum_{i=1}^m\int_{\Omega}\lvert\nabla u_{\vec{c},i}\rvert^2dx\quad\mbox{as}\ \vec{c}\to\vec{0}^+.\]
	Therefore,
	\begin{equation}\label{limvecctovec0+}
		\lim_{\vec{c}\to\vec{0}^+}\frac{\sum_{i=1}^m\int_{\Omega}\lvert\nabla u_{\vec{c},i}\rvert^2dx}{\sum_{i=1}^dc_{i}\Lambda_{k}+\sum_{i=d+1}^mc_{i}\Lambda_{1}}=1,
	\end{equation}
	which combining (\ref{framaxcimincileq}) implies
	\[\sum_{i=1}^m\int_{\Omega}\lvert\nabla v_{\vec{c},i}\rvert^2dx=\sum_{i=1}^m\frac{\int_{\Omega}\lvert\nabla u_{\vec{c},i}\rvert^2dx}{c_{i}}\]
	is bounded as $\vec{c}\to\vec{0}^+$, that is $\{\vec{v}_{\vec{c}}\}$ is bounded in $H_{0}^1(\Omega,\mathbb{R}^m)$.
	
	For $i=1,\cdots,m$, direct calculations show that $\vec{v}_{\vec{c}}$ satisfies
	\begin{equation}\label{vclid}
		-\Delta v_{\vec{c},i}+\lambda_{\vec{c},i}v_{\vec{c},i}=\mu_{i}c_{i}v_{\vec{c},i}^3+\sum_{j \neq i}\beta_{ij} c_{j}v_{\vec{c},j}^2v_{\vec{c},i}\quad \mbox{in}\ \Omega,
	\end{equation}
	which implies $\{\lambda_{\vec{c},i}\}$ is bounded, since $\lVert v_{\vec{c},i}\rVert_{L^2(\Omega)}=1$.
	
	\vskip0.1in
	\emph{Step 2:}
	Similar to Step 1 in the proof of Proposition \ref{propbifurcationsign-changing}, for $i=1,\cdots,m$, $v_{\vec{c},i}$ converges strongly to some eigenfunction of $-\Delta$ on $H_{0}^1(\Omega)$ and $-\lambda_{\vec{c},i}$ converges to corresponding eigenvalue up to a subsequence as $\vec{c}\to\vec{0}^+$. Moreover, for $i=d+1,\cdots,m$, since $v_{\vec{c},i}$ is positive in $\Omega$, we know that $v_{\vec{c},i}$ converges strongly to first eigenfunction $\phi_{1}$ in $H_{0}^1(\Omega)$ and $-\lambda_{\vec{c},i}$ converges to first eigenvalue $\Lambda_{1}$.
	
	\vskip0.1in
	\emph{Step 3:} Completion.
	
	We can prove this step after a small modification of Step 3 in the proof of Proposition \ref{propbifurcationsign-changing}. Let $\overline{c}$ satisfy 
	$$
	\frac{\Lambda_1}{\Lambda_k}\overline{c} > 1,
	$$
	and the definition of set $T$ is modified as
	\begin{equation} \label{def: set Tofd}
		T := \lbr{(t_1,\cdots,t_m) \in \R^m: t_1, \cdots, t_m > 0, ~\sum_{i =1}^mt_i = 1, ~\sbr{\max_{1 \le i \le m} t_{i}}\sbr{\min_{1 \le i \le m} t_{i}}^{-1}\leq \frac{\Lambda_1}{\Lambda_k}\overline{c}}.
	\end{equation}
	For any $\vec{t} \in T$ and $0 < r < \Lambda_1\hat{c}_{k,d}/\overline{c}$, we have that $\vec{c} = (rt_1/\Lambda_k, \cdots, rt_d/\Lambda_k, rt_{d+1}/\Lambda_1, \cdots, rt_m/\Lambda_1)$ satisfies \eqref{eqci2leqofd} and \eqref{framaxcimincileq}. Now we fix $\vec{t} = (t_1,\cdots,t_m) \in T$ and take $\vec{c} = (rt_1/\Lambda_k, \cdots, rt_d/\Lambda_k, rt_{d+1}/\Lambda_1, \cdots, rt_m/\Lambda_1)$ with $r \to 0$ such that $\vec{c} \to \vec{0}$. By Step 2, we assume that $-\lambda_{\vec{c},i}\to\Lambda_{j_{i}}$ up to a subsequence as $\vec{c}\to\vec{0}^+$, and $j_{d+1} = \cdots j_m = 1$. Then we obtain that
	\begin{align}
		\lim_{\vec{c}\to\vec{0}^+} \frac{\sum_{i=1}^m\int_{\Omega}\lvert\nabla u_{\vec{c},i}\rvert^2dx}{\sum_{i=1}^dc_{i}\Lambda_{k}+\sum_{i=d+1}^mc_{i}\Lambda_{1}} &=\lim_{\vec{c}\to\vec{0}^+}\Big(\frac{1}{\Lambda_{k}}\sum_{i=1}^dt_{i}\frac{\int_{\Omega}\lvert\nabla u_{\vec{c},i}\rvert^2dx}{c_{i}}+\frac{1}{\Lambda_{1}}\sum_{i=d+1}^mt_{i}\frac{\int_{\Omega}\lvert\nabla u_{\vec{c},i}\rvert^2dx}{c_{i}}\Big),\\
		&=\frac{1}{\Lambda_{k}}\sum_{i=1}^dt_{i}\Lambda_{j_{i}}+\sum_{i=d+1}^mt_{i}=1,
	\end{align}
	that is
	\[\sum_{i=1}^dt_{i}(\Lambda_{k}-\Lambda_{j_{i}})=0.\]
	Then, we can choose special $(t_{1},\cdots,t_{m}) \in T$ such that $\Lambda_{j_{1}}=\cdots=\Lambda_{j_{d}}=\Lambda_{k}$.
	
	
	Finally, when $d = 0$, our method proving Theorem \ref{thmbeta>0seminodal} still works, in which case the normalized solutions that we obtain are positive. By Step 2 above, we know \[(\lambda_{1},\cdots,\lambda_{m})=(\underbrace{\Lambda_{1},\cdots,\Lambda_{1}}_{m})\] is a nontrivial bifurcation point of system \eqref{eq:mainequation}. The proof is complete.
\end{proof}

Now we can complete the proof of Theorem \ref{limitbehaviorhsignchanging}.

\begin{proof}[Proof of Theorem \ref{limitbehaviorhsignchanging}]
	Let
	\[
	\mathbb{B} := \bigg\{ (\lambda_{1},\cdots,\lambda_{m}) \in \R^m: \la_i \text{ is an eigenvalue of $-\Delta$ on $H_0^1(\Omega)$ for each } i \in \big\{1,\cdots,m\big\} \bigg\}.
	\]
	On one hand, it is relatively standard to show that any nontrivial bifurcation point of system \eqref{eq:mainequation} is in $\mathbb{B}$. On the other hand, combining Proposition \ref{propbifurcationsign-changing} and Proposition \ref{propbifurcationsemi-nodal} we conclude that any point in $\mathbb{B}$ is a nontrivial bifurcation point of system \eqref{eq:mainequation}. Thus we complete the proof.
\end{proof}

\begin{remark} \label{rmk:mdimension}
	{\rm We remark that bifurcation phenomenon from the nontrivial bifurcation point of system \eqref{eq:mainequation} in $\mathbb{B}$ is $m$-dimensional. In fact, let $T$ be defined in \eqref{def: set T} or \eqref{def: set Tofd}, which is $(m-1)$-dimensional after removing $Z$, a countable union of zero $\R^{m-1}$-measure sets. Then both
	\[
	\lbr{\vec{c} = r\vec{t}: 0 < r < \frac{\hat{c}_k}{\overline{c}}, ~ \vec{t} \in T\backslash Z}
	\]
	and
	\[
	\lbr{\vec{c} = r\sbr{\frac{t_1}{\Lambda_k}, \cdots, \frac{t_d}{\Lambda_k}, \frac{t_{d+1}}{\Lambda_1}, \cdots, \frac{t_m}{\Lambda_1}}: 0 < r < \frac{\Lambda_1\hat{c}_{k,d}}{\overline{c}}, ~ \vec{t} \in T\backslash Z}
	\]
	are $m$-dimensional.}
\end{remark}

\subsection{Proof of Theorem \ref{thm:semi-trivialbifurpoint}}

At the end of this article, we prove Theorem \ref{thm:semi-trivialbifurpoint}.

\begin{proof}[Proof of Theorem \ref{thm:semi-trivialbifurpoint}]
	We merely prove that for any positive integer $1 \le d \le m-1$, the set of $d$-semi-trivial bifurcation points of system \eqref{eq:mainequation} is exactly
	\begin{align*}
		& ~~ \mathbb{B}_d := \\
		& \bigg\{ (\lambda_{1},\cdots,\lambda_{m}) \in \R^m: \exists i_1 < \cdots < i_d \in \big\{1,\cdots,m\big\}, ~ \la_{i_1}, \cdots, \la_{i_d} \text{ are eigenvalues of $-\Delta$ on } H_0^1(\Omega)\bigg\}.
	\end{align*}
	Then it follows that the set of semi-trivial bifurcation points is
	\[
	\bigg\{ (\lambda_{1},\cdots,\lambda_{m}) \in \R^m: \exists i \in \big\{1,\cdots,m\big\} \text{ such that } \la_i \text{ is an eigenvalue of $-\Delta$ on } H_0^1(\Omega)\bigg\}.
	\]
	
	On one hand, it is relatively standard to show that any $d$-semi-trivial bifurcation point of system \eqref{eq:mainequation} is in $\mathbb{B}_d$. On the other hand, let $(\lambda_{1},\cdots,\lambda_{m}) \in \mathbb{B}_d$. After rearranging $\lambda_{1},\cdots,\lambda_{m}$, we assume that $\lambda_{1},\cdots,\lambda_{d}$ are eigenvalues of $-\Delta$ on $H_0^1(\Omega)$. By Theorem \ref{limitbehaviorhsignchanging}, if $d \ge 2$ there exists $\{(\la_{n,1}, \cdots, \la_{n,d}; u_{n,1}, \cdots, u_{n,d})\} \subset \R^d \times H_0^1(\Omega,\R^d)$ such that $-\la_{n,i} \to \la_i$, $u_{n,i} \not\equiv 0$, $\|u_{n,i}\| \to 0$ for $i = 1,\cdots d$, and for each $n$, $(u_{n,1}, \cdots, u_{n,d})$ satisfies
	\begin{align*}
		-\Delta u_{n,i} + \la_{n,i} u_{n,i} =  \mu_i u_{n,i}^3+ \sum_{j=1,j\neq i}^d\beta_{ij} u_{n,j}^2 u_{n,i}  \quad \text{ in } \Omega.
	\end{align*}
	Such result also holds when $d = 1$, in a single equation case. Then considering the sequence $$\{(\la_{n,1}, \cdots, \la_{n,d}, \la_{d+1},\cdots,\la_m; u_{n,1},\cdots,u_{n,d},0,\cdots,0)\} \subset \R^m \times H_0^1(\Omega,\R^m),$$  we know $(\lambda_{1},\cdots,\lambda_{m})$ is a $d$-semi-trivial bifurcation point of system \eqref{eq:mainequation}. The proof is complete.
\end{proof}

\medskip

{\small \noindent \textbf{Acknowledgements:} This work is funded by National Key R\&D Program of China (Grant 2023YFA1010001) and NSFC (12571123). It was completed when Song was visiting the Universit\'e Marie et Louis Pasteur, LmB (UMR 6623).  Liu is supported by the "National Funded Postdoctoral Researcher Program" (GZB20240945) and "China Postdoctoral Science Foundation" (2025M784442). Song is supported by "China Postdoctoral Science Foundation" (2024T170452).
}

\medskip

{\small \noindent \textbf{Statements and Declarations:} The authors have no relevant financial or non-financial interests to disclose.}

\medskip

{\small \noindent \textbf{Data availability:} Data sharing is not applicable to this article as no datasets were generated or analysed during the current study.}

\end{document}